\documentclass[12pt]{amsart}

\usepackage{amsmath}
\usepackage{amssymb}
\usepackage{amscd}
\usepackage{xcolor}
\usepackage{hyperref}
\usepackage{mathrsfs}
\usepackage{upgreek}
\usepackage{eucal}
\usepackage{mathtools}
\usepackage{enumitem}
\usepackage{comment}

\usepackage{color}

\usepackage[normalem]{ulem}

\usepackage{float}
\restylefloat{table}

\usepackage{xy}
\xyoption{all}

\newenvironment{aenumerate}{\begin{enumerate}[label={\textup{(\alph*)}}]}{\end{enumerate}}

\newcommand{\nc}{\newcommand}

\nc{\rv}{\upupsilon}

\nc{\CC}{{\mathbb{C}}}
\nc{\LL}{{\mathbb{L}}}
\nc{\RR}{{\mathbb{R}}}
\renewcommand{\P}{{\mathbb{P}}}
\nc{\OO}{{\mathbb{O}}}

\nc{\QQ}{{\mathbb{Q}}}
\nc{\ZZ}{{\mathbb{Z}}}

\nc{\cA}{{\mathcal{A}}}
\nc{\cB}{{\mathcal{B}}}
\nc{\cC}{{\mathcal{C}}}
\nc{\cD}{{\mathcal{D}}}
\nc{\cE}{{\mathcal{E}}}
\nc{\tcE}{{\tilde{\mathcal{E}}}}
\nc{\cF}{{\mathcal{F}}}
\nc{\tcF}{{\widetilde{\mathcal{F}}}}
\nc{\cG}{{\mathcal{G}}}
\nc{\cH}{{\mathcal{H}}}
\nc{\cI}{{\mathcal{I}}}
\nc{\cJ}{{\mathcal{J}}}
\nc{\cK}{{\mathcal{K}}}
\nc{\cL}{{\mathcal{L}}}
\nc{\cM}{{\mathcal{M}}}
\nc{\cN}{{\mathcal{N}}}
\nc{\cO}{{\mathcal{O}}}
\nc{\cP}{{\mathcal{P}}}
\nc{\cQ}{{\mathcal{Q}}}
\nc{\cR}{{\mathcal{R}}}
\nc{\cS}{{\mathcal{S}}}
\nc{\cT}{{\mathcal{T}}}
\nc{\cU}{{\mathcal{U}}}
\nc{\cV}{{\mathcal{V}}}
\nc{\cW}{{\mathcal{W}}}
\nc{\cX}{{\mathcal{X}}}
\nc{\cY}{{\mathcal{Y}}}
\nc{\cZ}{{\mathcal{Z}}}

\nc{\baH}{\bar{H}}
\nc{\baS}{\bar{S}}
\nc{\baj}{\bar{\jmath}}

\nc{\bcU}{\bar{\cU}}

\nc{\rc}{{\mathrm{c}}}
\nc{\rf}{{\mathsf{f}}}
\nc{\rch}{{\mathrm{ch}}}
\nc{\rtd}{{\mathrm{td}}}

\nc{\rB}{{\mathrm{B}}}
\nc{\rC}{{\mathrm{C}}}
\nc{\rE}{{\mathrm{E}}}
\nc{\rF}{{\mathrm{F}}}
\nc{\rG}{{\mathrm{G}}}
\nc{\rH}{{\mathrm{H}}}
\nc{\rK}{{\mathrm{K}}}
\nc{\rL}{{\mathrm{L}}}
\nc{\rM}{{\mathrm{M}}}
\nc{\rN}{{\mathrm{N}}}
\nc{\rP}{{\mathrm{P}}}
\nc{\rR}{{\mathrm{R}}}
\nc{\rS}{{\mathrm{S}}}
\nc{\rT}{{\mathrm{T}}}
\nc{\rW}{{\mathrm{W}}}
\nc{\rX}{{\mathrm{X}}}

\nc{\rQ}{{\mathrm{Q}}}

\nc{\bA}{{\mathbf{A}}}
\nc{\bB}{{\mathbf{B}}}
\nc{\bC}{{\mathbf{C}}}
\nc{\bD}{{\mathbf{D}}}
\nc{\bE}{{\mathbf{E}}}
\nc{\bF}{{\mathbf{F}}}
\nc{\bG}{{\mathbf{G}}}
\nc{\bH}{{\mathbf{H}}}
\nc{\bI}{{\mathbf{I}}}
\nc{\bJ}{{\mathbf{J}}}
\nc{\bK}{{\mathbf{K}}}
\nc{\bL}{{\mathbf{L}}}
\nc{\bM}{{\mathbf{M}}}
\nc{\bN}{{\mathbf{N}}}
\nc{\bO}{{\mathbf{O}}}
\nc{\bP}{{\mathbf{P}}}
\nc{\bQ}{{\mathbf{Q}}}
\nc{\bR}{{\mathbf{R}}}
\nc{\bS}{{\mathbf{S}}}
\nc{\bT}{{\mathbf{T}}}
\nc{\bU}{{\mathbf{U}}}
\nc{\bV}{{\mathbf{V}}}
\nc{\bW}{{\mathbf{W}}}
\nc{\bX}{{\mathbf{X}}}
\nc{\bY}{{\mathbf{Y}}}
\nc{\bZ}{{\mathbf{Z}}}

\nc{\ba}{{\mathbf{a}}}
\nc{\bb}{{\mathbf{b}}}
\nc{\bc}{{\mathbf{c}}}
\nc{\bd}{{\mathbf{d}}}
\nc{\be}{{\mathbf{e}}}
\nc{\bg}{{\mathbf{g}}}
\nc{\bh}{{\mathbf{h}}}
\nc{\bi}{{\mathbf{i}}}
\nc{\bj}{{\mathbf{j}}}
\nc{\bk}{{\mathbf{k}}}
\nc{\bl}{{\mathbf{l}}}
\nc{\bm}{{\mathbf{m}}}
\nc{\bn}{{\mathbf{n}}}
\nc{\bo}{{\mathbf{o}}}
\nc{\bp}{{\mathbf{p}}}
\nc{\bq}{{\mathbf{q}}}
\nc{\br}{{\mathbf{r}}}
\nc{\bs}{{\mathbf{s}}}
\nc{\bt}{{\mathbf{t}}}
\nc{\bu}{{\mathbf{u}}}
\nc{\bv}{{\mathbf{v}}}
\nc{\bw}{{\mathbf{w}}}
\nc{\bx}{{\mathbf{x}}}
\nc{\by}{{\mathbf{y}}}
\nc{\bz}{{\mathbf{z}}}

\nc{\fA}{{\mathfrak{A}}}
\nc{\fB}{{\mathfrak{B}}}
\nc{\fC}{{\mathfrak{C}}}
\nc{\fD}{{\mathfrak{D}}}
\nc{\fE}{{\mathfrak{E}}}
\nc{\fF}{{\mathfrak{F}}}
\nc{\fG}{{\mathfrak{G}}}
\nc{\fH}{{\mathfrak{H}}}
\nc{\fI}{{\mathfrak{I}}}
\nc{\fJ}{{\mathfrak{J}}}
\nc{\fK}{{\mathfrak{K}}}
\nc{\fL}{{\mathfrak{L}}}
\nc{\fM}{{\mathfrak{M}}}
\nc{\fN}{{\mathfrak{N}}}
\nc{\fO}{{\mathfrak{O}}}
\nc{\fP}{{\mathfrak{P}}}
\nc{\fQ}{{\mathfrak{Q}}}
\nc{\fR}{{\mathfrak{R}}}
\nc{\fS}{{\mathfrak{S}}}
\nc{\fT}{{\mathfrak{T}}}
\nc{\fU}{{\mathfrak{U}}}
\nc{\fV}{{\mathfrak{V}}}
\nc{\fW}{{\mathfrak{W}}}
\nc{\fX}{{\mathfrak{X}}}
\nc{\fY}{{\mathfrak{Y}}}
\nc{\fZ}{{\mathfrak{Z}}}

\nc{\fa}{{\mathfrak{a}}}
\nc{\fb}{{\mathfrak{b}}}
\nc{\fc}{{\mathfrak{c}}}
\nc{\fd}{{\mathfrak{d}}}
\nc{\fe}{{\mathfrak{e}}}
\nc{\ff}{{\mathfrak{f}}}
\nc{\fg}{{\mathfrak{g}}}
\nc{\fh}{{\mathfrak{h}}}
\nc{\fj}{{\mathfrak{j}}}
\nc{\fk}{{\mathfrak{k}}}
\nc{\fl}{{\mathfrak{l}}}
\nc{\fm}{{\mathfrak{m}}}
\nc{\fn}{{\mathfrak{n}}}
\nc{\fo}{{\mathfrak{o}}}
\nc{\fp}{{\mathfrak{p}}}
\nc{\fq}{{\mathfrak{q}}}
\nc{\fr}{{\mathfrak{r}}}
\nc{\fs}{{\mathfrak{s}}}
\nc{\ft}{{\mathfrak{t}}}
\nc{\fu}{{\mathfrak{u}}}
\nc{\fv}{{\mathfrak{v}}}
\nc{\fw}{{\mathfrak{w}}}
\nc{\fx}{{\mathfrak{x}}}
\nc{\fy}{{\mathfrak{y}}}
\nc{\fz}{{\mathfrak{z}}}

\nc{\sA}{{\mathsf{A}}}
\nc{\sB}{{\mathsf{B}}}
\nc{\sC}{{\mathsf{C}}}
\nc{\sD}{{\mathsf{D}}}
\nc{\sE}{{\mathsf{E}}}
\nc{\sF}{{\mathsf{F}}}
\nc{\sG}{{\mathsf{G}}}
\nc{\sH}{{\mathsf{H}}}
\nc{\sI}{{\mathsf{I}}}
\nc{\sJ}{{\mathsf{J}}}
\nc{\sK}{{\mathsf{K}}}
\nc{\sL}{{\mathsf{L}}}
\nc{\sM}{{\mathsf{M}}}
\nc{\sN}{{\mathsf{N}}}
\nc{\sO}{{\mathsf{O}}}
\nc{\sP}{{\mathsf{P}}}
\nc{\sQ}{{\mathsf{Q}}}
\nc{\sR}{{\mathsf{R}}}
\nc{\sS}{{\mathsf{S}}}
\nc{\sT}{{\mathsf{T}}}
\nc{\sU}{{\mathsf{U}}}
\nc{\sV}{{\mathsf{V}}}
\nc{\sW}{{\mathsf{W}}}
\nc{\sX}{{\mathsf{X}}}
\nc{\sY}{{\mathsf{Y}}}
\nc{\sZ}{{\mathsf{Z}}}

\nc{\sa}{{\mathsf{a}}}
\nc{\sd}{{\mathsf{d}}}
\nc{\se}{{\mathsf{e}}}
\nc{\sg}{{\mathsf{g}}}
\nc{\sh}{{\mathsf{h}}}
\nc{\si}{{\mathsf{i}}}
\nc{\sj}{{\mathsf{j}}}
\nc{\sk}{{\mathsf{k}}}
\nc{\sm}{{\mathsf{m}}}
\nc{\sn}{{\mathsf{n}}}
\nc{\so}{{\mathsf{o}}}
\nc{\sq}{{\mathsf{q}}}
\nc{\sr}{{\mathsf{r}}}
\nc{\st}{{\mathsf{t}}}
\nc{\su}{{\mathsf{u}}}
\nc{\sv}{{\mathsf{v}}}
\nc{\sw}{{\mathsf{w}}}
\nc{\sx}{{\mathsf{x}}}
\nc{\sy}{{\mathsf{y}}}
\nc{\sz}{{\mathsf{z}}}

\nc{\oA}{{\overline{A}}}
\nc{\oB}{{\overline{B}}}
\nc{\oC}{{\overline{C}}}
\nc{\oD}{{\overline{D}}}
\nc{\oE}{{\overline{E}}}
\nc{\oF}{{\overline{F}}}
\nc{\oG}{{\overline{G}}}
\nc{\oH}{{\overline{H}}}
\nc{\oI}{{\overline{I}}}
\nc{\oJ}{{\overline{J}}}
\nc{\oK}{{\overline{K}}}
\nc{\oL}{{\overline{L}}}
\nc{\oM}{{\overline{M}}}
\nc{\oN}{{\overline{N}}}
\nc{\oO}{{\overline{O}}}
\nc{\oP}{{\overline{P}}}
\nc{\oQ}{{\overline{Q}}}
\nc{\oR}{{\overline{R}}}
\nc{\oS}{{\overline{S}}}
\nc{\oT}{{\overline{T}}}
\nc{\oU}{{\overline{U}}}
\nc{\oV}{{\overline{V}}}
\nc{\oW}{{\overline{W}}}
\nc{\oX}{{\overline{X}}}
\nc{\oY}{{\overline{Y}}}
\nc{\oZ}{{\overline{Z}}}

\nc{\oa}{{\overline{a}}}
\nc{\ob}{{\overline{b}}}
\nc{\oc}{{\overline{c}}}
\nc{\od}{{\overline{d}}}
\nc{\of}{{\overline{f}}}
\nc{\og}{{\overline{g}}}
\nc{\oh}{{\overline{h}}}
\nc{\oi}{{\overline{i}}}
\nc{\oj}{{\overline{j}}}
\nc{\ok}{{\overline{k}}}
\nc{\ol}{{\overline{l}}}
\nc{\om}{{\overline{m}}}
\nc{\on}{{\overline{n}}}
\nc{\oo}{{\overline{o}}}
\nc{\op}{{\overline{p}}}
\nc{\oq}{{\overline{q}}}
\nc{\os}{{\overline{s}}}
\nc{\ot}{{\overline{t}}}
\nc{\ou}{{\overline{u}}}
\nc{\ov}{{\overline{v}}}
\nc{\ow}{{\overline{w}}}
\nc{\ox}{{\overline{x}}}
\nc{\oy}{{\overline{y}}}
\nc{\oz}{{\overline{z}}}

\nc{\tA}{{\tilde{A}}}
\nc{\tB}{{\tilde{B}}}
\nc{\tC}{{\tilde{C}}}
\nc{\tD}{{\tilde{D}}}
\nc{\tE}{{\tilde{E}}}
\nc{\tF}{{\tilde{F}}}
\nc{\tG}{{\tilde{G}}}
\nc{\tH}{{\tilde{H}}}
\nc{\tI}{{\tilde{I}}}
\nc{\tJ}{{\tilde{J}}}
\nc{\tK}{{\tilde{K}}}
\nc{\tL}{{\tilde{L}}}
\nc{\tM}{{\tilde{M}}}
\nc{\tN}{{\tilde{N}}}
\nc{\tO}{{\tilde{O}}}
\nc{\tP}{{\tilde{P}}}
\nc{\tQ}{{\tilde{Q}}}
\nc{\tR}{{\tilde{R}}}
\nc{\tS}{{\widetilde{S}}}
\nc{\tT}{{\tilde{T}}}
\nc{\tU}{{\tilde{U}}}
\nc{\tV}{{\tilde{V}}}
\nc{\tW}{{\tilde{W}}}
\nc{\tX}{{\widetilde{X}}}
\nc{\tY}{{\tilde{Y}}}
\nc{\tZ}{{\tilde{Z}}}

\nc{\ta}{{\tilde{a}}}
\nc{\tb}{{\tilde{b}}}
\nc{\tc}{{\tilde{c}}}
\nc{\td}{{\tilde{d}}}
\nc{\te}{{\tilde{e}}}
\nc{\tf}{{\tilde{f}}}
\nc{\tg}{{\tilde{g}}}
\nc{\ti}{{\tilde{i}}}
\nc{\tj}{{\tilde{j}}}
\nc{\tk}{{\tilde{k}}}
\nc{\tl}{{\tilde{l}}}
\nc{\tm}{{\tilde{m}}}
\nc{\tn}{{\tilde{n}}}
\nc{\tp}{{\tilde{p}}}
\nc{\tq}{{\tilde{q}}}
\nc{\tr}{{\tilde{r}}}
\nc{\ts}{{\tilde{s}}}
\nc{\tu}{{\tilde{u}}}
\nc{\tv}{{\tilde{v}}}
\nc{\tw}{{\tilde{w}}}
\nc{\tx}{{\tilde{x}}}
\nc{\ty}{{\tilde{y}}}
\nc{\tz}{{\tilde{z}}}

\nc{\hA}{{\hat{A}}}
\nc{\hB}{{\hat{B}}}
\nc{\hC}{{\hat{C}}}
\nc{\hD}{{\hat{D}}}
\nc{\hE}{{\hat{E}}}
\nc{\hF}{{\hat{F}}}
\nc{\hG}{{\hat{G}}}
\nc{\hH}{{\hat{H}}}
\nc{\hI}{{\hat{I}}}
\nc{\hJ}{{\hat{J}}}
\nc{\hK}{{\hat{K}}}
\nc{\hL}{{\hat{L}}}
\nc{\hM}{{\hat{M}}}
\nc{\hN}{{\hat{N}}}
\nc{\hO}{{\hat{O}}}
\nc{\hP}{{\hat{P}}}
\nc{\hQ}{{\hat{Q}}}
\nc{\hR}{{\hat{R}}}
\nc{\hS}{{\widehat{S}}}
\nc{\hT}{{\hat{T}}}
\nc{\hU}{{\widehat{U}}}
\nc{\hV}{{\hat{V}}}
\nc{\hW}{{\hat{W}}}
\nc{\hX}{{\hat{X}}}
\nc{\hY}{{\hat{Y}}}
\nc{\hZ}{{\hat{Z}}}

\nc{\ha}{{\hat{a}}}
\nc{\hb}{{\hat{b}}}
\nc{\hc}{{\hat{c}}}
\nc{\hd}{{\hat{d}}}
\nc{\he}{{\hat{e}}}
\nc{\hf}{{\widehat{f}}}
\nc{\hg}{{\hat{g}}}
\nc{\hh}{{\hat{h}}}
\nc{\hi}{{\hat{i}}}
\nc{\hj}{{\hat{j}}}
\nc{\hk}{{\hat{k}}}
\nc{\hl}{{\hat{l}}}
\nc{\hm}{{\hat{m}}}
\nc{\hn}{{\hat{n}}}
\nc{\ho}{{\hat{o}}}
\nc{\hp}{{\hat{p}}}
\nc{\hq}{{\hat{q}}}
\nc{\hr}{{\hat{r}}}
\nc{\hs}{{\hat{s}}}
\nc{\hu}{{\hat{u}}}
\nc{\hv}{{\hat{v}}}
\nc{\hw}{{\hat{w}}}
\nc{\hx}{{\hat{x}}}
\nc{\hy}{{\hat{y}}}
\nc{\hz}{{\hat{z}}}

\nc{\eps}{\varepsilon}
\nc{\lan}{\big\langle}
\nc{\ran}{\big\rangle}
\nc{\kk}{{\Bbbk}}

\nc{\et}{{\mathrm{\acute{e}t}}}
\nc{\num}{{\mathrm{num}}}

\nc{\xrightiso}{ \xrightarrow{\ \raisebox{-0.5ex}[0ex][0ex]{$\sim$}\ }}

\def\onto{\ensuremath{\twoheadrightarrow}}

\def\bw#1#2{\textstyle{\bigwedge\hskip-0.9mm^{#1}}\hskip0.2mm{#2}}

\DeclareMathOperator{\Hom}{\mathrm{Hom}}
\DeclareMathOperator{\Ext}{\mathrm{Ext}}
\DeclareMathOperator{\End}{\mathrm{End}}

\DeclareMathOperator{\cRHom}{\mathrm{R}\mathcal{H}\mathit{om}}

\DeclareMathOperator{\cExt}{\mathcal{E}\!\mathit{xt}}
\DeclareMathOperator{\cEnd}{\mathcal{E}\mathit{nd}}

\DeclareMathOperator{\Bl}{\mathrm{Bl}}
\DeclareMathOperator{\Bs}{\mathrm{Bs}}
\DeclareMathOperator{\Sing}{\mathrm{Sing}}
\DeclareMathOperator{\Pic}{\mathrm{Pic}}

\DeclareMathOperator{\Sym}{\mathrm{Sym}}
\DeclareMathOperator{\Ker}{\mathrm{Ker}}
\DeclareMathOperator{\Coker}{\mathrm{Coker}}

\DeclareMathOperator{\Gr}{\mathrm{Gr}}
\DeclareMathOperator{\OGr}{\mathrm{OGr}}

\DeclareMathOperator{\LGr}{\mathrm{LGr}}

\DeclareMathOperator{\ev}{\mathrm{ev}}
\DeclareMathOperator{\rank}{\mathrm{rk}}

\DeclareMathOperator{\ch}{\mathrm{ch}}

\nc{\bkk}{{\overline{\kk}}}
\newcommand{\g}{{\mathrm{g}}}

\theoremstyle{plain}

\newtheorem{theorem}{Theorem}[section]

\newtheorem{lemma}[theorem]{Lemma}
\newtheorem{proposition}[theorem]{Proposition}
\newtheorem{corollary}[theorem]{Corollary}

\theoremstyle{definition}
\newtheorem{definition}[theorem]{Definition}

\theoremstyle{remark}
\newtheorem{remark}[theorem]{Remark}
\newtheorem*{remark*}{Remark}

\makeatletter
\@namedef{subjclassname@2020}{%
  \textup{2020} Mathematics Subject Classification}
\makeatother

\title{Mukai bundles on Fano threefolds}

\author{Arend Bayer}
\address{{\sloppy
\parbox{0.99\textwidth}{
School of Mathematics, University of Edinburgh\\
JCMB, Peter Guthrie Tait Road, Edinburgh EH9 3FD, UK
}\medskip}}
\email{arend.bayer@ed.ac.uk}

\author{Alexander Kuznetsov}
\address{{\sloppy
\parbox{0.99\textwidth}{
Algebraic Geometry Section, Steklov Mathematical Institute of Russian Academy of Sciences\\
8 Gubkin str., Moscow 119991 Russia
\\[3pt]
Laboratory of Algebraic Geometry, NRU Higher School of Economics, Russian Federation
}\medskip}}
\email{akuznet@mi-ras.ru}

\author{Emanuele Macr\`i}
\address{{\sloppy
\parbox{0.99\textwidth}{
Universit\'e Paris-Saclay, CNRS, Laboratoire de Math\'ematiques d'Orsay\\
Rue Michel Magat, B\^at. 307, 91405 Orsay, France
}\medskip}}
\email{emanuele.macri@universite-paris-saclay.fr}

\thanks{A.B. was partially supported by the EPSRC grant EP/R034826/1 and the ERC grant ERC-2018-CoG819864-WallCrossAG. 
A.K. was partially supported by the HSE University Basic Research Program. 
E.M. was partially supported by the ERC grant ERC-2020-SyG-854361-HyperK}

\begin{document}

\begin{abstract}
We give a proof of Mukai's Theorem on the existence of certain exceptional vector bundles on prime Fano threefolds.
To our knowledge this is the first complete proof in the literature. 
The result is essential for Mukai's biregular classification of prime Fano threefolds, 
and for the existence of semiorthogonal decompositions in their derived categories.

Our approach is based on Lazarsfeld's construction that produces vector bundles on a variety 
from globally generated line bundles on a divisor, 
on Mukai's theory of stable vector bundles on K3 surfaces, 
and on Brill--Noether properties of curves and (in the sense of Mukai) of K3 surfaces.
\end{abstract}

\maketitle

\setcounter{tocdepth}{1}
\tableofcontents

%%%%%%%%%%%%%%%%%%%%%%%%%

\section{Introduction}\label{sec:intro}

Let~$X$ be a complex smooth prime Fano threefold,
i.e., a smooth proper threefold such that
\begin{equation*}
\Pic(X) = \ZZ \cdot K_X
\end{equation*}
and~$-K_X$ is ample. 
The discrete invariant of~$X$ is given by the genus~$\g(X)$, defined by the equality
\begin{equation*}
(-K_X)^3 = 2\g(X)-2
\qquad\text{or}\qquad 
h^0(X,\cO_X(-K_X)) = \g(X) + 2,
\end{equation*}
and by Iskovskikh's fundamental result, see~\cite[Theorem 4.3.3]{Fano-book}, $g = \g(X)$ satisfies 
\begin{equation}
\label{eq:genus-bound}
2 \le g \le 12,
\qquad 
g \neq 11.
\end{equation} 
Following Fano's ideas, Iskovskikh also gave a birational description of all prime Fano threefolds.

\subsection{Mukai's Theorems}\label{subsec:IntroMukai}

If~$g \le 5$, the image of the anticanonical morphism~$X \to \P^{g + 1}$ has small codimension, 
which one can use to describe~$X$ as a complete intersection in a weighted projective space, see~\cite[Section 12.2]{Fano-book}. 
For~$g \ge 6$, the image of the anticanonical morphism is no longer a complete intersection, 
and an alternative method for the biregular classification of~$X$ was needed. 
This was provided by Mukai (see Section~\ref{subsec:history} for a brief account of the development of his ideas), 
and the main result of our paper is a complete proof of his theorem.

\begin{theorem}
\label{thm:Main}
Let~$\kk$ be an algebraically closed field of characteristic zero.
Let~$X$ be a smooth prime Fano threefold over~$\kk$ of genus $\g(X) = r \cdot s\geq 6$, for integers~$r,s \ge 2$.
Then there exists a unique vector bundle~$\cU_r$ on~$X$ such that
\begin{equation*}
\label{eq:intro-invariants}
\rank(\cU_r) = r,
\qquad 
\rc_1(\cU_r) = K_X,
\qquad 
\rH^\bullet(X,\cU_r) = 0,
\qquad
\Ext^\bullet(\cU_r,\cU_r) = \kk,
\end{equation*}
and~$\cU_r^\vee$ is globally generated with~$\dim \rH^0(X, \cU_r^\vee) = r + s$ and $\rH^{>0}(X, \cU_r^\vee)=0$.
\end{theorem}

We also prove a more general result (Theorem~\ref{thm:mb-x} and Corollary~\ref{cor:eu}),
saying that any prime Fano threefold~$X$ 
with factorial terminal singularities of genus~$\g(X) = r \cdot s$ with~$r,s \ge 2$
carries a maximal Cohen--Macaulay \emph{Mukai sheaf}~$\cU_X$ (see Definition~\ref{def:mb-x}), 
which is unique if~$\g(X) \ge 6$.
If, moreover, $\cU_X$ is locally free (which is automatic for smooth~$X$) and~$\g(X) \ge 6$, 
it satisfies all the properties of Theorem~\ref{thm:Main}.
In the follow-up paper~\cite{BKM:models} we prove that the Mukai sheaf~$\cU_X$ \emph{is} locally free, 
thus deducing Theorem~\ref{thm:Main} for all prime Fano threefolds with factorial terminal singularities.

Since~$\cU_r^\vee$ is globally generated, the anticanonical morphism of~$X$ factors, up to a linear projection, through a morphism
\begin{equation*}
X \to \Gr(r, r + s),
\end{equation*}
to the Grassmannian of $r$-dimensional subspaces in a vector space of dimension~$r + s$.
Using the decompositions~$g = 2 \cdot s$ for~$g \in \{6,8,10\}$ and~$g = 3 \cdot s$ for~$g \in \{9, 12\}$ and studying this morphism, Mukai found an explicit description of any prime Fano threefold~$X$ of genus~$g \ge 6$.

\begin{theorem} \label{thm:Mmodel}
\label{thm:descriptions}
Let~$\kk$ be an algebraically closed field of characteristic zero and let~$X$ be a smooth prime Fano threefold over~$\kk$ of genus~$g \ge 6$.
\begin{itemize}
\item 
If~$g = 6 $ then~$X$ is a complete intersection of a quadric and three hyperplanes 
in the cone in~$\P^{10}$ over~$\Gr(2,5) \subset \P^9$;
\item 
if~$g = 7$ then~$X$ is a dimensionally transverse linear section of~$\OGr_+(5,10) \subset \P^{15}$, 
the connected component of the Grassmannian of isotropic $5$-dimensional subspaces 
in a $10$-dimensional vector space endowed with a non-degenerate symmetric bilinear form;
\item 
if~$g = 8$ then~$X$ is a dimensionally transverse linear section of~$\Gr(2,6) \subset \P^{14}$;
\item 
if~$g = 9$ then~$X$ is a dimensionally transverse linear section of~$\LGr(3,6) \subset \P^{13}$, 
the Grassmannian of isotropic $3$-dimensional subspaces 
in a $6$-dimensional vector space endowed with a non-degenerate skew-symmetric bilinear form;
\item 
if~$g = 10$ then~$X$ is a dimensionally transverse linear section in~$\P^{13}$
of a $5$-dimensional homogeneous variety of the simple algebraic group of Dynkin type~$\mathbf{G}_2$; and
\item 
if~$g = 12$ then~$X$ is the subvariety in~$\Gr(3,7) \subset \P^{34}$, 
parameterizing $3$-dimensional subspaces isotropic for three skew-symmetric bilinear forms.
\end{itemize}
\end{theorem}

In~\cite[Theorem~1.1]{BKM:models} we prove this theorem for all prime Fano threefolds~$X$
with factorial terminal singularities of genus~$\g(X) \ge 6$.
We also generalize Theorems~\ref{thm:Main} and~\ref{thm:Mmodel} 
to Fano varieties of dimension~$n \ge 4$ with~$\Pic(X) = \ZZ \cdot H$ and~$K_X = -(n-2) H$.

The significance of Mukai's results for the study of Fano threefolds and K3 surfaces is hard to overestimate.
They have applications ranging from computing quantum periods (e.g., \cite{Corti:QuantumPeriods}), 
the classification of 2-Fano manifolds (e.g., \cite{AC:2Fano}), 
and Picard groups of moduli spaces of K3 surfaces (e.g., \cite{GLT}), 
to the study of curves (e.g., \cite{CFM,CU}) and automorphism groups (e.g., \cite{KPS}). 
They were also used to define interesting semiorthogonal decompositions of derived categories of Fano threefolds, see~\cite{K08}.

However, to the best of our knowledge, a complete direct proof of these results 
still cannot be found in the literature, see Section~\ref{subsec:history} for a detailed discussion.
Thus, our main goal is to fill this gap 
by giving a complete proof of Theorem~\ref{thm:Main} along with the generalizations discussed above in this paper, 
and of Theorem~\ref{thm:descriptions} along with generalizations also to higher dimension in~\cite{BKM:models}.

\subsection{Lazarsfeld bundles on K3 surfaces}\label{subsec:IntroLazarsfeldMukai}

Our proof of Theorem~\ref{thm:Main} uses ideas of Mukai and two extra ingredients: 
Lazarsfeld's fundamental observation in~\cite{Laz} that Brill--Noether properties of line bundles on a divisor 
are naturally encoded by its associated vector bundle on the ambient variety (perhaps first constructed in~\cite{Maruyama}), 
and Brill--Noether and Petri generality properties of K3 surfaces and curves, 
see~\S\ref{sec:BNtheory} for a review of what is relevant for us.

In particular, given a curve~$C$ with an embedding $j \colon C \hookrightarrow S$ into a K3 surface~$S$, and a globally generated line bundle~$\xi$ on~$C$, the vector bundle~$\bL_S(\xi)$ on~$S$ is defined by the exact sequence
\begin{equation*}
0 \to \bL_S(\xi) \to \rH^0(C, \xi) \otimes \cO_S \to j_*\xi \to 0.
\end{equation*}
We apply this construction when~$C$ is a smooth Brill--Noether--Petri general curve and~$\xi$ satisfies
\begin{equation}
\label{eq:intro-xi}
\deg(\xi) = (r - 1)(s + 1)
\qquad\text{and}\qquad
h^0(\xi) = r,
\qquad\text{hence}\qquad 
h^1(\xi) = s,
\end{equation}
where~$\g(C) = r \cdot s$ is a fixed factorization. 
(The existence of such~$\xi$ follows from the Brill--Noether generality of~$C$; 
however, $\xi$ is not unique, see Remark~\ref{rem:number-mp}.)
The vector bundle~$\bL_S(\xi)$ has the same invariants as the restriction of a Mukai bundle 
from an ambient Fano threefold, and we show that Petri generality of~$C$ implies that~$\bL_S(\xi)$ is simple and rigid, i.e., spherical.

It is easy to see that~$\bL_S(\xi)$ is stable if~$\Pic(S)$ is generated by the class of~$C$.
The first important step in our proof of Theorem~\ref{thm:Main} 
is the following general stability criterion for~$\bL_S(\xi)$
in terms of an extension of the class~$\xi$ to~$\Pic(S)$.

\begin{theorem}[cf.~{Theorem~\ref{thm:mb-s}}]
\label{obs:laz-s-stability}
Let~$\bL_S(\xi)$ be the spherical vector bundle associated to 
a smooth Brill--Noether--Petri general curve~$C$ on a $K3$ surface~$S$ 
and a globally generated line bundle~$\xi$ on~$C$ satisfying~\eqref{eq:intro-xi}, 
where~$r \in \{2,3\}$ and~$r \cdot s = \g(C)$.
Let~$H$ be the class of~$C$ in~$\Pic(S)$.
If~$\bL_S(\xi)$ is not $H$-Gieseker stable then~$\xi$ extends to~$S$. 
More precisely, there is a base point free divisor class~$\Xi \in \Pic(S)$ 
such that~$\cO_S(\Xi)\vert_C \cong \xi$ and
\begin{equation}
\label{eq:intro-Xi}
H \cdot \Xi = (r - 1)(s + 1),
\qquad
h^0(\cO_S(\Xi)) = r,
\qquad\text{and}\qquad
h^1(\cO_S(\Xi)) = h^2(\cO_S(\Xi)) = 0.
\end{equation}
Moreover, if~$\Pic(S)$ does not have a class~$\Xi$ satisfying~\eqref{eq:intro-Xi} 
then~$\bL_S(\xi)$ does not depend on~$C$ or~$\xi$.
\end{theorem}

We call divisor classes~$\Xi$ satisfying~\eqref{eq:intro-Xi} {\sf special Mukai classes of type~$(r,s)$}. 
In addition to $\Xi$ being an extension of $\xi$ from $C$ to $S$---with both $\xi$ and $\Xi$ 
being extremal in the sense of the Brill--Noether theory for $C$ and $S$, respectively---the significance of $\Xi$ 
is that the analogue $\bL_X(\Xi)$ on $X$ of the Lazarsfeld construction $\bL_S(\xi)$ on $S$ 
will produce our sought-after Mukai bundle on $X$.

The next result characterizes the restriction of Lazarsfeld bundles~$\bL_S(\xi)$ back to the curve~$C$.

\begin{theorem}[cf.~{Theorem~\ref{thm:mukai-extension}}]
\label{obs:mukai-extension}
Let~$\bL_S(\xi)$ be the spherical vector bundle associated to a Brill--Noether--Petri general curve~$C$ on a $K3$ surface~$S$ 
and a globally generated line bundle~$\xi$ on~$C$ satisfying~\eqref{eq:intro-xi},
where~$r \cdot s = \g(C)$.
Then there is an exact sequence
\begin{equation}
\label{eq:intro-me}
0 \to \big( \rH^0(C,\xi)^\vee \otimes \cO_C \big) / \xi^{-1} \to \bL_S(\xi)^{\vee}\vert_C \to \xi^{-1}(K_C) \to 0.
\end{equation}
Moreover, if~$\Pic(S) = \ZZ \cdot [C]$ and~$r \in \{2,3\}$, the extension class of~\eqref{eq:intro-me} 
is uniquely determined by the property that the connecting homomorphism
\begin{equation*}
\rH^0(C, \xi^{-1}(K_C)) \to \rH^1(C, (\rH^0(C,\xi)^\vee \otimes \cO_C) / \xi^{-1})
\end{equation*}
in the cohomology exact sequence of~\eqref{eq:intro-me} is zero.
\end{theorem}

Theorems~\ref{obs:laz-s-stability} and~\ref{obs:mukai-extension} 
form the technical core of our paper.
We prove them for~$r \in \{2,3\}$ (but for arbitrary~$g$), 
but we expect these results to be true for any~$r$ and~$g$.

Theorem \ref{obs:mukai-extension} is a variation of a long-running theme. 
When one assumes additionally that the extension class~\eqref{eq:intro-me} 
produces a slope-semistable vector bundle on~$C$, 
it fits into a series of results that stable vector bundles on~$C$ 
with large number of sections are restrictions of vector bundles on~$S$, 
see, e.g., \cite[Section~10]{Mukai:noncommutativizability}, \cite{ABS:Mukai}, \cite{Fey:Mukai}. 
Closest to our case is~\cite{Fey:HK}, as it covers the case of spherical vector bundles; 
it treats the case when~$g \gg 0$; see~\cite[Remark~4.3]{Fey:HK} for the precise list of conditions. 
Our proof is more elementary, but its idea is motivated by Feyzbakhsh' approach.

Even closer are the results of~\cite[Section~3]{Voisin:Wahl}: 
the rank~2 case of Theorem~\ref{obs:mukai-extension} follows from Voisin's results in the case where~$S$ is general.
Indeed, Lemma~3.18 of~\cite{Voisin:Wahl} shows the uniqueness of the extension under the assumption~\cite[3.1(i)]{Voisin:Wahl}. 
Moreover, by the proof of~\cite[Proposition~4.1]{Voisin:Wahl}, this assumption holds 
when~$C$ is a general curve on a general K3 surface with~$\Pic(S) = \ZZ\cdot [C]$.

Meanwhile, Theorem~\ref{obs:laz-s-stability} is reminiscent to results on the Donagi--Morrison Conjecture; 
see, e.g.,~\cite{L-C,AH, Hubarcak}. 
Often, such results show that line bundles on~$C$ with negative Brill--Noether number are, 
after adding an effective divisor, obtained by restriction of a line bundle on~$S$, 
and take advantage of the fact that the associated Lazarsfeld bundle is not simple. 
In our case, the line bundle on~$C$ has zero Brill--Noether number,
the Lazarsfeld bundle is simple but unstable, 
and we take advantage of instability to obtain its precise extension to a line bundle on~$S$.

\subsection{Our argument for Theorem~\ref{thm:Main}}\label{ss:argument}

We use Theorems~\ref{obs:laz-s-stability} and~\ref{obs:mukai-extension} 
to prove Theorem~\ref{thm:Main} and Theorem~\ref{thm:mb-x} as follows.
Let~$X$ be a prime Fano threefold of genus~$g \ge 6$
and let~$S \subset X$ be a smooth hyperplane section such that
\begin{equation}
\label{eq:intro-pic-1}
\Pic(S) = \ZZ \cdot H_S, 
\end{equation}
where~$H_S$ stands for the restriction of~$H \coloneqq -K_X$ to~$S$.
First, we apply~\cite[Theorem]{Laz} to find a Brill--Noether--Petri general curve~$C \subset S$ in the linear system~$|H_S|$ 
such that the pencil~$\{S_t\}_{t \in \P^1}$ of anticanonical divisors of~$X$ through~$C$ 
is a Lefschetz pencil.
This pencil contains~$S$, hence~\mbox{$\Pic(S_t) \cong \ZZ \cdot H_{S_t}$} for very general~$t$.

Next, we consider the blowup~$\tX \coloneqq \Bl_C(X)$ as a family of K3 surfaces over~$\P^1$, 
containing the fixed curve $C$ as the exceptional divisor $E = C \times \P^1$. 
We choose a globally generated line bundle~$\xi \in \Pic(C)$ satisfying~\eqref{eq:intro-xi}, where~$r \cdot s = \g(C)$,
and apply a relative version  of the Lazarsfeld construction 
to the line bundle~$\xi \boxtimes \cO_{\P^1}$ on~$E$. As~$\Pic(S_t) = \ZZ \cdot H_{S_t}$ for very general~$t \in \P^1$, 
the restriction~$\bL_{\tX/\P^1}(\xi \boxtimes \cO)\vert_{S_t}$ 
of the Lazarsfeld bundle~$\bL_{\tX/\P^1}(\xi \boxtimes \cO)$ on~$\tX$ is stable in this case.
Then Theorem~\ref{obs:mukai-extension} implies that the extension classes 
of the sequences~\eqref{eq:intro-me} associated with the restrictions
\begin{equation*}
(\bL_{\tX/\P^1}(\xi \boxtimes \cO)\vert_{S_t}) \vert_C \cong
(\bL_{S_t}(\xi)) \vert_C
\end{equation*}
are all proportional to a fixed class. 
It follows that this extension class vanishes for some~$t_0 \in \P^1$, 
hence the Lazarsfeld bundle~$\bL_{\tX/\P^1}(\xi \boxtimes \cO)\vert_{S_{t_0}}$ is not stable and, 
by Theorem~\ref{obs:laz-s-stability}, that the surface~$S_{t_0}$ has a special Mukai class~$\Xi$ of type~$(r,s)$.
More precisely, since~$S_{t_0}$ may be singular, we consider the minimal resolution~$\sigma \colon \tS_{t_0} \to S_{t_0}$ 
and apply Theorem~\ref{thm:mb-s}, a more precise version of Theorem~\ref{obs:laz-s-stability},
to deduce the existence of the divisor class~$\Xi$ on the minimal resolution~$\tS_{t_0}$ of~$S_{t_0}$.

Finally, we apply the Lazarsfeld construction to the sheaf~$\sigma_*\cO_{S_{t_0}}(\Xi)$ to 
obtain a maximal Cohen--Macaulay sheaf~$\bL_X(\Xi)$ on~$X$, and check that, 
if~$g \ge 6$ it satisfies the requirements of Theorem~\ref{thm:Main}, 
see Corollaries~\ref{cor:mb-x-existence} and~\ref{cor:eu} for details.
To prove exceptionality (and uniqueness) of~$\bL_X(\Xi)$ 
we use an extra cohomology vanishing for the Lazarsfeld bundle on~$S$ proved in Proposition~\ref{prop:l-l-m-h}.

\subsection{Further generalizations}
\label{ss:further}

It would be interesting and useful to extend our argument 
(and thus Theorem~\ref{thm:Main}, its extension to the factorial terminal case in Theorem~\ref{thm:mb-x}, 
as well as Theorem~\ref{thm:descriptions} and the more general version proved in \cite{BKM:models}) to more general situations.

First, Mukai claims that Theorem~\ref{thm:Main} (and even Theorem~\ref{thm:Mmodel}) 
extends to \emph{Brill--Noether general} Fano threefolds
(see~\cite[Definition~6.4 and Theorem~6.5]{Muk:New}).
Moreover, similar results hold for some smooth Fano threefolds which are not Brill--Noether general:
one can construct Mukai bundles on some Fano threefolds of higher Picard rank (see~\cite[Propositions~7.7 and~7.12]{K22})
and ``almost Mukai'' bundles (i.e., vector bundles that satisfy all the properties of Theorem~\ref{thm:Main} 
except possibly for the global generation of the dual) 
on some nonfactorial 1-nodal Fano threefolds (see~\cite[Proposition~3.3 and Remark~3.5]{KS23}).
This suggests that Theorem~\ref{thm:Main} can be generalized even further.
However,  Theorem~\ref{thm:descriptions} fails for some nonfactorial threefolds, see~\cite{Muk22}.

The second possible direction of generalization is to positive characteristic.

In both cases the most problematic step of our argument 
is the existence of a smooth anticanonical divisor~$S \subset X$ satisfying~\eqref{eq:intro-pic-1}. 
The existence of such~$S$ is used twice in the proof: first to apply~\cite{Laz} 
and conclude that~$X$ contains a Brill--Noether--Petri general curve~$C$, 
and second to prove the uniqueness of the extension class of the sequence~\eqref{eq:intro-me} in Theorem~\ref{thm:mukai-extension}.
Condition~\eqref{eq:intro-pic-1} cannot be satisfied over~$\bar{\mathbb{F}}_p$ with~$p \ne 2$
(because in this case the Picard number of a K3 surface is even, see, e.g., \cite[Corollary~17.2.9]{Hu}), 
nor for nonfactorial~$X$ (because the Picard number of an anticanonical divisor of~$X$ 
is greater or equal than the rank of the class group of~$X$).
However, we hope that~\eqref{eq:intro-pic-1} can be replaced by Brill--Noether generality of~$S$.

Another interesting question is the following. Theorems~\ref{obs:laz-s-stability} and~\ref{obs:mukai-extension} are concerned with special vector bundles on K3 surfaces and curves and do not involve a Fano threefold.
Therefore, they make sense for any value of~$g$, not only for the Fano range~\eqref{eq:genus-bound}, and consequently it is interesting to ask if the corresponding results (Theorems~\ref{thm:mb-s} and~\ref{thm:mukai-extension}) stay true for~$r > 3$.
Again, we currently have no approach towards proofs in this generality.

\subsection{History of Theorem~\ref{thm:Main}}
\label{subsec:history}

Theorem~\ref{thm:Main} was first announced by Mukai in~\cite{Mukai:PNAS}. 
The argument sketched there, expanded upon in~\cite[\S5.2]{Fano-book}, 
relies on Fujita's extension theorem in~\cite{Fujita:extendampledivisor} 
for sheaves~$\cF$ on an ample divisor~$D$ to the ambient variety~$X$. 
However, Fujita's theorem does not apply when~$D$ is a surface, 
as it relies on the vanishing of~$\rH^2(D, \cEnd(\cF)\otimes \cO_X(-nD)\vert_D)$ for~$n \ge 1$. 

We are also aware of a sketch of an argument written by Mukai in~\cite{Muk:New} below Theorem~6.5,
which relies on the analogue~\cite[Theorem~4.7]{Muk:New} of Theorem~\ref{thm:descriptions} for K3 surfaces, 
the proof of which in turn is also only given as a sketch.

An attempt at a different proof was given in~\cite[Theorem 6.2]{BLMS}. 
However, the proof given there is also incomplete at best.
It claims that for a K3 surface~$S \subset X$, there is a \emph{stable} bundle~$\cU$ of the right invariants; 
for higher Picard rank of~$S$, this is not clear a priori. 
More crucially, it claims that as~$S$ varies in a pencil, 
the restriction of~$\cU$ to the base locus remains constant; again, this statement is correct (as it can be deduced from Theorem~\ref{thm:mukai-extension}), but 
 the deformation argument sketched in \cite{BLMS} does not work.

Finally, in the cases where~$g \in \{6,8\}$, Theorem~\ref{thm:Main} (and even Theorem~\ref{thm:descriptions}) 
was deduced by Gushel\ensuremath{'} in~\cite{Gu6,Gu8} from the existence of special elliptic curves 
on smooth hyperplane sections~$S \subset X$ 
(these elliptic curves are precisely special Mukai classes of types~$(2,3)$ and~$(2,4)$, respectively).
The existence of elliptic curves on other Fano threefolds of even genus and a proof of Theorem~\ref{thm:Main} for~$r = 2$
was recently given in~\cite{CFK}.
This argument seems to be independent of Mukai's classification results. 
It is more indirect, as it relies on the irreducibility of the moduli space of Fano threefolds, proved in~\cite[Theorem~7]{CLM}.

\subsection{Notation and conventions}\label{sec:notation}

We work over an algebraically closed field~$\kk$ of characteristic zero;
all schemes and morphisms are assumed to be~$\kk$-linear.
Given a morphism~$f$ we let~$f_*$ and~$f^*$ denote the \emph{underived} pushforward and pullback functors;
the corresponding derived functors are denoted by~$\rR f_*$ and~$\rL f^*$, respectively.
We write~$[1]$ for the shift functor in the derived category.

For a coherent sheaf~$\cF$ on a scheme~$S$ we write
\begin{equation*}
h^i(\cF) \coloneqq \dim \rH^i(S, \cF).
\end{equation*}
Furthermore, we write
\begin{equation*}
\upchi(\cF) \coloneqq \sum (-1)^i h^i(\cF)
\qquad\text{and}\qquad 
\upchi(\cF_1, \cF_2) \coloneqq \sum (-1)^i \dim \Ext^i(\cF_1, \cF_2)
\end{equation*}
for the Euler characteristic of a sheaf~$\cF$ and the Euler bilinear form, respectively.

For a sheaf~$\cF$ of finite projective dimension on a K3 surface~$S$ with du Val singularities we write
\begin{equation*}
	\rv(\cF) \coloneqq (\rank(\cF), \rc_1(\cF), \ch_2(\cF) + \rank(\cF)) \in \ZZ \oplus \Pic(S) \oplus \ZZ
\end{equation*}
for the {\sf Mukai vector} of~$\cF$.
The {\sf Mukai pairing} on~$\ZZ \oplus \Pic(S) \oplus \ZZ$ is given by
\begin{equation*}
	\langle (r_1,D_1,s_1), (r_2,D_2,s_2) \rangle \coloneqq -r_1s_2 + D_1\cdot D_2 - s_1r_2.
\end{equation*} 
The Riemann--Roch Theorem translates then into the equality
\[
\upchi(\cF_1,\cF_2) = -\langle \rv(\cF_1), \rv(\cF_2) \rangle.
\]
In particular, since~$\rv(\cO_S) = (1, 0 , 1)$, 
for~$\rv(\cF) = (r, D, s)$ we obtain
\begin{equation}
\label{eq:rr-k3}
\upchi(\cF) = r + s,
\qquad 
\upchi(\cF, \cF) = 2rs - D^2.
\end{equation}
Finally, given an ample divisor $A$ on $S$, 
we define the slope and the reduced Euler characteristic of~$\cF$ as follows.
\begin{equation}
\label{eq:Paris20231230}
\upmu_A(\cF) \coloneqq \frac{A \cdot \rc_1(\cF)}{\rank(\cF)}, 
\qquad  
\updelta(\cF)\coloneqq \frac{\upchi(\cF)}{\rank(\cF)}.
\end{equation}

\subsection*{Acknowledgments}
The paper benefited from many useful discussions with the following people which we gratefully thank: 
Asher Auel,
Marcello Bernardara, 
Tzu-Yang Chou, 
Ciro Ciliberto,
Olivier Debarre,
Daniele Faenzi,
Flaminio Flamini,
Andreas Knutsen,
Chunyi Li, 
Shengxuan Liu, 
Jonathan Ng,
Yuri Prokhorov,
Claire Voisin,
Peter Yi Wei.
We also thank Gavril Farkas for helping us put Theorems~\ref{obs:laz-s-stability} and~\ref{obs:mukai-extension} into appropriate context and the referee for many useful comments and suggestions.
We also acknowledge the influence of the work of Shigeru Mukai, which provided and keeps providing a source of inspiration for us.

%%%%%%%%%%%%%%%%%%%%%%%%%

\section{Brill--Noether theory}\label{sec:BNtheory}

In this section we recall some facts from Brill--Noether theory for curves and K3 surfaces
and introduce Mukai special classes that play an important role in the rest of the paper.

\subsection{Curves}

We refer to~\cite[Ch.~IV--V]{ACGH} for a general treatment of Brill--Noether theory for curves.

Let~$C$ be a smooth proper curve of genus~$g$.
For~$1 \le d \le 2g-3$ and~$r \ge 0$ the {\sf Brill--Noether locus}~$\rW_{d}^r(C) \subset \Pic^d(C)$ is defined as
\begin{equation*}
\rW_d^r(C) \coloneqq \{ \cL \in \Pic^d(C) \mid h^0(\cL) \ge r + 1,\ h^1(\cL) \ge 1 \}.
\end{equation*}
If a line bundle~$\cL$ corresponds to a point of~$\rW_d^r(C) \setminus \rW_d^{r+1}(C)$ the cotangent space to~$\rW_d^r(C)$ at this point is equal to the cokernel of the {\sf Petri map}
\begin{equation*}
\rH^0(C,\cL) \otimes \rH^0(C, \cL^{-1}(K_C)) \to \rH^0(C, \cO_C(K_C)).
\end{equation*}

It is well known (see, e.g., \cite[Theorem~V.1.1]{ACGH}) that for~$r \ge d - g$ one has
\begin{equation*}
\dim \rW_d^r(C) \ge g - (r + 1)(g - d + r),
\end{equation*}
and that if the right-hand side (the expected dimension of~$\rW_d^r(C)$) is nonnegative, $\rW_d^r(C) \ne \varnothing$.

\begin{definition}
\label{def:bn-curves}
A smooth proper curve~$C$ is {\sf Brill--Noether general} ({\sf BN-general}) if the locus~$\rW_d^r(C)$ is nonempty if and only if~$(r + 1)(g - d + r) \le g$; in other words,
\begin{equation*}
h^0(\cL) \cdot h^1(\cL) \le g
\end{equation*}
for any line bundle~$\cL$ on~$C$.
Furthermore, we will say that~$C$ is {\sf Brill--Noether--Petri general} ({\sf BNP-general}) if it is BN-general and the Petri map is injective for any line bundle~$\cL$ on~$C$.
\end{definition}

A general curve of genus~$g$ is known to be BNP-general, see~\cite{Gie} and~\cite{Laz}.
For BNP-general curves the locus~$\rW_d^r(C) \setminus \rW_d^{r+1}(C)$ is smooth and nonempty of dimension~$g - (r + 1)(g - d + r)$ whenever this number is nonnegative (\cite[Proposition~IV.4.2 and Theorem~V.1.7]{ACGH}).

The following definition axiomatizes the situation studied in~\cite[\S3]{Muk3}.

\begin{definition}
\label{def:mukai-pair}
A pair of line bundles~$(\xi,\eta)$ on a curve~$C$ of genus~$g = r \cdot s$ is a {\sf Mukai pair of type~$(r,s)$} if~$\xi \otimes \eta \cong \cO_C(K_C)$,
\begin{equation}
\label{eq:h0h1-xi}
\deg(\xi) = (r - 1)(s + 1),
\qquad
h^0(\xi) = r,
\qquad\text{hence}\qquad 
h^1(\xi) = s,
\end{equation} 
and both~$\xi$ and~$\eta$ are globally generated.
\end{definition}

If~$(\xi,\eta)$ is a Mukai pair of type~$(r,s)$ then Serre duality implies
\begin{equation*}
\label{eq:eta-h0}
\deg(\eta) = (r + 1)(s - 1),
\qquad 
h^0(\eta) = s,
\qquad\text{and}\qquad 
h^1(\eta) = r,
\end{equation*}
and~$(\eta,\xi)$ is a Mukai pair of type~$(s,r)$. 
Thus, the definition is symmetric.

The following lemma deduces the existence of Mukai pairs from Brill--Noether theory.

\begin{lemma}
\label{lem:xi-eta}
If~$C$ is a BN-general curve of genus~$g = r \cdot s$ 
and a line bundle~$\xi \in \Pic(C)$ satisfies~$\deg(\xi) = (r - 1)(s + 1)$ and~$h^0(\xi) \ge r$
then~$(\xi,\xi^{-1}(K_C))$ is a Mukai pair of type~$(r,s)$.
In particular, any BN-general curve of genus~$g = r \cdot s$ has a Mukai pair of type~$(r,s)$.

If~$C$ is BNP-general and~$(\xi,\eta)$ is a Mukai pair, the Petri maps for~$\xi$ and~$\eta$ are isomorphisms.
\end{lemma}

\begin{proof}
If~$\deg(\xi) = (r - 1)(s + 1)$ then~$\upchi(\xi) = (r - 1)(s + 1) + 1 - rs = r - s$ by Riemann--Roch.
So, if~$h^0(\xi) > r$ then~$h^1(\xi) > s$ which contradicts the BN-generality of~$C$, hence~$h^0(\xi) = r$ and~$h^1(\xi) = s$.
Similarly, if~$\xi$ is not globally generated at a point~$P$ then~$\xi(-P)$ violates the BN property.
Therefore, $\xi$ is globally generated, and by the symmetry of the definition the same holds for $\eta \coloneqq \xi^{-1}(K_C)$.
Thus, $(\xi,\eta)$ is a Mukai pair of type~$(r,s)$.

Since~$r \cdot s = g$, the expected dimension of~$\rW_{(r-1)(s+1)}^{r-1}(C)$ is~$0$, 
hence~$\rW_{(r-1)(s+1)}^{r-1}(C) \ne \varnothing$ by~\cite[Theorem~V.1.1]{ACGH}. 
This means that there is a line bundle~$\xi$ of degree~$(r-1)(s+1)$ with~$h^0(\xi) \ge r$,
hence~$(\xi,\xi^{-1}(K_C))$ is a Mukai pair of type~$(r,s)$.

Finally, if~$C$ is BNP-general, the Petri map for~$\xi$ is injective,
and as the dimensions of its source and target are the same, 
it is an isomorphism; again, the same property for $\eta$ follows.
\end{proof}

\begin{remark}
\label{rem:number-mp}
If~$C$ is BNP-general, the Brill--Noether locus~$\rW^{r-1}_{(r-1)(s+1)}(C)$ is reduced and zero-dimensional of length equal to 
\begin{equation*}
\rN(r,s) \coloneqq 
(rs)! \, \frac{ \prod_{i=1}^{r-1} i! \prod_{i=1}^{s-1} i!}{ \prod_{i=1}^{r+s-1} i!},
\end{equation*}
the degree of the Grassmannian~$\Gr(r,r+s)$, see~\cite[Theorem~1]{EH}.
Therefore, this degree is exactly the number of Mukai pairs of type~$(r,s)$ on a BNP-general curve.
\end{remark}

The following proposition will be important for the study of Lazarsfeld bundles.
Part~\ref{it:xi-k}, for any $r$, is attributed to Castelnuovo in~\cite[Theorem~1.11]{Ciliberto}.

\begin{proposition}\label{prop:xi-eta}
Let~$C$ be a BNP-general curve of genus~$g = r \cdot s$ with~$s \ge r$ and~$r \in \{2, 3\}$.
\begin{aenumerate}
\item
\label{it:xi-k}
If~$(\xi, \eta)$ is a Mukai pair on~$C$ of type~$(r,s)$ then the natural morphism
\begin{equation*}
\rH^0(C,\xi) \otimes \rH^0(C, \cO_C(K_C)) \to \rH^0(C, \xi(K_C))
\end{equation*}
is surjective.
\item
\label{it:xi-eta-k}
If~$(\xi_1, \eta_1)$ and~$(\xi_2, \eta_2)$ are  Mukai pairs of type~$(r,s)$ with~$\xi_1 \not\cong \eta_2$ then the natural morphism
\begin{equation*}
\rH^0(C, \xi_1) \otimes \rH^0(C, \eta_2(K_C)) \to \rH^0(C, \xi_1 \otimes \eta_2(K_C))
\end{equation*}
is surjective.
\end{aenumerate}
\end{proposition}

\begin{proof}
\ref{it:xi-k}
First, assume~$r = 2$. As~$\xi$ is globally generated by definition, we have an exact sequence
\begin{equation*}
0 \to \xi^{-1} \to \rH^0(C, \xi) \otimes \cO_C \to \xi \to 0.
\end{equation*}
Twisting it by~$K_C$, we obtain
\begin{equation*}
0 \to \eta \to \rH^0(C, \xi) \otimes \cO_C(K_C) \to \xi(K_C) \to 0.
\end{equation*}
As $h^1(\eta) = h^0(\xi) = h^1(\rH^0(C, \xi) \otimes \cO_C(K_C))$ by Serre duality 
and $h^1(\xi(K_C)) = 0$ because the degree of~$\xi$ is positive, 
this exact sequence induces a surjection of global sections, which is our claim.

Now assume~$r = 3$, and let~$V = \rH^0(C, \xi)$. 
Since~$\xi$ is globally generated, it induces
a morphism~$C \to \P^2 = \P(V^{\vee})$.
The pullback of the Koszul complex on~$\P^2$ is an exact sequence
\begin{equation*}
0 \to 
\wedge^3 V \otimes \xi^{-3} \to 
\wedge^2 V \otimes \xi^{-2} \to 
V \otimes \xi^{-1} \to 
\cO_C \to 0.
\end{equation*}
Twisting it by~$\xi(K_C)$, we obtain an exact sequence
\begin{equation*}
0 \to 
\wedge^3 V \otimes \xi^{-2}(K_C) \to 
\wedge^2 V \otimes \xi^{-1}(K_C) \to 
V \otimes \cO_C(K_C) \to 
\xi(K_C) \to 0.
\end{equation*}
The $\bE_1$ page of its hypercohomology spectral sequence has the form
\begin{equation*}
\xymatrix@C=.9em@R=2ex{
\wedge^3 V \otimes \rH^1(C,\xi^{-2}(K_C)) \ar[r] &
\wedge^2 V \otimes \rH^1(C,\xi^{-1}(K_C)) \ar[r] \ar@{-->}[drr] &
V &
0
\\
\wedge^3 V \otimes \rH^0(C,\xi^{-2}(K_C)) \ar[r] &
\wedge^2 V \otimes \rH^0(C,\xi^{-1}(K_C)) \ar[r] &
V  \otimes \rH^0(C,\cO_C(K_C)) \ar[r] &
\rH^0(C,\xi(K_C)).
}
\end{equation*}
Since it converges to zero, surjectivity of the last arrow in the lower row 
is equivalent to the dashed arrow in the~$\bE_2$ page being zero. 
To show this, it is enough to check that the upper row is exact in the second term.
Tensoring it with~$\wedge^3 V^\vee$, 
using the natural identifications~$V \otimes \wedge^3 V^\vee = \wedge^2 V^\vee$ 
and~$\wedge^2 V \otimes \wedge^3 V^\vee = V^\vee$,
and dualizing it afterwards, we obtain the sequence
\begin{equation*}
\wedge^2 V \to V \otimes V \to \rH^0(C,\xi^2),
\end{equation*}
where the first arrow is the natural embedding.
Therefore it is enough to check that the natural map~$\Sym^2V = \Sym^2\rH^0(C,\xi) \to \rH^0(C,\xi^2)$ is injective.
But if it is not injective, the image of the map~$C \to \P^2$ given by~$\xi$ is contained in a conic in~$\P^2$, hence this map factors through a map~$C \to \P^1$ of degree~$\tfrac12\deg(\xi) = \tfrac12(r - 1)(s + 1) = s + 1$, hence~$\rW_{s + 1}^1(C) \ne \varnothing$, which contradicts the BN-generality of the curve~$C$ as $g - 2 (g - s) = 2s - g < 0$.

\ref{it:xi-eta-k}
As before, consider the morphism~$C \to \P(\rH^0(C, \xi_1)^\vee)$ and the pullback of the Koszul complex
\begin{equation*}
\dots \to 
\wedge^2\rH^0(C, \xi_1) \otimes \xi_1^{-2} \to 
\rH^0(C, \xi_1) \otimes \xi_1^{-1} \to
\cO_C \to 0.
\end{equation*}
Tensoring it by~$\xi_1 \otimes \eta_2(K_C)$, we obtain an exact sequence
\begin{equation*}
\dots \to 
\wedge^2\rH^0(C, \xi_1) \otimes \xi_1^{-1} \otimes \eta_2(K_C) \to 
\rH^0(C, \xi_1) \otimes \eta_2(K_C) \to
\xi_1 \otimes \eta_2(K_C) \to 0.
\end{equation*}
Now note that~$\deg(\xi_1) = r \le s = \deg(\eta_2)$,
and if this is an equality, $\xi_1^{-1} \otimes \eta_2$ is a nontrivial 
(by the assumption~$\xi_1 \not\cong \eta_2$) line bundle of degree zero, 
hence~$\rH^1(C, \xi_1^{-1} \otimes \eta_2(K_C)) = 0$.
Therefore, the hypercohomology spectral sequence proves the surjectivity of the required map.
\end{proof}

\subsection{Quasipolarized K3 surfaces}
\label{ss:qp-k3} 

In Section~\ref{sec:fano} we will have to deal with polarized K3 surfaces with du Val singularities;
they will appear as special members of generic pencils of hyperplane sections of terminal Fano threefolds.
A minimal resolution of such surface is a smooth quasipolarized K3 surface.
In this subsection we explain this relation and prove a few useful results. 
For readers only interested in the case where~$X$ is smooth, 
it is enough to consider the case of K3 surfaces that have at most one ordinary double point,
since in this case a generic pencil is a Lefschetz pencil. 

So, let~$(\baS,\baH)$ be a K3 surface with du Val singularities and an ample class~$\baH \in \Pic(\baS)$.
Let 
\begin{equation*}
\sigma \colon S \to \baS,
\qquad 
H \coloneqq \sigma^*(\baH)
\end{equation*}
be the minimal resolution of singularites and the pullback of the ample class.
Then~$S$ is a smooth K3 surface and~$H \in \Pic(S)$ is a {\sf quasipolarization}, i.e., a big and nef divisor class.
The exceptional locus of~$\sigma$ is formed by a finite number of smooth rational curves~$R_i \subset S$
that are characterized by the condition~$H \cdot R_i = 0$ and form an ADE configuration that we denote by~$\fR(S,H)$.

Conversely, if~$(S,H)$ is a quasipolarized K3 surface, 
all irreducible curves orthogonal to~$H$ are smooth and rational and they
form an ADE configuration~$\fR(S,H)$, see~\cite[(4.2)]{SD}; moreover,  a multiple of~$H$ defines a proper birational morphism~$\sigma \colon S \to \baS$ onto a K3 surface~$\baS$ with du Val singularities
contracting the configuration~$\fR(S,H)$ so that~$H$ is the pullback of an ample class~$\baH \in \Pic(\baS)$.
The morphism~$\sigma$ will be referred to as {\sf the contraction of~$\fR(S,H)$}.

The two above constructions are mutually inverse, so the language of polarized du Val K3 surfaces 
is equivalent to the language of quasipolarized smooth K3 surfaces.
Below we use the language that is more convenient, depending on the situation.

If~$H$ is a quasipolarization of a smooth (or a polarization of a du Val) K3 surface~$S$ then
\begin{equation}
\label{eq:genus-s}
H^2 = 2g - 2,
\qquad 
h^0(\cO_S(H)) = g + 1,
\qquad\text{and}\qquad 
h^1(\cO_S(H)) = h^2(\cO_S(H)) = 0,
\end{equation}
where~$g \ge 2$ is an integer called {\sf the genus} of~$(S,H)$.
The line bundle~$\cO_S(H)$ is globally generated if and only if the linear system~$|H|$ 
contains a smooth connected curve~$C \subset S$ of genus~$g$;
one direction follows from Bertini's Theorem and~\cite[Proposition~2.6]{SD}, the other is~\cite[Theorem~3.1]{SD}.

If~$(S,H)$ is a quasipolarized smooth K3 surface, 
the irreducible curves~$R_i$ that form the ADE configuration~$\fR(S,H)$ are called {\sf the simple roots}; 
they satisfy~$R_i^2 = -2$ and generate a root system in~$H^\perp \subset \Pic(S) \otimes \RR$;
its {\sf positive roots} are nonnegative linear combinations~$R = \sum a_i R_i$ such that~$R^2 = -2$,
and the reflections in~$R_i$ generate an action of the corresponding Weyl group.

A basic fact about Weyl groups (see, e.g., Corollary~2 to Proposition~IV.17 in~\cite{Bourbaki}) 
tells us that for any non-simple positive root~$R$ there is a simple root~$R_j$ such that~$R \cdot R_j = -1$.
Then~$R' = R - R_j$ is also a positive root and~$R' \cdot R_j = 1$. 
Therefore, considering~$R$ and~$R'$ as Cartier divisors on~$S$, we obtain exact sequences
\begin{equation}
\label{eq:root-seq-1}
0 \to \cO_{R_j}(-1) \to \cO_R \to \cO_{R'} \to 0
\end{equation}
and
\begin{equation}
\label{eq:root-seq-2}
0 \to \cO_{R'}(-R_j) \to \cO_R \to \cO_{R_j} \to 0.
\end{equation}
These sequences are useful for inductive arguments about positive roots.

\begin{lemma}\label{lem:h0-roots}
If~$D \in \Pic(S)$ and~$R = \sum a_iR_i$ is a positive root in~$\fR(S,H)$ then
\begin{equation*}
\upchi(\cO_R(D)) = 1 + D \cdot R.
\end{equation*}
Moreover, if~$D \cdot R_i \le 0$ for all~$R_i$ with~$a_i > 0$ and~$D \cdot R < 0$ then~$h^0(\cO_R(D)) = 0$.
\end{lemma}

\begin{proof}
The formula for~$\upchi(\cO_R(D))$ follows immediately from the Riemann--Roch Theorem.

The second statement is clear if~$R$ is a simple root.
Otherwise, we consider exact sequences~\eqref{eq:root-seq-1} and~\eqref{eq:root-seq-2},
where~$R_j$ is a simple root such that~$R \cdot R_j = -1$ and~$R' = R- R_j$, so that~$R' \cdot R_j = 1$.
Since~$R'$ is a positive root, we have~$a_j > 0$, hence~$D \cdot R_j \le 0$ by the hypothesis of the lemma.

If~$D \cdot R' < 0$, twisting~\eqref{eq:root-seq-1} by~$D$ 
and noting that~$h^0(\cO_{R_j}(D \cdot R_j - 1)) = 0$ because~$D \cdot R_j \le 0$, 
we see that~$h^0(\cO_R(D)) \le h^0(\cO_{R'}(D))$, which is zero by induction, hence~$h^0(\cO_R(D)) = 0$.

On the other hand, the hypotheses of the lemma imply that~$D \cdot R' = \sum_i (a_i - \delta_{i,j}) D \cdot R_i \le 0$,
so if~$D \cdot R' \ge 0$, it follows that~$D \cdot R' = 0$ and~$D \cdot R_j < 0$, hence~$h^0(\cO_{R_j}(D)) = 0$.
It also follows that~$a_j = 1$; 
therefore, twisting~\eqref{eq:root-seq-2} by~$D$ we obtain~$h^0(\cO_R(D)) = h^0(\cO_{R'}(D - R_j))$,
and since~\mbox{$(D - R_j) \cdot R' = -R' \cdot R_j = -1$},
this is zero, again by induction.
\end{proof}

Using further the terminology of Weyl groups, 
we will say that a divisor class~$D$ on a quasipolarized K3 surface~$(S,H)$ is {\sf minuscule}, if
\begin{equation*}
D \cdot R \in \{-1,0,1\}
\end{equation*}
for any positive root~$R$ in the ADE configuration~$\fR(S,H)$.

If~$D$ is minuscule, by~\cite[Exercises 23, 24 to Ch.~VI.2]{Bourbaki} 
there is a sequence of minuscule divisor classes~$D_0 = D, D_1, \dots, D_n$ 
such that for each~$1 \le k \le n$ the difference~$R_{i_k} \coloneqq D_{k-1} - D_{k}$ 
is a simple root satisfying~$D_k \cdot R_{i_k} = - D_{k-1} \cdot R_{i_k} = 1$, 
so that there are exact sequences
\begin{equation}
\label{eq:minuscule-seq}
0 \to \cO_S(D_k) \to \cO_S(D_{k-1}) \to \cO_{R_{i_k}}(-1) \to 0,
\end{equation}
and~$D_n \cdot R \in \{0,1\}$ for all positive roots~$R$ in~$\fR(S,H)$;
in other words, $D_n$ is nef over the contraction~$\baS$ of~$\fR(S,H)$.
The sequence~$\{D_k\}$ as above is not unique, but the last element~$D_n$ is unique;
we will call it {\sf the dominant replacement} for~$D$ and denote~$D_+$.
Note that
\begin{equation}
\label{eq:dplus-d}
D_+^2 = D^2,
\qquad
H \cdot D_+ = H \cdot D,
\qquad\text{and}\qquad
\rH^\bullet(S, \cO_S(D_+)) \cong \rH^\bullet(S, \cO_S(D)).
\end{equation}
Moreover, if~$\sigma \colon S \to \baS$ is the contraction of~$\fR(S,H)$ then 
\begin{equation}
\label{eq:rsi-dplus}
\rR\sigma_*\cO_S(D) \cong \rR\sigma_*\cO_S(D_+).
\end{equation}
Indeed, both~\eqref{eq:dplus-d} and~\eqref{eq:rsi-dplus} follow
from the definition and sequences~\eqref{eq:minuscule-seq} by induction.

Recall that a coherent sheaf~$\cF$ on a scheme~$X$ is {\sf maximal Cohen--Macaulay} 
if~$\cExt^i(\cF, \cO_X) = 0$ for all~$i \ge 1$, see~\cite[Definition~4.2.1]{Buch};
in other words if the derived dual of~$\cF$ is a pure sheaf.
We refer to~\cite{Buch} for the basic properties of maximal Cohen--Macaulay sheaves.

\begin{lemma}
\label{lem:pf-gg-mcm}
Let~$(S,H)$ be a quasipolarized~$K3$ surface of genus~\mbox{$g = r \cdot s$}.
Let~$\sigma \colon S \to \bar{S}$ be the contraction of the ADE configuration~$\fR(S,H)$.
If~$D \in \Pic(S)$ is a minuscule divisor class on~$S$ 
then~$\rR^1\sigma_*\cO_S(D) = 0$ and
\begin{equation}
\label{eq:ssod-vee}
(\sigma_*\cO_S(D))^\vee \cong \sigma_*\cO_S(-D).
\end{equation}
In particular, $\sigma_*\cO_S(D)$ is a maximal Cohen--Macaulay sheaf on~$\bar{S}$.
If, moreover, the dominant replacement~$D_+$ is globally generated, then so is the sheaf~$\sigma_*\cO_S(D)$.
\end{lemma}

\begin{proof}
The dominant replacement~$D_+$ for~$D$ is nef over~$\baS$;
hence we have~$\rR^1\sigma_*\cO_S(D_+) = 0$ by the Kawamata--Viehweg vanishing theorem.
Applying~\eqref{eq:rsi-dplus}, we obtain~$\rR^1\sigma_*\cO_S(D) = 0$.
Furthermore, 
\begin{equation*}
\cRHom(\sigma_*\cO_{S}(D), \cO_{\bar{S}}) \cong
\rR\sigma_*\cRHom(\cO_{S}(D), \sigma^!\cO_{\bar{S}}) \cong
\rR\sigma_*\cRHom(\cO_{S}(D), \cO_{S}) \cong
\rR\sigma_*\cO_{S}(-D)
\end{equation*}
(where the first isomorphism is the Grothendieck duality and the second is crepancy of~$\sigma$). 
Clearly, the divisor~$-D$ is also minuscule, hence the right-hand side is a pure sheaf, 
hence~$\sigma_*\cO_{S}(D)$ is a maximal Cohen--Macaulay sheaf and~\eqref{eq:ssod-vee} holds.

If~$D$ is minuscule and~$D_+$ is globally generated
then the zero locus of a general pair of global sections of~$\cO_S(D_+)$ 
is a zero-dimensional subscheme of~$S$, disjoint from the union~$\cup R_i$ of all simple roots in~$\fR(S,H)$.
Therefore, the corresponding Koszul complex
\begin{equation*}
0 \to \cO_S(-D_+) \to \cO_S \oplus \cO_S \to \cO_S(D_+) \to 0
\end{equation*}
is exact in a neighborhood of~$\cup R_i$.
Pushing it forward to~$\bar{S}$ and taking into account that~$-D_+$ is also minuscule, and hence~$\rR^1\sigma_*\cO_S(-D_+) = 0$, 
we see that~$\sigma_*\cO_S(D_+)$ is globally generated in a neighborhood of~$\Sing(\bar{S})$.
On the other hand, over the complement of~$\Sing(\bar{S})$ the morphism~$\sigma$ is an isomorphism,
hence~$\sigma_*\cO_S(D_+)$ is globally generated away from~$\Sing(\bar{S})$, hence everywhere.
Since~$\sigma_*\cO_S(D) \cong \sigma_*\cO_S(D_+)$ by~\eqref{eq:rsi-dplus}, it is also globally generated.
\end{proof}

\subsection{Brill--Noether theory for K3 surfaces}
\label{ss:bn-k3}

The following definition was introduced in~\cite[Definition~3.8]{Muk:New} for polarized K3 surfaces; we generalize it to the case of quasipolarizations.

\begin{definition}
A quasipolarized K3 surface~$(S,H)$ is called {\sf Brill--Noether general} ({\sf BN-general}) if
\begin{equation*}\label{eq:bn-inequality}
h^0(\cO_S(D)) \cdot h^0(\cO_S(H - D)) < h^0(\cO_S(H)) = g + 1
\end{equation*}
for all~$D \not\in \{0,H\}$, where~$g$ is the genus of~$(S,H)$, see~\eqref{eq:genus-s}.
\end{definition}

\begin{remark} 
In general, a quasipolarization~$H$ may have base points,
but if~$(S,H)$ is BN-general, it is base point free.
Indeed, if $H$ is not not base point free,
then by~\cite[Corollary~3.2]{SD} the linear system $|H|$ has a fixed component $D$. 
In particular, $h^0(\cO_S(H-D)) = h^0(\cO_S(H)) = g+1$ and
\begin{equation*}
h^0(\cO_S(D)) \cdot h^0(\cO_S(H-D)) = 1 \cdot (g+1) = g+1,
\end{equation*}
contradicting the assumption that $(S, H)$ is BN-general.
\end{remark}

The following theorem relates Brill--Noether properties of curves and K3 surfaces.

\begin{theorem}
\label{thm:bnpc}
Let~$(S,H)$ be a quasipolarized K3 surface of genus~$g \ge 2$.
\begin{aenumerate}
\item 
\label{it:bnp-bn}
If~$|H|$ contains a smooth connected BN-general curve~$C \subset S$ then~$(S,H)$ is BN-general.
\item 
\label{it:general-bnp}
If every curve in~$|H|$ is reduced and irreducible then a general curve in~$|H|$ is BNP-general.
\end{aenumerate}
\end{theorem}

\begin{proof}
Part~\ref{it:bnp-bn} is easy.
Assume~$h^0(\cO_S(D)) \cdot h^0(\cO_S(C-D)) \ge g + 1$.
Using the exact sequence
\begin{equation*}
0 \to \cO_S(D - C) \to \cO_S(D) \to \cO_C(D\vert_C) \to 0
\end{equation*}
and effectivity and non-triviality of~$C - D$ we obtain~$h^0(\cO_C(D\vert_C)) \ge h^0(\cO_S(D))$.
Using the similar inequality for~$C - D$, we obtain
\begin{equation*}
h^0(\cO_C(D\vert_C)) \cdot h^0(\cO_C((C-D)\vert_C)) \ge g + 1
\end{equation*}
which contradicts the BN-generality of~$C$.

Part~\ref{it:general-bnp} is much harder; it is proved in~\cite{Laz}.
\end{proof}

\begin{remark}
Part~\ref{it:general-bnp} has been partially extended in~\cite[Theorem~1]{Hubarcak} 
to BN-general K3 surfaces of genus~$g \le 19$, 
showing that in these cases a smooth curve in~$|C|$ is BN-general.
\end{remark}

The following definition is analogous to Definition~\ref{def:mukai-pair} of a Mukai pair on a curve;
it plays an important role in the rest of the paper.

\begin{definition}\label{def:xi-big}
Let~$r, s \ge 2$.
{\sf A special Mukai class of type~$(r,s)$} on a quasipolarized smooth K3 surface~$(S,H)$ of genus~$g = r\cdot s$
is a globally generated class~$\Xi \in \Pic(S)$ such that
\begin{equation}
\label{eq:d2-dh-h0d}
\Xi \cdot H = (r - 1)(s + 1),
\qquad
h^0(\cO_S(\Xi)) = r,
\qquad\text{and}\qquad 
h^1(\cO_S(\Xi)) = h^2(\cO_S(\Xi)) = 0.
\end{equation}
\end{definition}

Recall from Section~\ref{ss:qp-k3} the definition of minuscule divisor classes.

\begin{lemma}
\label{lem:xi-minuscule}
If~$\Xi$ is a special Mukai class of type~$(r,s)$ 
on a BN-general quasipolarized~$K3$ surface~$(S,H)$ of genus~$g = r \cdot s$ then
\begin{equation*}
(H - \Xi) \cdot H = (r + 1)(s - 1),
\quad
h^0(\cO_S(H - \Xi)) = s,
\quad\text{and}\quad 
h^1(\cO_S(H - \Xi)) = h^2(\cO_S(H - \Xi)) = 0.
\end{equation*}
Moreover, $\Xi$ and~$H - \Xi$ are minuscule.
\end{lemma}

\begin{proof}
We have~$H \cdot (H - \Xi) = 2rs - 2 - (r - 1)(s + 1) = (r + 1)(s - 1) > 0$, hence~$h^0(\cO_S(\Xi - H)) = 0$
and the natural exact sequence
\begin{equation*}
0 \to \cO_S(\Xi - H) \to \cO_S(\Xi) \to \cO_C(\Xi\vert_C) \to 0
\end{equation*}
shows that~$h^0(\cO_C(\Xi\vert_C)) \ge h^0(\cO_S(\Xi)) = r$.
Since, on the other hand~$\deg(\cO_S(\Xi)\vert_C) = (r - 1)(s + 1)$,
it follows that~$h^1(\cO_C(\Xi\vert_C)) \ge s$, and then the above exact sequence implies that
\begin{equation*}
h^0(\cO_S(H - \Xi)) = h^2(\cO_S(\Xi - H)) = h^1(\cO_C(\Xi\vert_C)) \ge s.
\end{equation*}
By BN-generality of~$S$, this must be an equality, and then, using Riemann--Roch 
and taking the vanishing~$h^2(\cO_S(H - \Xi)) = h^0(\cO_S(\Xi - H))$ into account,
we deduce the first part of the lemma.

For the second part, since~$(H - \Xi) \cdot R = - \Xi \cdot R$, it is enough to show that~$\Xi$ is minuscule.
So, let~$R$ be a positive root.
We have~$\Xi \cdot R \ge 0$ because~$\Xi$ is globally generated, hence nef, and~$R$ is effective, and it remains to show that~$\Xi \cdot R \le 1$.
So, assume~$\Xi \cdot R \ge 2$.
Then~$(\Xi + R) \cdot R \ge 0$, hence~$\upchi(\cO_R(\Xi + R)) > 0$ by Lemma~\ref{lem:h0-roots}, hence~$h^0(\cO_R(\Xi + R)) > 0$, and the exact sequence
\begin{equation*}
0 \to \cO_S(\Xi) \to \cO_S(\Xi + R) \to \cO_R(\Xi + R) \to 0
\end{equation*}
implies that~$h^0(\cO_S(\Xi + R)) > h^0(\cO_S(\Xi)) = r$.

On the other hand, $(H - \Xi) \cdot R_i = - \Xi \cdot R_i \le 0$ for any simple root~$R_i$, because~$\Xi$ is nef, 
and~$(H - \Xi) \cdot R = - \Xi \cdot R\le -2$, hence we have~$h^0(\cO_R(H - \Xi)) = 0$, again by Lemma~\ref{lem:h0-roots}. 
Now, the exact sequence
\begin{equation*}
0 \to \cO_S(H - \Xi - R) \to \cO_S(H - \Xi) \to \cO_R(H - \Xi) \to 0
\end{equation*}
implies that~$h^0(\cO_S(H - \Xi - R)) = h^0(\cO_S(H - \Xi)) = s$.
We conclude that
\begin{equation*}
h^0(\cO_S(\Xi + R)) \cdot h^0(\cO_S(H - \Xi - R)) > rs = g, 
\end{equation*}
contradicting  Brill--Noether generality of~$(S,H)$.
Thus, $\Xi$ is minuscule.
\end{proof}

The next result gives a precise relation between special Mukai classes and Mukai pairs.
Recall from Section~\ref{ss:qp-k3} the definition of the dominant replacement~$D_+$ of a minuscule divisor class~$D$.

\begin{proposition}
\label{prop:h-xi}
Let~$\Xi$ be a special Mukai class of type~$(r,s)$ 
on a BN-general quasipolarized $K3$ surface~$(S,H)$ of genus~$g  =r \cdot s \ge 4$.
If~$(S,H)$ contains a BNP-general curve~$C \in |H|$ then~$(H - \Xi)_+$ 
is a special Mukai class of type~$(s,r)$ and~$(\cO_S(\Xi)\vert_C, \cO_S(H - \Xi)\vert_C)$ is a Mukai pair.
\end{proposition}

\begin{proof}
Using Lemma~\ref{lem:xi-minuscule} and~\eqref{eq:dplus-d}, 
we see that~$(H - \Xi)_+$ satisfies the conditions~\eqref{eq:d2-dh-h0d} 
(where the role of~$r$ and~$s$ is swapped);
in particular, Riemann--Roch implies that~$(H - \Xi)_+^2 = 2s - 4$.
So it remains to show that it is globally generated.

Assume to the contrary that~$(H - \Xi)_+$ is not globally generated.
We saw in the proof of Lemma~\ref{lem:xi-minuscule} that~$h^0(\cO_C(\Xi\vert_C)) \ge r$.
Since~$C$ is BN-general, $(\cO_S(\Xi)\vert_C, \cO_S(H - \Xi)\vert_C)$ is a Mukai pair by Lemma~\ref{lem:xi-eta},
and in particular~$\cO_C((H - \Xi)_+\vert_C) \cong \cO_C((H - \Xi)\vert_C)$ is globally generated.
Since~$h^1(\cO_S(-\Xi)) = 0$, 
the restriction morphism~$\rH^0(S, \cO_S((H - \Xi))) \to \rH^0(C, \cO_C((H - \Xi)\vert_C))$ is surjective,
and then~\eqref{eq:dplus-d} implies that~$\rH^0(S, \cO_S((H - \Xi)_{+})) \to \rH^0(C, \cO_C((H - \Xi)_{+}\vert_C))$ 
is surjective as well,
hence~$(H - \Xi)_+$ is globally generated in a neighborhood of~$C$.
Since~$C$ intersects any curve on~$S$ except for those in~$\fR(S,H)$
and since the base locus of a divisor on a smooth K3 surface is a Cartier divisor (see~\cite[Corollary~3.2]{SD}),
we conclude that the base locus of~$(H - \Xi)_+$ is contained in the ADE configuration~$\fR(S,H)$.
Thus,
\begin{equation*}
(H - \Xi)_+ \cdot \Gamma \ge 0
\end{equation*}
for any irreducible curve not contained in~$\fR(S,H)$.
But the same inequality also holds for any~$\Gamma$ contained in~$\fR(S,H)$ by the definition of dominant replacement.
Therefore, $(H - \Xi)_+$ is nef.
Since it is not globally generated by our assumption, 
\cite[Theorem~3.8(b)]{Reid93} implies that~$(H - \Xi)_+^2 > 0$, 
then~\cite[Theorem~3.8(c)]{Reid93} shows that~$(H - \Xi)_+$ is big,
and~\cite[Theorem~3.8(d)]{Reid93} proves that
\begin{equation*}
(H - \Xi)_+ = aE + R_i,
\end{equation*}
where~$R_i$ is a simple root, $|E|$ is an elliptic pencil, $E \cdot R_i = 1$, and~$a \ge 2$;
moreover, computing the square of both sides we find that~$a = s - 1$, in particular~$s \ge 3$.
Now, let~$\zeta \coloneqq \cO_S(E)\vert_C$. 
Then~$\cO_S((H - \Xi)_+\vert_C) \cong \zeta^{(s-1)}$, and~$\zeta$ is a line bundle with~$h^0(\zeta) \ge 2$ of degree
\begin{equation*}
\deg(\zeta) = \tfrac1{s-1} (H - \Xi)_+ \cdot H = \tfrac1{s-1}(r + 1)(s - 1) = r + 1.
\end{equation*}
Then~$\upchi(\zeta) = r + 2 - rs$, hence~$h^1(\zeta) \ge r(s - 1)$,
hence~$h^0(\zeta)h^1(\zeta) \ge 2r(s-1)$; since $s \ge 3$, this is larger than the genus~$g = rs$ of~$C$.
Thus, we obtain a contradiction with BN generality of~$C$.
\end{proof}

\begin{corollary}
\label{cor:bs-h-xi}
Let~$\Xi$ be a special Mukai class of type~$(r,s)$ 
on a BN-general quasipolarized~$K3$ surface~$(S,H)$ of genus~\mbox{$g = r \cdot s \ge 4$}.
If~$\sigma \colon S \to \bar{S}$ is the contraction of~$\fR(S,H)$ then
\begin{equation*}
(\sigma_*\cO_S(\Xi))^\vee \cong \sigma_*\cO_S(-\Xi).
\end{equation*}
Moreover, $\sigma_*\cO_S(\Xi)$ and~$\sigma_*\cO_S(H - \Xi)$ are globally generated maximal Cohen--Macaulay sheaves on~$\bar{S}$
with~$h^0(\sigma_*\cO_S(\Xi)) = r$, $h^0(\sigma_*\cO_S(H - \Xi)) = s$, and no higher cohomology.
\end{corollary}

\begin{proof}
This is a combination of Lemma~\ref{lem:pf-gg-mcm} and~\eqref{eq:dplus-d} with Lemma~\ref{lem:xi-minuscule} and Proposition~\ref{prop:h-xi}.
\end{proof}

%%%%%%%%%%%%%%%%%%%%%%%%%

\section{Mukai bundles on K3 surfaces}\label{sec:mbs}

The goal of this section is to prove Theorem~\ref{thm:mb-s}, a criterion for stability 
of Lazarsfeld bundles on K3 surfaces in terms of the existence of special Mukai classes. 

Recall from Section~\ref{sec:notation} that, for a sheaf~$\cF$ of finite projective dimension on a K3 surface with du Val singularities,
we write~$\rv(\cF)$ for its {\sf Mukai vector}. 
Also recall that the Euler bilinear form~$\upchi(\cF,\cF)$ is even, and that
a sheaf~$\cF$ is called
\begin{itemize}
\item 
{\sf simple}, if~$\Hom(\cF, \cF) = \kk$, hence~$\Ext^2(\cF, \cF) = \kk$,
\item 
{\sf rigid}, if~$\Ext^1(\cF,\cF) = 0$, and
\item 
{\sf spherical}, if it is simple and rigid, i.e., $\Ext^\bullet(\cF, \cF) = \kk \oplus \kk[-2]$.
\end{itemize}
In particular, if~$\cF$ is rigid then~$\upchi(\cF, \cF) \ge 2$, and a rigid sheaf is spherical if and only if~$\upchi(\cF, \cF) < 4$.
Similarly, if~$\cF$ is simple then~$\upchi(\cF, \cF) \le 2$, and a simple sheaf is spherical if and only if~$\upchi(\cF, \cF) > 0$.

\begin{definition}
\label{def:mb-s}
Let~$(\baS,\baH)$ be a polarized K3 surface of genus~$g = r \cdot s$ with du Val singularities.
Let~\mbox{$\sigma \colon S \to \baS$} be the minimal resolution
and let~$H = \sigma^*(\baH)$ be the induced quasipolarization.
\begin{enumerate}[wide, label={\textup{(\roman*)}}]
\item 
A vector bundle~$\bcU$ on~$(\baS,\baH)$ with Mukai vector
\begin{equation}
\label{eq:mv-bcus}
\rv(\bcU) = (r, -\baH, s),
\end{equation}
is called a {\sf Mukai bundle of type~$(r,s)$} if it is $\baH$-Gieseker stable.
\item 
A vector bundle~$\cU$ on ~$(S,H)$ is called a {\sf Mukai bundle of type~$(r,s)$} if~$\cU \cong \sigma^*(\bcU)$, 
where~$\bcU$ is a Mukai bundle on~$(\baS,\baH)$ of type~$(r,s)$.
\end{enumerate}
\end{definition}

Note that~\eqref{eq:mv-bcus} is equivalent to the equalities~$\rank(\bcU) = r$, $\rc_1(\bcU) = -\baH$, and~$\upchi(\bcU) = r + s$,
and implies~$\upchi(\bcU, \bcU) = 2rs - \baH^2 = 2$.
Moreover, if~$\cU \cong \sigma^*(\bcU)$ is a Mukai bundle on~$(S,H)$ then
\begin{equation*}
\label{eq:mv-cus}
\rv(\cU) = (r, -H, s)
\end{equation*} 
and~$\bcU \cong \sigma_*(\cU)$ by the projection formula.

A standard argument proves the uniqueness of a Mukai bundle.

\begin{lemma}
\label{lem:mbs-unique}
If a Mukai bundle of type~$(r,s)$ on~$(S,H)$ or~$(\baS,\baH)$ exists, it is unique.
\end{lemma}

\begin{proof}
By our definitions, it is enough to prove uniqueness of the latter, which is standard.
Indeed, if~$\bcU_1$ and~$\bcU_2$ are two Mukai bundles on~$(\baS,\baH)$ of type~$(r,s)$, we have
\begin{equation*}
\upchi(\bcU_1,\bcU_2) = -\langle \rv(\bcU_1), \rv(\bcU_2) \rangle = - \langle \rv(\bcU_1), \rv(\bcU_1)\rangle = \upchi(\bcU_1,\bcU_1) = 2,
\end{equation*}
hence either~$\Hom(\bcU_1,\bcU_2) \ne 0$ or~$\Hom(\bcU_2,\bcU_1) \cong \Ext^2(\bcU_1,\bcU_2)^\vee \ne 0$. 
So we may assume that there is a morphism~$\bcU_1 \to \bcU_2$,
and since~$\bcU_1$ and~$\bcU_2$ are $\baH$-stable, it is an isomorphism.
\end{proof}

If~$\bcU$ is a Mukai bundle of type~$(r,s)$ and the dual Mukai bundle~$\bcU^\vee$ 
is globally generated and satisfies~$\rH^{> 0}(\baS,\bcU^\vee) = 0$ (which is often the case) then the bundle
\begin{equation*}\label{eq:def-uperp}
\bcU^\perp \coloneqq \Ker(\rH^0(\baS, \bcU^\vee) \otimes \cO_{\baS} \twoheadrightarrow \bcU^\vee)
\end{equation*}
has Mukai vector~$(s, -\baH,r)$, so if it is $\baH$-stable (which is also often the case), it is a Mukai bundle of type~$(s,r)$.
Thus, the definition is almost symmetric.
Moreover, the Mukai bundle of type~$(1,g)$ is isomorphic to~$\cO_{\baS}(-\baH)$, so from now on we assume
\begin{equation}
\label{eq:rs-assumption}
s \ge r \ge 2.
\end{equation} 

\subsection{The Lazarsfeld bundle}
\label{ss:lazarseld}

In this section, we explain a construction, due to Lazarsfeld, 
of a vector bundle on~$\baS$ with Mukai vector~\eqref{eq:mv-bcus} and state a criterion for stability of this bundle.
This will give us a construction of a Mukai bundle.

Note that if~$(\baS,\baH)$ is a polarized du Val K3 surface,
a smooth curve~$C \subset \baS$ in~$|\baH|$ is contained in the smooth locus of~$\baS$ (because it is a Cartier divisor).
Similarly, if~$(S,H)$ is a quasipolarized smooth K3 surface, 
a smooth curve~$C \subset S$ in~$|H|$ does not intersect the simple roots~$R_i$.
Thus, if~$\sigma \colon S \to \baS$ is the contraction of~$\fR(S,H)$ and~$H = \sigma^*(\baH)$
we have a bijection between smooth curves on~$\baS$ in~$|\baH|$ and smooth curves on~$S$ in~$|H|$.
Given such a curve~$C$ we will denote its embedding into both~$S$ and~$\baS$ by~$j$, 
hoping that this will not cause any confusion.

Now assume that~$C$ is a smooth BNP-general curve.
We fix a factorization~$g = r \cdot s$, where~$r$ and~$s$ are as in~\eqref{eq:rs-assumption}.
By Lemma~\ref{lem:xi-eta} the curve~$C$ has a Mukai pair~$(\xi,\eta)$ of type~$(r,s)$.
If~$j \colon C \to S$ is an embedding into a smooth or du Val K3 surface,
we define, following~\cite{Laz} the {\sf Lazarsfeld bundle}~$\bL_S(\xi)$ by the exact sequence
\begin{equation}
\label{eq:def-cus}
0 \to \bL_{S}(\xi) \xrightarrow\quad \rH^0(C,\xi) \otimes \cO_{S} \xrightarrow{\ \ev\ } j_*\xi \to 0,
\end{equation}
where~$\ev$ is the evaluation morphism (note that~$\xi$ is globally generated by Definition~\ref{def:mukai-pair}).
Note that~$\bL_{S}(\xi) \cong \sigma^*\bL_{\baS}(\xi)$.

The most important properties of~$\bL_{\baS}(\xi)$ are summarized below.

\begin{lemma}
\label{lem:lazarsfeld-bundle}
Let~$(S,H)$ be a smooth quasipolarized or du Val polarized $K3$ surface of genus~$g = r \cdot s$
and let~$C$ be a smooth BNP-general curve in~$|H|$.
If~$(\xi,\eta)$ is a Mukai pair of type~$(r,s)$ on~$C$ 
then the sheaf~$\bL_{S}(\xi)$ defined by~\eqref{eq:def-cus} is locally free, spherical, 
has Mukai vector $(r, -H, s)$ and satisfies
\begin{equation}
\label{eq:cus-h}
h^0(\bL_{S}(\xi)) = h^1(\bL_{S}(\xi)) = 0
\qquad\text{and}\qquad 
h^2(\bL_{S}(\xi)) = r + s.
\end{equation} 
The dual Lazarsfeld bundle~$\bL_{S}(\xi)^\vee$ fits into an exact sequence
\begin{equation}
\label{eq:cus-dual}
0 \to \rH^0(C,\xi)^\vee \otimes \cO_{S} \to \bL_{S}(\xi)^\vee \to j_*\eta \to 0.
\end{equation}
It is globally generated and satisfies 
\begin{equation}
\label{eq:cus-h-dual}
h^0(\bL_{S}(\xi)^\vee) = r + s
\qquad\text{and}\qquad 
h^1(\bL_{S}(\xi)^\vee) = h^2(\bL_{S}(\xi)^\vee) = 0.
\end{equation} 
\end{lemma}

\begin{proof}
The sheaf~$\bL_{S}(\xi)$ is locally free because~$j_*\xi$ has projective dimension~$1$ and the evaluation morphism is surjective.
The computation of the Mukai vector of~$\bL_{S}(\xi)$ is obvious.

The equalities~\eqref{eq:cus-h} follow from the cohomology exact sequence of~\eqref{eq:def-cus} and~\eqref{eq:h0h1-xi}.
Applying Serre duality we deduce~\eqref{eq:cus-h-dual}.
The sequence~\eqref{eq:cus-dual} follows from~\eqref{eq:def-cus} by dualizing, 
and the global generation of~$\bL_{S}(\xi)^\vee$ follows from $\rH^1(S, \cO_S) = 0$ and the global generation of~$\cO_{S}$ and~$\eta$.

To compute~$\Ext^\bullet(\bL_{S}(\xi), \bL_{S}(\xi))$ and check that~$\bL_{S}(\xi)$ is spherical, we use the defining short exact sequence~\eqref{eq:def-cus}.
Taking into account the Grothendieck duality isomorphism
\begin{equation*}
\Ext^p(j_*\xi, \cO_{S}) \cong 
\Ext^p(\xi, j^!\cO_{S}) \cong 
\Ext^p(\xi, \cO_C(K_C)[-1]) \cong 
\Ext^{p-1}(\xi, \cO_C(K_C)),
\end{equation*}
we obtain a self-dual spectral sequence with the~$\bE_1$ page
\begin{equation*}
\xymatrix@R=0ex{
\rH^0(\xi) \otimes \Ext^1(\xi, \cO_C(K_C)) \ar[r] &
\rH^0(\xi) \otimes \rH^0(\xi)^\vee \oplus \Ext^2(j_*\xi, j_*\xi) &
0
\\
\rH^0(\xi) \otimes \Hom(\xi, \cO_C(K_C)) \ar[r] &
\Ext^1(j_*\xi, j_*\xi) \ar[r] &
\rH^1(\xi) \otimes \rH^0(\xi)^\vee
\\
0 &
\rH^0(\xi) \otimes \rH^0(\xi)^\vee \oplus \Hom(j_*\xi, j_*\xi) \ar[r] &
\rH^0(\xi) \otimes \rH^0(\xi)^\vee
}
\end{equation*}
that converges to~$\Ext^\bullet(\bL_{S}(\xi), \bL_{S}(\xi))$.
Obviously, the arrow of the bottom row induces an isomorphism~$\rH^0(\xi) \otimes \rH^0(\xi)^\vee \xrightiso{} \rH^0(\xi) \otimes \rH^0(\xi)^\vee$.
Since~$\Hom(j_*\xi, j_*\xi) \cong \Hom(\xi, \xi) \cong \kk$, to check that~$\bL_{S}(\xi)$ is simple, 
it is enough to show that the first arrow in the middle row is injective.
For this note that we have a natural self-dual exact sequence
\begin{equation*}
0 \to \Ext^1(\xi, \xi) \to \Ext^1(j_*\xi, j_*\xi) \to \Hom(\xi, \xi(K_C)) \to 0
\end{equation*}
and the composition 
\begin{equation*}
\rH^0(\xi) \otimes \Hom(\xi, \cO_C(K_C)) \to \Ext^1(j_*\xi, j_*\xi)  \to \Hom(\xi, \xi(K_C))
\end{equation*}
of the first arrow in the middle row of the spectral sequence with the second arrow in the exact sequence above 
is the Petri map of~$\xi$, hence it is an isomorphism by Lemma~\ref{lem:xi-eta}.
Therefore, the first arrow in the middle row of the spectral sequence is injective, hence~$\Hom(\bL_{S}(\xi), \bL_{S}(\xi)) = \kk$.
Finally, by~\eqref{eq:rr-k3} we have~$\upchi(\bL_{S}(\xi), \bL_{S}(\xi)) = 2rs - H^2 = 2$, hence~$\bL_{S}(\xi)$ is spherical.
\end{proof}

Lemma~\ref{lem:lazarsfeld-bundle} shows that~$\bL_{\baS}(\xi)$ is a Mukai bundle if and only if it is $\baH$-Gieseker stable.
The main result of this section is a criterion for stability of~$\bL_{\baS}(\xi)$ 
in terms of special Mukai classes on the minimal resolution~$S$ of~$\bar{S}$
defined in~\S\ref{ss:bn-k3} (see Definition~\ref{def:xi-big}).

\begin{theorem}\label{thm:mb-s}
Let~$(\baS,\baH)$ be a polarized du Val~$K3$ surface of genus~$g = r \cdot s$ with
\begin{equation*}
r \in \{2,3\} 
\qquad\text{and}\qquad 
s \ge r.
\end{equation*}
Let~$\sigma \colon S \to \baS$ be the minimal resolution and let~$H \coloneqq \sigma^*(\baH)$.
Assume~$|\baH|$ contains a \textup{(}smooth\textup{)} BNP-general curve~$C \subset \baS$.
Then the following conditions are equivalent:
\begin{aenumerate}
\item
\label{it:lazarseld-stable}
The Lazarsfeld bundle~$\bL_{\baS}(\xi)$ associated with 
a Mukai pair~$(\xi,\eta)$ of type~$(r,s)$ on~$C$ is not $\baH$-Gieseker stable.
\item
\label{it:extremal-class}
The surface~$(S,H)$ has a special Mukai class~$\Xi$ of type~$(r,s)$ such that~$\cO_S(\Xi)\vert_C \cong \xi$.
\end{aenumerate}
Moreover, if~$(S,H)$ does not have special Mukai classes of type~$(r,s)$ 
then~$\bL_{\baS}(\xi)$ does not depend on the choice of the curve~$C$ and Mukai pair~$(\xi,\eta)$ 
and it is a Mukai bundle of type~$(r,s)$ on~$\baS$.
\end{theorem}

We will prove this theorem in~\S\ref{ss:proof-surface} after some preparation.

\begin{remark}
We expect that Theorem~\ref{thm:mb-s} holds for any~$r,s \ge 2$.
\end{remark}

\subsection{Multispherical filtration}
\label{ss:srs-sheaves}

In this and the next subsections we work on a smooth quasipolarized K3 surface~$S$.
We start with a few well-known results.

\begin{lemma}
\label{lem:hom-back}
Let~$0 \to \cF_1 \to \cF \to \cF_2 \to 0$ be an exact sequence.
If~$\cF$ is a simple sheaf
then~$\Hom(\cF_2,\cF_1) = 0$.
\end{lemma}

\begin{lemma}
\label{lem:chern0}
Let~$S$ be a smooth projective surface.
\begin{aenumerate}
\item
\label{it:mono}
If~$\cF \hookrightarrow \cO_S^{\oplus n}$ is a monomorphism from a sheaf~$\cF$ with~$\rc_1(\cF) \ge 0$ 
then~$\cF^{\vee\vee} \cong \cO_S^{\oplus m}$.
\item
\label{it:epi}
If~$\cO_S^{\oplus n} \twoheadrightarrow \cF$ is an epimorphism onto a torsion free sheaf~$\cF$ with~$\rc_1(\cF) \le 0$ then~$\cF \cong \cO_S^{\oplus m}$.
\end{aenumerate}
\end{lemma}

\begin{proof}
\ref{it:mono}
Set~$m \coloneqq \rank(\cF)$. 
The embedding~$\cF \hookrightarrow \cO_S^{\oplus n}$ induces an embedding~$\cF^{\vee\vee} \hookrightarrow \cO_S^{\oplus n}$.
Composing it with a sufficiently general projection~\mbox{$\cO_S^{\oplus n} \twoheadrightarrow \cO_S^{\oplus m}$}
we obtain an injective morphism~$\varphi \colon \cF^{\vee\vee} \hookrightarrow \cO_S^{\oplus m}$ of locally free sheaves of rank~$m$.
Its determinant~$\det(\varphi) \colon \det(\cF^{\vee\vee}) \to \cO_S$ is then also injective, 
and since~$\rc_1(\cF^{\vee\vee}) = \rc_1(\cF) \ge 0$, 
it follows that~$\det(\varphi)$ is an isomorphism, hence~$\varphi$ is an isomorphism as well.

\ref{it:epi}
Dualizing the epimorphism, we obtain a monomorphism~$\cF^\vee \hookrightarrow \cO_S^{\oplus n}$.
Since~$\cF$ is torsion free, we have~$\rc_1(\cF^\vee) = - \rc_1(\cF) \ge 0$.
Therefore, \ref{it:mono} implies that~$\cF^\vee \cong \cO_S^{\oplus m}$, hence~$\cF^{\vee\vee} \cong \cO_S^{\oplus m}$.
The composition
\begin{equation*}
\cO_S^{\oplus n} \twoheadrightarrow \cF \to \cF^{\vee\vee} \cong \cO_S^{\oplus m} 
\end{equation*}
is generically surjective, hence it is surjective everywhere.
Finally, the middle map is injective because~$\cF$ is torsion free, 
hence it is an isomorphism, and therefore~$\cF \cong \cO_S^{\oplus m}$.
\end{proof}

\begin{corollary}
\label{cor:gg-quotient}
Let~$\cF$ be a globally generated sheaf on a smooth $K3$ surface.
If~$\cF \to \cF'$ is an epimorphism of vector bundles, the line bundle~$\det(\cF')$ is globally generated 
and~$h^2(\det(\cF')) = 0$ unless~$\cF' \cong \cO^{\oplus m}$ and the surjection~$\cF \twoheadrightarrow \cF'$ splits.
\end{corollary}

\begin{proof}
By assumptions we have epimorphisms~$\cO^{\oplus n} \twoheadrightarrow \cF \twoheadrightarrow \cF'$.
If~$\rank(\cF') = m$, taking the $m$-th wedge power of the composition, we conclude that~$\det(\cF')$ is globally generated.

Assume~$h^2(\det(\cF')) \ne 0$.
By Serre duality there is a generically injective morphism~\mbox{$\det(\cF') \to \cO$}, hence~$\rc_1(\cF') \le 0$.
Applying Lemma~\ref{lem:chern0}\ref{it:epi} we conclude that~$\cF' \cong \cO^{\oplus m}$.
Finally, since the composition~$\cO^{\oplus n} \twoheadrightarrow \cF \twoheadrightarrow \cF' \cong \cO^{\oplus m}$ is surjective, 
there is a splitting~$\cO^{\oplus m} \hookrightarrow \cO^{\oplus n}$ of the composition, 
which induces a splitting~$\cF' \hookrightarrow \cF$ of the epimorphism~$\cF \twoheadrightarrow \cF'$.
\end{proof}

We will also need the following three fundamental results of Mukai.

\begin{proposition}[{\cite[Proposition~3.14]{Mukai:vectorbundles}}]
\label{prop:Picardrankonestable}
Let~$S$ be a smooth $K3$ surface of Picard rank one. 
Then every spherical bundle on~$S$ is slope-stable, hence also Gieseker-stable.
\end{proposition}

\begin{proposition}[{\cite[Proposition~3.3]{Mukai:vectorbundles}}]\label{prop:rigid-lf}
Any rigid torsion free sheaf on a smooth $K3$ surface is locally free.
\end{proposition}

\begin{lemma}[{\cite[Proposition~2.7 and Corollary~2.8]{Mukai:vectorbundles}}]\label{lem:mukai-lemma}
Let~$\cF$ be a rigid sheaf on a smooth~$K3$ surface.
If
\begin{equation*}
0 \to \cF_1 \to \cF \to \cF_2 \to 0
\end{equation*}
is an exact sequence and~$\Hom(\cF_1,\cF_2) = 0$ then~$\cF_1$ and~$\cF_2$ are both rigid.

Moreover, if~$\eps \in \Ext^1(\cF_2,\cF_1)$ is the extension class of the sequence then
\begin{aenumerate}
\item 
if~$\cF_1$ is simple then the map~$\Hom(\cF_2,\cF_2) \xrightarrow{\ \eps\ } \Ext(\cF_2,\cF_1)$ is surjective;
\item 
if~$\cF_2$ is simple then the map~$\Hom(\cF_1,\cF_1) \xrightarrow{\ \eps\ } \Ext(\cF_2,\cF_1)$ is surjective.
\end{aenumerate}
\end{lemma}

The main goal of this subsection is to construct a filtration of any rigid sheaf on a smooth quasipolarized K3 surface~$(S,H)$ into stable spherical sheaves. 
Since Gieseker-stability for a quasi-polarization is not well--defined, we use a notion of stability 
that is equivalent to Gieseker stability with respect to a small ample perturbation~$H + \epsilon A$ of the quasipolarization~$H$.
By Riemann--Roch the reduced Hilbert polynomial~$\bp_{H + \epsilon A}(\cF)(t)$ of a sheaf~$\cF$ 
with respect to polarization~$H + \epsilon A$ can be expressed 
in terms of the slope and the reduced Euler characteristic of~$\cF$ defined in~\eqref{eq:Paris20231230} as
\begin{equation*}
\bp_{H + \epsilon A}(\cF)(t) = \tfrac12 (H + \epsilon A)^2 t^2 + \upmu_H(\cF) t + \epsilon\,\upmu_A(\cF) t + \updelta(\cF).
\end{equation*}
Consequently, $(H + \epsilon A)$-stability for~$0 < \epsilon \ll 1$ can be rephrased in the following terms.

\begin{definition}
\label{def:ha-stability}
Let~$S$ be a smooth K3 surface, let~$H$ be a quasipolarization, 
let~$\sigma \colon S \to \baS$ be the contraction of~$\fR(S,H)$,
and let~$A$ be a $\sigma$-ample divisor class on~$S$.

Given two sheaves~$\cF_1, \cF_2$ on~$S$ of positive rank we write~$\bp_{H, A}(\cF_1) \prec \bp_{H, A}(\cF_2)$ if 
\begin{itemize}
\item 
$\upmu_H(\cF_1) < \upmu_H(\cF_2)$, or
\item 
$\upmu_H(\cF_1) = \upmu_H(\cF_2)$ and $\upmu_A(\cF_1) < \upmu_A(\cF_2)$, or
\item 
$\upmu_H(\cF_1) = \upmu_H(\cF_2)$ and $\upmu_A(\cF_1) = \upmu_A(\cF_2)$ 
and $\updelta(\cF_1) < \updelta(\cF_2)$.
\end{itemize}
Similarly, we write~$\bp_{H, A}(\cF_1) = \bp_{H, A}(\cF_2)$ if
\begin{equation*}
\upmu_H(\cF_1) = \upmu_H(\cF_2),
\qquad
\upmu_A(\cF_1) = \upmu_A(\cF_2),
\qquad\text{and}\qquad 
\updelta(\cF_1) = \updelta(\cF_2).
\end{equation*}
We say that a torsion free sheaf~$\cF$ is {\sf $(H, A)$-(semi)stable} 
if~$\bp_{H, A}(\cF') \prec (\preceq)\ \bp_{H, A}(\cF)$ for every subsheaf~$\cF' \subset \cF$.
\end{definition}

\begin{remark} 
\label{rem:saturatedness}
If~$\cF_1 \subset \cF_2$ is a subsheaf of positive rank such that~$\cF_2/\cF_1$ is torsion and nonzero, 
then~$\bp_{H, A}(\cF_1) \prec \bp_{H, A}(\cF_2)$. Indeed, 
we have~$H \cdot \rc_1(\cF_2/\cF_1) \ge 0$ and $A \cdot \rc_1(\cF_2/\cF_1) \ge 0$, 
and if both~$H \cdot \rc_1(\cF_2/\cF_1) = 0$ and $A\cdot \rc_1(\cF_2/\cF_1) = 0$ 
then~$\cF_2/\cF_1$ has $0$-dimensional support, and therefore~$\upchi(\cF_2/\cF_1) > 0$.
Thus, in the definition of $(H,A)$-(semi)stability it is enough to assume that the respective inequality holds 
only for saturated subsheaves of~$\cF$, i.e., subsheaves~$\cF'$ such that~$\cF/\cF'$ is torsion free.
\end{remark}

The following proposition is the main result of this section.

\begin{proposition}
\label{prop:filtration-msph}
Let~$S$ be a smooth $K3$ surface, let~$H$ be a quasipolarization, 
let~$\sigma \colon S \to \baS$ be the contraction of~$\fR(S,H)$,
and let~$A$ be a $\sigma$-ample divisor class on~$S$.
Every rigid torsion free sheaf~$\cF$ on~$S$ has a filtration such that its factors~$\cF_1,\dots,\cF_n$ have the following properties:
\begin{aenumerate}
\item\label{it:factors-msph}
for all~$i$ we have~$\cF_i \cong \cG_i^{\oplus m_i}$, where~$\cG_i$ is an $(H, A)$-stable spherical sheaf, and
\item\label{it:factors-hom-zero}
for all~$ i \le j$ we have~$\upmu_H(\cF_i) \ge \upmu_H(\cF_{j})$
and~$\Hom(\cF_i, \cF_{i+1}) = 0$.
\end{aenumerate}
Moreover, $\sum m_i = 1$ if and only if~$\cF$ is $(H, A)$-stable.
\end{proposition}

We will call a filtration as in Proposition~\ref{prop:filtration-msph} {\sf a multispherical filtration} of~$\cF$. 
Note that we do not claim~$\Hom(\cF_i, \cF_j) = 0$ for~$j > i+1$; 
indeed, a priori the same stable factor could appear multiple times in the filtration.

\begin{proof}
We first treat the case where the sheaf~$\cF$ is $(H, A)$-semistable. 
If $\cF$ is $(H, A)$-stable, there is nothing to prove, 
because~$\cF$ is rigid (by assumption) and simple (by stability), hence spherical.
Otherwise, let~$\cG \subsetneq \cF$ be an $(H, A)$-stable subsheaf such that~$\bp_{H, A}(\cG) = \bp_{H, A}(\cF)$.

Now among all short exact sequences
\begin{equation*}
0 \to \cF' \to  \cF \to \cF'' \to 0 
\end{equation*}
with the property that~$\cF'$ admits a filtration with all associated factors isomorphic to the sheaf~$\cG$
(hence~$\bp_{H,A}(\cF') = \bp_{H,A}(\cG) = \bp_{H,A}(\cF)$), 
we consider one where~$\cF'$ has maximal possible rank.

Note that~$\cF''$ is torsion free: 
otherwise, by Remark~\ref{rem:saturatedness} the saturation~$\tcF'$ in~$\cF$ of the subsheaf~$\cF'$ 
has~$\bp_{H, A}(\tcF') \succ \bp_{H, A}(\cF') = \bp_{H, A}(\cF)$ and destabilizes~$\cF$. 
It is also easy to see that $\cF''$ is $(H, A)$-semistable:
otherwise the preimage in~$\cF$ of a destabilizing subsheaf of~$\cF''$ destabilizes~$\cF$.
We claim that~$\Hom(\cG, \cF'') = 0$. 
Indeed, as~$\cG$ is $(H, A)$-stable with~$\bp_{H, A}(\cG) = \bp_{H, A}(\cF'')$ 
any morphism~\mbox{$\cG \to \cF''$} would be injective, 
hence the preimage of~$\cG$ under the surjection~$\cF \onto \cF''$ 
would be an extension of~$\cG$ by~$\cF'$ of bigger rank than~$\cF'$, 
in contradiction to our assumption on~$\cF'$.

It follows that~$\Hom(\cF', \cF'') = 0$, and therefore by Lemma~\ref{lem:mukai-lemma} both~$\cF'$ and~$\cF''$ are rigid. 
This implies that~$\upchi(\cF', \cF') > 0$. 
Since~$\cF'$ admits a filtration with factors isomorphic to~$\cG$,
its Mukai vector~$\rv(\cF')$ is proportional to~$\rv(\cG)$, hence we also have~$\upchi(\cG, \cG) > 0$. 
Since~$\cG$ is $(H,A)$-stable, hence simple, it must be spherical. 
In particular, $\Ext^1(\cG, \cG) = 0$, hence~$\cF' \cong \cG^{\oplus m}$.

By induction on the rank, this concludes the case where~$\cF$ is $(H, A)$-semistable. 

Now assume that~$\cF$ is not $(H, A)$-semistable. 
The same arguments as for Gieseker stability (see, e.g., \cite[Theorem~1.3.4]{HL}) 
prove the existence of a  Harder--Narasimhan filtration for $\cF$ with respect to $(H, A)$-stability. 
Applying Lemma~\ref{lem:mukai-lemma} again, we see that every Harder--Narasimhan filtration factor of~$\cF$ is rigid. 
Combined with the filtrations of the semistable factors proven in the previous case, 
this immediately gives a filtration satisfying~\ref{it:factors-msph}. 

For part~\ref{it:factors-hom-zero}, just note that our construction in fact implies~$\bp_{H,A}(\cF_i) \succeq \bp_{H,A}(\cF_{i+1})$, 
hence a fortiori~$\upmu_{H}(\cF_i) \ge \upmu_{H}(\cF_{i+1})$, and, moreover, $\Hom(\cF_i,\cF_{i+1}) = 0$.
\end{proof}

\subsection{Brill--Noether inequality}\label{ss:bn}

The following is the key result for the proof of Theorem~\ref{thm:mb-s}.

\begin{proposition}\label{prop:rs1}
Let~$(S,H)$ be a Brill--Noether general quasipolarized smooth~$K3$ surface.
Let~$\cF$ be a globally generated spherical bundle on~$S$ with~$\rv(\cF) = (r, H, s)$.
If
\begin{equation*}
0 \to \cF_1 \to \cF \to \cF_2 \to 0
\end{equation*}
is an exact sequence of nontrivial spherical bundles with~$\upmu_H(\cF_1) \ge \upmu_H(\cF) \ge \upmu_H(\cF_2)$ then
\begin{equation*}
\rank(\cF_1) =  \rch_2(\cF_2) + \rank(\cF_2) = 1.
\end{equation*}
Moreover, the class~$\Xi \coloneqq \rc_1(\cF_2)$ is a special Mukai class of type~$(r,s)$.
\end{proposition}

\begin{proof}
Let~$\rv(\cF_i) = (r_i, D_i, s_i)$, so that~$\rv(\cF) = (r_1 + r_2, D_1 + D_2, s_1 + s_2)$; in particular~$D_1 + D_2 = H$.
Since~$\cF_1$, $\cF_2$, and~$\cF$ are spherical, \eqref{eq:rr-k3} implies that
\begin{equation}
\label{eq:sph}
D_1^2 = 2r_1s_1 - 2,
\qquad 
D_2^2 = 2r_2s_2 - 2,
\qquad 
(D_1 + D_2)^2 = 2(r_1 + r_2)(s_1 + s_2) - 2.
\end{equation}
Subtracting the sum of the first two equalities from the last one and dividing by~$2$, we obtain
\begin{equation*}
D_1 \cdot D_2 = r_1s_2 + r_2s_1 + 1.
\end{equation*}
This implies~$H \cdot D_1 = 2r_1s_1 + r_1s_2 + r_2s_1  - 1$ and~$H \cdot D_2 = 2r_2s_2  + r_1s_2 + r_2s_1  - 1$, hence
\begin{equation*}
2s_1 +  s_2 + \frac{r_2s_1 - 1}{r_1} =
\frac1{r_1}H \cdot D_1 = 
\upmu_H(\cF_1) \ge
\upmu_H(\cF_2) =
\frac1{r_2}H \cdot D_2 = 
s_1 + 2s_2 + \frac{r_1s_2 - 1}{r_2}.
\end{equation*}
Equivalently,
\begin{equation}
\label{eq:inequality}
s_1 - s_2 \ge \frac{r_1s_2 - 1}{r_2} - \frac{r_2s_1 - 1}{r_1}.
\end{equation} 
This inequality will be used a few times below.

By Corollary~\ref{cor:gg-quotient} the line bundle~$\det(\cF_2) \cong \cO_S(D_2)$ is globally generated.
Therefore, $D_2^2 \ge 0$, hence~$r_2s_2 \ge 1$ by~\eqref{eq:sph}, and since~$r_2 \ge 1$, we have~$s_2 \ge 1$.

Assume~$s_1 \le 0$.
Since~$r_1$ and~$r_2$ are positive, we obtain~$r_1s_2 - 1 \ge 0$, $r_2s_1 - 1 < 0$, 
hence the right side of~\eqref{eq:inequality} is positive, while the left side is negative, which is absurd.
Therefore, $s_1 \ge 1$.

Further, $h^2(\cO_S(D_2)) = 0$ by Corollary~\ref{cor:gg-quotient}.
Moreover, $h^2(\cO_S(D_1)) = 0$; indeed, otherwise Serre duality shows that~$-D_1$ is effective, hence~$\upmu_H(\cF_1) \le 0$, 
contradicting to the assumption.
Therefore,
\begin{equation*}
h^0(\cO_S(D_i)) \ge \upchi(\cO_S(D_i)) = \tfrac12D_i^2 + 2 = r_is_i + 1.
\end{equation*}
Since~$r_i \ge 1$ and~$s_i \ge 1$, this is positive, and since~$D_1 + D_2 = H$ and~$g = (r_1 + r_2)(s_1 + s_2)$ 
by the last equality in~\eqref{eq:sph}, the Brill--Noether property of~$S$ implies that
\begin{equation*}
(r_1s_1 + 1)(r_2s_2 + 1) \le h^0(\cO_S(D_1)) \cdot h^0(\cO_S(D_2)) \le (r_1 + r_2)(s_1 + s_2).
\end{equation*}
Expanding the products we obtain~$r_1r_2s_1s_2 + r_1s_1 + r_2s_2 + 1 \le r_1s_1 + r_2s_2 + r_1s_2 + r_2s_1$, hence 
\begin{equation*}
\label{eq:factors}
(r_1s_2 - 1)(r_2s_1 - 1) = r_1r_2s_1s_2 - r_1s_2 - r_2s_1 + 1 \le 0.
\end{equation*}
But as we have shown above, both factors in the left side are nonnegative, therefore one of them is zero.
If~$r_1s_2 - 1 \ge 1$, hence~$r_2s_1 - 1 = 0$, then~$r_2 = s_1 = 1$, 
the right side in~\eqref{eq:inequality} is positive, hence~$s_2 < s_1$, in contradiction to~$s_2 \ge 1$.
Therefore, $r_1s_2 - 1 = 0$, hence~$r_1 = s_2 = 1$, i.e.,
\begin{equation*}
\rank(\cF_1) =  \rch_2(\cF_2) + \rank(\cF_2) = 1.
\end{equation*}
Moreover, $h^0(\cO_S(D_1)) \ge r_1s_1 + 1 = s_1 + s_2 = s$ and~$h^0(\cO_S(D_2)) \ge r_2s_2 + 1 = r_2 + r_1 = r$, 
so the Brill--Noether property of~$S$ implies that these are equalities, hence~$h^1(\cO_S(D_2)) = 0$.

Finally, it follows that
\begin{equation*}
D_2 \cdot H = 2r_2s_2  + r_1s_2 + r_2s_1 - 1 = 
2r_2 + 1 + r_2s_1 - 1 =
r_2(s_1 + 2) = 
(r - 1)(s + 1),
\end{equation*}
hence the divisor class~$\Xi \coloneqq D_2$ is a special Mukai class of type~$(r,s)$.
\end{proof}

We will also need the following similar result.

\begin{lemma}\label{lem:two-plus-one}
Let~$(S,H)$ be a Brill--Noether general quasipolarized smooth~$K3$ surface.
Let~$\cF$ be a globally generated spherical bundle on~$S$ with~$\rank(\cF) = 3$ and~$\rc_1(\cF) = H$.
Then~$\cF$ cannot fit into an exact sequence
\begin{equation*}
0 \to \cF_1^{\oplus 2} \to \cF \to \cF_2 \to 0
\qquad\text{or}\qquad 
0 \to \cF_1 \to \cF \to \cF_2^{\oplus 2} \to 0,
\end{equation*}
where~$\cF_1$ and~$\cF_2$ are line bundles and~$\upmu_H(\cF_1) \ge \upmu_H(\cF_2)$.
\end{lemma}

\begin{proof}
Let~$\rv(\cF_i) = (1, D_i, s_i)$, so that~$D_i^2 = 2s_i - 2$.

Assume we have the first sequence. 
Then~$\rv(\cF) = (3, 2D_1 + D_2, 2s_1 + s_2)$, and since~$\cF$ is spherical, we have~$(2D_1 + D_2)^2 = 12s_1 + 6s_2 - 2$, which implies
\begin{equation}
\label{eq:d1-d2}
D_1 \cdot D_2 = s_1 + s_2 + 2.
\end{equation}
Moreover, as in Proposition~\ref{prop:rs1} the global generation of~$\cF_2$ implies~$s_2 \ge 1$, and since
\begin{equation*}
2(2s_1 - 2) + (s_1 + s_2 + 2) = \upmu_H(\cF_1) \ge \upmu_H(\cF_2) = 2(s_1 + s_2 + 2) + (2s_2 - 2),
\end{equation*}
we see that~$3s_1 \ge 3s_2 + 4$, hence~$s_1 \ge 3$.
Finally, we have~$h^2(\cO_S(D_1 + D_2)) = h^2(\cO_S(D_1)) = 0$ by the argument of Proposition~\ref{prop:rs1}, hence~$h^0(\cO_S(D_1 + D_2)) \ge \upchi(\cO_S(D_1 + D_2)) = 2s_1 + 2s_2 + 2$ and~$h^0(\cO_S(D_1)) \ge \upchi(\cO_S(D_1)) = s_1 + 1$, and the Brill--Noether inequality gives
\begin{equation*}
2(s_1 + s_2 + 1)(s_1 + 1) \le 6s_1 + 3s_2.
\end{equation*}
This can be rewritten as~$(2s_1 + 2s_2 - 1)(2s_1 - 1) + 3 \le 0$, contradicting~$s_1 \ge 3$ and~$s_2 \ge 1$.

Similarly, assuming the second sequence, we obtain~$\rv(\cF) = (3, D_1 + 2D_2, s_1 + 2s_2)$, which again implies~\eqref{eq:d1-d2}.
Arguing as before, we obtain~$s_2 \ge 1$ and~$3s_1 \ge 3s_2 - 4$, hence~$s_1 \ge 0$, and this time the Brill--Noether inequality gives
\begin{equation*}
2(s_1 + s_2 + 1)(s_2 + 1) \le 3s_1 + 6s_2
\end{equation*}
which can be rewritten as~$(2s_2 - 1)(2s_1 + 2s_2 - 1) + 3 \le 0$, contradicting $s_1 \ge 0$ and~$s_2 \ge 1$.
\end{proof}

\subsection{Proof of the theorem}\label{ss:proof-surface}

Now we can finally start proving Theorem~\ref{thm:mb-s}, so we return to the setup of Section~\ref{ss:lazarseld}.
We use the notation introduced therein.
Also recall Definition~\ref{def:ha-stability} and notation~\eqref{eq:Paris20231230}.

\begin{lemma}
\label{lem:barh-ha}
If the Lazarsfeld bundle~$\bL_{\baS}(\xi)$ is not $\baH$-stable then there is a $\sigma$-ample divisor class~$A \in \Pic(S)$ 
such that~$\bL_S(\xi) \cong \sigma^*(\bL_{\baS}(\xi))$ is not~$(H,A)$-stable.
\end{lemma}

\begin{proof}
Note that the functor~$\sigma_*$ defines a bijection between the set of all saturated subsheaves in~$\bL_S(\xi)$
and the set of all saturated subsheaves in~$\bL_{\baS}(\xi)$ 
(this can be seen directly; for the general result, see~\cite[Lemma and definition~(2.2)]{Esnault}).
Therefore, if~$\bL_{\baS}(\xi)$ is not $\baH$-stable 
there is a saturated subsheaf~$\cF \subset \bL_S(\xi)$ such that~$\sigma_*\cF \subset \bL_{\baS}(\xi)$ 
is a destabilizing saturated subsheaf.
Let~$A = \sum a_i R_i$ be a $\sigma$-ample class such that~$a_i < 0$ for all~$i$.
We will show that~$\cF$ destabilizes~$\bL_S(\xi)$ with respect to~$(H,A)$-stability.

First, we have~$H = \sigma^*(\baH)$ and~$\sigma_*(\rc_1(\cF)) = \rc_1(\sigma_*\cF)$, hence the projection formula implies
\begin{equation*}
\label{eq:muh-ineq}
\upmu_H(\cF) = 
\tfrac1{\rank(\cF)} \rc_1(\cF) \cdot \sigma^*(\baH) =
\tfrac1{\rank(\sigma_*\cF)} \rc_1(\sigma_*\cF) \cdot \baH =
\upmu_{\baH}(\sigma_*\cF) \ge 
\upmu_{\baH}(\bL_{\baS}(\xi)) =
\upmu_H(\bL_S(\xi)).
\end{equation*}
Moreover, if this is an equality 
then we must have~$\updelta(\sigma_*\cF) \ge \updelta(\bL_{\baS}(\xi)) = \updelta(\bL_{S}(\xi))$.

Next, for any curve~$R_i$ in~$\fR(S,H)$
we have~$H \cdot R_i = 0$, hence~$\rc_1(\bL_S(\xi)) \cdot A = 0$.
Note also that~$\bL_S(\xi)$ is trivial on~$R_i$, because it is a pullback along~$\sigma$.
Thus, $\cF\vert_{R_i}$ is a subsheaf in a trivial vector bundle, hence we have~$\rc_1(\cF) \cdot R_i \le 0$ for all~$i$.
Therefore, we have~$\rc_1(\cF) \cdot A \ge 0$, hence
\begin{equation}
\label{eq:mua-ineq}
\upmu_{A}(\cF) = 
\tfrac1{\rank(\cF)}{\rc_1(\cF)\cdot A} \ge 
0 = 
\tfrac1{\rank(\bL_S(\xi))}{\rc_1(\bL_S(\xi))\cdot A} =
\upmu_{A}(\bL_S(\xi)).
\end{equation}
Finally, if~\eqref{eq:mua-ineq} is an equality, 
we have~$\rc_1(\cF) \cdot R_i = 0$ for each~$R_i$,
hence~$\cF$ is trivial on each of these curves, hence~$\cF \cong \sigma^*(\sigma_*\cF)$,
hence~$\upchi(\cF) = \upchi(\sigma_*\cF)$, 
hence~$\updelta(\cF) = \updelta(\sigma_*\cF)$, 
and we conclude that~$\bp_{H,A}(\cF) \succeq \bp_{H,A}(\bL_S(\xi))$,
so that~$\bL_S(\xi)$ is not $(H,A)$-stable.
\end{proof}

\begin{lemma}
\label{lem:r23-instability}
If~$r \in \{2,3\}$ and the Lazarsfeld bundle~$\bL_{\baS}(\xi)$ is not $\baH$-Gieseker stable, 
then its pullback~$\bL_S(\xi) \cong \sigma^*\bL_{\baS}(\xi)$ fits into an exact sequence
\begin{equation*}
0 \to \cG_1^{\oplus m_1} \to \bL_S(\xi) \to \cG_2^{\oplus m_2} \to 0,
\end{equation*}
where~$\cG_1$ and~$\cG_2$ are spherical bundles with~$\upmu_H(\cG_1) \ge \upmu_H(\cG_2)$ and~$(m_1,m_2) \in \{(1,1), (1,2), (2,1)\}$.
\end{lemma}

\begin{proof}
By Lemma~\ref{lem:barh-ha} the sheaf~$\bL_S(\xi)$ is not $(H,A)$-stable for appropriate~$A$,
hence the multispherical filtration of~$\bL_S(\xi)$ provided by Proposition~\ref{prop:filtration-msph} is not trivial.
In the case~$r = 2$, therefore, it has the above form with~$(m_1,m_2) = (1,1)$.

Now assume~$r = 3$. 
The only case where the conclusion is not immediately obvious 
is where the multispherical filtration of~$\bL_S(\xi)$ has three spherical factors~$\cG_1 \not\cong \cG_2 \not\cong \cG_3$,
By Proposition~\ref{prop:rigid-lf} the sheaves~$\cG_i$ are locally free, hence they are line bundles.
Note also that
\begin{equation*}
\Hom(\cG_1,\cG_2) = \Hom(\cG_2,\cG_3) = \Hom(\cG_1,\cG_3) = 0.
\end{equation*}
Indeed, the first two spaces are zero by Proposition~\ref{prop:filtration-msph}.
Moreover, if~$\upmu_H(\cG_1) > \upmu_H(\cG_3)$, the third space is zero by $(H,A)$-stability of~$\cG_1$ and~$\cG_3$, 
and otherwise we would have~$\upmu_H(\cG_i) = \upmu_H(\bL_S(\xi))$ for all~$i$, hence~$\upmu_H(\bL_S(\xi))$ would be integral, 
while in fact~$\upmu_H(\bL_S(\xi)) = \tfrac13H^2 = 2s - \tfrac23$.

Now, consider the sheaves~$\cG_{1,2} \coloneqq \Ker(\bL_S(\xi) \to \cG_3)$ and~$\cG_{2,3} \coloneqq \bL_S(\xi)/\cG_1$. 
It follows that
\begin{equation*}
\Hom(\cG_{1,2}, \cG_3) = \Hom(\cG_1, \cG_{2,3}) = 0
\end{equation*}
hence~$\cG_{1,2}$ and~$\cG_{2,3}$ are rigid by Lemma~\ref{lem:mukai-lemma}.
Furthermore, Lemma~\ref{lem:mukai-lemma} implies that the maps
\begin{equation*}
\kk = \Hom(\cG_i,\cG_i) \to \Ext^1(\cG_i,\cG_{i-1})
\qquad\text{and}\qquad 
\kk = \Hom(\cG_i,\cG_i) \to \Ext^1(\cG_{i+1},\cG_i)
\end{equation*}
are surjective, in particular, the spaces~$\Ext^1(\cG_2,\cG_1)$ and~$\Ext^1(\cG_3,\cG_2)$ are at most 1-dimensional.

If~$\Hom(\cG_2,\cG_1) \ne 0$ and~$\Hom(\cG_3,\cG_2) \ne 0$, 
the composition of nontrivial morphisms~\mbox{$\cG_3 \to \cG_2$} and~\mbox{$\cG_2 \to \cG_1$} is nontrivial, 
hence~$\Hom(\cG_3,\cG_1) \ne 0$, which is impossible by Lemma~\ref{lem:hom-back} because~$\bL_S(\xi)$ is spherical.
Thus, one of the above spaces must be zero.

Assume~$\Hom(\cG_2,\cG_1) = 0$ and~$\Hom(\cG_3,\cG_2) \ne 0$.
If~$\Ext^1(\cG_2,\cG_1) = 0$ then~$\cG_{1,2} \cong \cG_1 \oplus \cG_2$, 
hence we have~$\Hom(\cG_3, \cG_{1,2}) \ne 0$, which is impossible by Lemma~\ref{lem:hom-back}. 
Therefore, $\Ext^1(\cG_2,\cG_1) = \kk$, hence we have~$\upchi(\cG_{1,2},\cG_{1,2}) = 2$,
and since the sheaf~$\cG_{1,2}$ is rigid, it is spherical.
Therefore the filtration~\mbox{$0 \to \cG_{1,2} \to \bL_S(\xi) \to \cG_3 \to 0$} has the required properties.

The case where~$\Hom(\cG_2,\cG_1) \ne 0$ and~$\Hom(\cG_3,\cG_2) = 0$ is considered analogously; 
in this case the filtration~$0 \to \cG_1 \to \bL_S(\xi) \to \cG_{2,3}\to 0$ has the required properties.

Finally, assume~$\Hom(\cG_2,\cG_1) = \Hom(\cG_3,\cG_2) = 0$.
If also~$\Ext^1(\cG_2,\cG_1) = \Ext^1(\cG_3,\cG_2) = 0$ then it is easy to see that~$\cG_2$ is a direct summand of~$\bL_S(\xi)$, which is impossible because~$\bL_S(\xi)$ is spherical.
Therefore, $\Ext^1(\cG_i,\cG_{i-1}) = \kk$ for~$i = 2$ or~$i = 3$, 
hence~$\cG_{i-1,i}$ is spherical and as in one of the two previous cases 
we obtain a filtration of~$\bL_S(\xi)$ with the required properties.
\end{proof}

\begin{proof}[Proof of Theorem~\textup{\ref{thm:mb-s}}]
Since the surface~$S$ contains a BNP-general curve, it is BN-general (see Theorem~\ref{thm:bnpc}\ref{it:bnp-bn}).
Assume~$\bL_{\baS}(\xi)$ is not $\baH$-Gieseker stable.
Dualizing the sequence produced by Lemma~\ref{lem:r23-instability}, we obtain an exact sequence
\begin{equation*}
0 \to \cF_1 \to \bL_S(\xi)^\vee \to \cF_2 \to 0
\end{equation*}
of multispherical sheaves,
and we deduce from Lemma~\ref{lem:two-plus-one} that 
the cases where~\mbox{$(m_1,m_2) = (1,2)$} or~\mbox{$(m_1,m_2) = (2,1)$ are impossible},
hence both~$\cF_1$ and~$\cF_2$ must be spherical.
Therefore, Proposition~\ref{prop:rs1} proves that~\mbox{$\Xi \coloneqq \rc_1(\cF_2)$} is a special Mukai class of type~$(r,s)$ on~$S$
and~$\cF_1$ is a line bundle, hence~\mbox{$\cF_1 \cong \cO_S(H - \Xi)$}.
Consider the composition
\begin{equation*}
\cO_S(H - \Xi) \cong \cF_1 \hookrightarrow \bL_S(\xi)^\vee \to j_*\eta,
\end{equation*}
where the last arrow comes from exact sequence~\eqref{eq:cus-dual}.
If the composition vanishes, sequence~\eqref{eq:cus-dual} implies 
that the embedding~$\cO_S(H - \Xi) \hookrightarrow \bL_S(\xi)^\vee$ factors through~$\rH^0(C,\xi)^\vee \otimes \cO_{S}$,
which is absurd because~$H \cdot (H-\Xi) = (r + 1)(s - 1) > 0$.
Therefore, the composition is nonzero, so it factors through a nonzero morphism
\begin{equation*}
\cO_S(H - \Xi)\vert_C \to \eta.
\end{equation*}
The source and target are line bundles of the same degree~$(r + 1)(s - 1)$, hence the above morphism is an isomorphism,
and we conclude that~$\cO_S(\Xi)\vert_C \cong \eta^{-1}(K_C) \cong \xi$.

Conversely, assume~$\Xi$ is a special Mukai class of type~$(r,s)$ on~$S$.
Let~$C \subset S$ be a BNP-general curve in~$|H|$ and let 
\begin{equation*}
\xi \coloneqq \cO_S(\Xi)\vert_C,
\qquad 
\eta \coloneqq \cO_S(H - \Xi)\vert_C.
\end{equation*}
We know from Proposition~\ref{prop:h-xi} that~$(\xi,\eta)$ is a Mukai pair of type~$(r,s)$ 
and the restriction induces an isomorphism~$\rH^0(S,\cO_S(\Xi)) \cong \rH^0(C,\xi)$.
Consider the commutative diagram with exact rows
\begin{equation*}
\xymatrix{
0 \ar[r] & 
0 \ar[r] \ar[d] &
\rH^0(S,\cO_S(\Xi)) \otimes \cO_S \ar@{=}[r] \ar[d] & 
\rH^0(C,\xi) \otimes \cO_S \ar[r] \ar[d] &
0
\\
0 \ar[r] &
\cO_S(\Xi-H) \ar[r] &
\cO_S(\Xi) \ar[r] &
j_*\xi \ar[r] &
0.
}
\end{equation*}
The middle vertical arrow is surjective, because~$\Xi$ is globally generated,
and the kernel of the right vertical arrow is~$\sigma^*\bL_{S}(\xi)$.
It follows that there is an epimorphism~$\bL_{S}(\xi) \twoheadrightarrow \cO_S(\Xi-H)$ and
\begin{equation*}
\Ker\bigl(\bL_{S}(\xi) \twoheadrightarrow \cO_S(\Xi-H)\bigr) \cong \Ker\bigl(\rH^0(S,\cO_S(\Xi)) \otimes \cO_S \twoheadrightarrow \cO_S(\Xi)\bigr).
\end{equation*}
Moreover, as~$\sigma_*\cO_S(\Xi)$ is globally generated by Corollary~\ref{cor:bs-h-xi}, 
the first direct image of the right-hand side vanishes, hence the same is true for the first direct image of the left-hand side, 
hence the induced morphism~$\bL_{\baS}(\xi) \cong \sigma_*\bL_{S}(\xi) \to \sigma_*\cO_S(\Xi-H)$ is surjective,
and since
\begin{equation*}
\upmu_{\baH}(\bL_{\baS}(\xi)) = \tfrac{2}{r} - 2s \ge -(r+1)(s-1) = \upmu_{\baH}(\sigma_*\cO_S(\Xi-H)),
\end{equation*}
such an epimorphism violates stability of~$\bL_{\baS}(\xi)$.

Now assume that~$(S,H)$ does not have special Mukai classes of type~$(r,s)$.
Then for any BNP-general curve~$C$ and any Mukai pair~$(\xi,\eta)$ on it the bundle~$\bL_{\baS}(\xi)$ is stable, 
hence it is a Mukai bundle of type~$(r,s)$.
Finally, since a Mukai bundle is unique by Lemma~\ref{lem:mbs-unique}, 
$\bL_{\baS}(\xi)$ does not depend on~$C$ or~$\xi$.
\end{proof}

%%%%%%%%%%%%%%%%%%%%%%%%%

\section{Mukai extension classes}

In this section we define and study Mukai extension classes on curves~$C \subset S$. 
In particular, when~$\Pic(S) = \bZ \cdot H$ we will show 
that the defining property of a Mukai extension class (see Definition~\ref{def:mukai-extension})
uniquely characterises the restriction of the (dual of) the Lazarsfeld bundle on~$S$, 
see Definition~\ref{def:lec} and Theorem~\ref{thm:mukai-extension}. 

Throughout this section we work on a smooth quasipolarized K3 surface~$S$.

\subsection{Restriction of Lazarsfeld bundles}

Here we establish some cohomology vanishings for the Lazarsfeld bundles and deduce some corollaries about their restrictions to curves.

The first result is elementary.

\begin{lemma}
\label{lem:l-m-h}
Let~$(S,H)$ be a quasipolarized~$K3$ surface of genus~$g = r \cdot s \ge 4$
and let~$C \subset S$ be a BNP-general curve in~$|H|$.
Assume~$s \ge r \in \{2,3\}$.
If~$(\xi,\eta)$ is a Mukai pair of type~$(r,s)$ on~$C$ then
\begin{alignat*}{3}
\rH^2(S,\bL_S(\xi) \otimes \cO_S(H)) &= \rH^1(S,\bL_S(\xi) \otimes \cO_S(H)) && = 0 \quad \text{and} \\
\rH^1(S,\bL_S(\xi)^\vee \otimes \cO_S(-H)) &= \rH^0(S,\bL_S(\xi)^\vee \otimes \cO_S(-H)) && = 0.
\end{alignat*}
\end{lemma}

\begin{proof}
Twisting~\eqref{eq:def-cus} by~$\cO_S(H)$ we obtain an exact sequence
\begin{equation*}
0 \to \bL_S(\xi) \otimes \cO_S(H) \xrightarrow\quad \rH^0(C,\xi) \otimes \cO_S(H) \xrightarrow{\ \ev\ } j_*(\xi(K_C)) \to 0.
\end{equation*}
Since~$\rH^1(S, \cO_S(H)) = \rH^2(S, \cO_S(H)) = 0$ by~\eqref{eq:genus-s}
and~$\rH^1(C, \xi(K_C)) = 0$ because~$\deg(\xi) > 0$,
the vanishing of~$\rH^2(S,\bL_S(\xi) \otimes \cO_S(H))$ is obvious,
and to prove the vanishing of~$\rH^1(S,\bL_S(\xi) \otimes \cO_S(H))$ it is enough to check that the morphism
\begin{equation*}
\rH^0(C,\xi) \otimes \rH^0(C, \cO_C(K_C)) \to \rH^0(C, \xi(K_C))
\end{equation*}
is surjective, which is proved in Proposition~\ref{prop:xi-eta}\ref{it:xi-k}.
The remaining two vanishings follow from the first two by Serre duality.
\end{proof}

The second vanishing result is more complicated; in particular, it relies on Theorem~\ref{thm:mb-s}.

\begin{proposition}
\label{prop:l-l-m-h}
Let~$(S,H)$ be a quasipolarized~$K3$ surface of genus~$g = r \cdot s \ge 6$
and let~$C \subset S$ be a BNP-general curve in~$|H|$.
Assume~$s \ge r \in \{2,3\}$.
If~$S$ does not have special Mukai classes of type~$(r,s)$, then for any Mukai pair~$(\xi,\eta)$ of type~$(r,s)$ on~$C$ we have
\begin{equation*}
\rH^1(S, \bL_S(\xi)^\vee \otimes \bL_S(\xi) \otimes \cO_S(-H)) = 
\rH^1(S, \bL_S(\xi)^\vee \otimes \bL_S(\xi) \otimes \cO_S(H)) = 0.
\end{equation*}
\end{proposition}

\begin{proof}
First, tensoring~\eqref{eq:cus-dual} by~$\bL_S(\xi) \otimes \cO_S(H)$ we obtain an exact sequence
\begin{equation*}
0 \to 
\rH^0(C, \xi)^\vee \otimes \bL_S(\xi) \otimes \cO_S(H) \to 
\bL_S(\xi)^\vee \otimes \bL_S(\xi) \otimes \cO_S(H) \to j_*(j^*\bL_S(\xi) \otimes \eta(K_C)) \to 0.
\end{equation*}
Its first term has no first cohomology by Lemma~\ref{lem:l-m-h}, 
so to prove the vanishing of the first cohomology of the middle term, 
it is enough to check that~$\rH^1(C, j^*\bL_S(\xi) \otimes \eta(K_C)) = 0$.

Now we note that the number of Mukai pairs on~$C$ of type~$(r,s)$ 
is equal to the degree~$\rN(r,s)$ of~$\Gr(r,r + s)$, see Remark~\ref{rem:number-mp}.
Since~$r,s \ge 2$ and~$(r,s) \ne (2,2)$, this number is greater than~$2$; in particular, we can choose a Mukai pair~$(\xi',\eta')$ of type~$(r,s)$ such that~$\xi \not\cong \xi'$, $\eta \not\cong \xi'$.
Applying Theorem~\ref{thm:mb-s} we obtain an isomorphism
\begin{equation*}
\bL_S(\xi) \cong \bL_S(\xi'),
\end{equation*}
hence~$j^*\bL_S(\xi) \cong j^*\bL_S(\xi')$, so it is enough to check that~$\rH^1(C, j^*\bL_S(\xi') \otimes \eta(K_C)) = 0$.

Now consider sequence~\eqref{eq:def-cus} for the line bundle~$\xi'$ and its restriction to~$C$:
\begin{equation*}
0 \to \xi'(-K_C) \to j^*\bL_S(\xi') \to \rH^0(C, \xi') \otimes \cO_C \to \xi' \to 0.
\end{equation*}
Tensoring this by~$\eta(K_C)$, we obtain the following exact sequence
\begin{equation*}
0 \to \xi' \otimes \eta \to j^*\bL_S(\xi') \otimes \eta(K_C) \to \rH^0(C, \xi') \otimes \eta(K_C) \to \xi' \otimes \eta(K_C) \to 0.
\end{equation*}
We have~$\rH^1(C, \xi' \otimes \eta) \cong \rH^1(C, \xi' \otimes \xi^{-1} \otimes \cO_C(K_C))$, and since~$\xi' \otimes \xi^{-1}$ is a nontrivial line bundle of degree~$0$, this space vanishes.
Moreover, $\rH^1(C, \eta(K_C)) = 0$ because~$\deg(\eta) > 0$.
Therefore,
\begin{equation*}
\label{eq:h1-coker}
\rH^1(C, j^*\bL_S(\xi') \otimes \eta(K_C)) \cong 
\Coker\Big(\rH^0(C, \xi') \otimes \rH^0(C, \eta(K_C)) \to \rH^0(C, \xi' \otimes \eta(K_C))\Big),
\end{equation*}
and it remains to note that the morphism in the right side is surjective by Proposition~\ref{prop:xi-eta}\ref{it:xi-eta-k}.
Therefore, $\rH^1(S, \bL_S(\xi)^\vee \otimes \bL_S(\xi) \otimes \cO_S(H)) = 0$, and the other vanishing follows by Serre duality.
\end{proof}

\subsection{Lazarsfeld extension class}

For any globally generated line bundle~$\xi$ on a smooth curve~$C$ we consider the evaluation morphism~$\ev \colon \rH^0(C, \xi) \otimes \cO_C \twoheadrightarrow \xi$.
Its dual gives an exact sequence
\begin{equation}
\label{eq:def-rxi}
0 \to \xi^{-1} \xrightarrow{\ \ev^\vee\ } \rH^0(C,\xi)^\vee \otimes \cO_C \xrightarrow{\quad} \bR_C(\xi^{-1}) \to 0,
\end{equation}
defining a vector bundle~$\bR_C(\xi^{-1})$ on~$C$.

\begin{lemma}\label{lem:rxivee-h0}
For any globally generated line bundle~$\xi$ on a smooth curve~$C$ the bundle~$\bR_C(\xi^{-1})$ is globally generated, 
$\rH^0(C,\bR_C(\xi^{-1})^\vee) = 0$, and~$\rH^1(C,\bR_C(\xi^{-1}) \otimes \cO_C(K_C)) = 0$.
\end{lemma}

\begin{proof}
First, the sheaf~$\bR_C(\xi^{-1})^\vee$ is the kernel of the evaluation morphism~$\rH^0(C,\xi) \otimes \cO_C \xrightarrow{\ \ev\ } \xi$, hence~$\rH^0(C, \bR_C(\xi^{-1})^\vee) = 0$.
Applying Serre duality we obtain~$\rH^1(C,\bR_C(\xi^{-1}) \otimes \cO_C(K_C)) = 0$.
Global generation of~$\bR_C(\xi^{-1})$ is immediate from~\eqref{eq:def-rxi}.
\end{proof}

If~$(\xi,\eta)$ is a Mukai pair, the sheaf~$\bR_C(\xi^{-1})$ has stronger properties.

\begin{lemma}
\label{lem:rxi-h0}
If~$(\xi,\eta)$ is a Mukai pair of type~$(r,s)$ with~$s \ge r \in \{2,3\}$ 
on a BNP-general curve~$C$ of genus~$g = r \cdot s$, then
\begin{equation*}
\rH^0(C,\bR_C(\xi^{-1})) \cong \rH^0(C,\xi)^\vee
\qquad\text{and}\qquad  
\Hom(\bR_C(\xi^{-1}), \eta) = 0.
\end{equation*}
Moreover, the sheaf~$\bR_C(\xi^{-1})$ is simple.
\end{lemma}

\begin{proof}
Applying the functor~$\Hom(-, \eta)$ to~\eqref{eq:def-rxi} we obtain a left-exact sequence
\begin{equation*}
0 \to \Hom(\bR_C(\xi^{-1}), \eta) \to \rH^0(C,\xi) \otimes \rH^0(C, \eta) \to \rH^0(C, \xi \otimes \eta).
\end{equation*}
Its second arrow is induced by the evaluation morphism of~$\xi$, hence it coincides with the Petri map.
Since the curve~$C$ is BNP-general and~$(\xi,\eta)$ is a Mukai pair, this map is an isomorphism (see Lemma~\ref{lem:xi-eta}), hence~$\Hom(\bR_C(\xi^{-1}), \eta) = 0$.

Furthermore, since~$\rH^0(C,\xi^{-1}) = 0$, the long exact sequence of cohomology of~\eqref{eq:def-rxi} looks like
\begin{equation*}
0 \to 
\rH^0(C,\xi)^\vee \to 
\rH^0(C,\bR_C(\xi^{-1})) \to 
\rH^1(C,\xi^{-1}) \to 
\rH^0(C,\xi)^\vee \otimes \rH^1(C,\cO_C) \to 
\rH^1(C,\bR_C(\xi^{-1})) \to 
0.
\end{equation*}
We need to check that the connecting map~$\rH^0(C,\bR_C(\xi^{-1})) \to \rH^1(C,\xi^{-1})$ is zero, 
i.e., that the map~$\rH^1(C,\xi^{-1}) \to \rH^0(C,\xi)^\vee \otimes \rH^1(C,\cO_C)$ is injective, 
i.e., that its dual map is surjective.
But the dual map is nothing but the map 
\begin{equation*}
\rH^0(C,\xi) \otimes \rH^0(C, \cO_C(K_C)) \to \rH^0(C, \xi(K_C))
\end{equation*}
whose surjectivity was established in Proposition~\ref{prop:xi-eta}\ref{it:xi-k}.
Thus, $\rH^0(C,\bR_C(\xi^{-1})) \cong \rH^0(C,\xi)^\vee$.

Finally, since the second arrow in~\eqref{eq:def-rxi} induces an isomorphism of global sections, it is the evaluation morphism for~$\bR_C(\xi^{-1})$, hence any endomorphism of~$\bR_C(\xi^{-1})$ extends to an endomorphism of the exact sequence~\eqref{eq:def-rxi}. 
On the other hand, since~$\Hom(\bR_C(\xi^{-1}), \cO_C) = 0$ by Lemma~\ref{lem:rxivee-h0}, an endomorphism of the exact sequence~\eqref{eq:def-rxi} is determined uniquely by its endomorphism of~$\xi^{-1}$.
As~$\xi^{-1}$ is simple, the same holds for~$\bR_C(\xi^{-1})$.
\end{proof}

The sheaf~$\bR_C(\xi^{-1})$ is naturally related to the restriction~$j^*\bL_S(\xi)$ of the Lazarsfeld bundle.

\begin{lemma}\label{lem:restricted-lb}
Let~$(\xi,\eta)$ be a Mukai pair on a BNP-general curve~$C$ 
which lies on a smooth quasipolarized $K3$ surface~$S$. 
If~$\phi \colon \bL_S(\xi)^\vee \to j_*\eta$ is the epimorphism from~\eqref{eq:cus-dual} there is a canonical exact sequence
\begin{equation}
\label{eq:rxi-blxi}
0 \to \bR_C(\xi^{-1}) \xrightarrow{\qquad} j^*\bL_S(\xi)^\vee \xrightarrow{\ j^*(\phi)\ } \eta \to 0.
\end{equation}
\end{lemma}

\begin{proof}
Pulling back~\eqref{eq:cus-dual} and using an isomorphism~$\eta(-K_C) \cong \xi^{-1}$, we obtain an exact sequence
\begin{equation*}
0 \to 
\xi^{-1} \xrightarrow{\qquad} 
\rH^0(C,\xi)^\vee \otimes \cO_C \xrightarrow{\qquad} 
j^*\bL_S(\xi)^\vee \xrightarrow{\ j^*(\phi)\ } 
\eta \to 0,
\end{equation*}
in particular~$j^*(\phi)$ is surjective.
Moreover, this sequence is dual to the exact sequence
\begin{equation*}
0 \to \eta^{-1} \xrightarrow{\quad} j^*\bL_S(\xi) \xrightarrow{\quad} \rH^0(C,\xi) \otimes \cO_C \xrightarrow{\ \ev\ } \xi \to 0
\end{equation*}
obtained by pulling back~\eqref{eq:def-cus} (because their middle maps are mutually dual), 
hence the first map in the first sequence is the dual evaluation morphism~$\ev^\vee$.
Therefore, its cokernel is isomorphic to~$\bR_C(\xi^{-1})$, 
hence~$\Ker(j^*(\phi)) \cong \bR_C(\xi^{-1})$, and we obtain~\eqref{eq:rxi-blxi}.
\end{proof}

\begin{definition}
\label{def:lec}
We call the extension class 
\begin{equation*}
\epsilon_{\bL_S(\xi)} \in \Ext^1(\eta, \bR_C(\xi^{-1}))
\end{equation*}
of the exact sequence~\eqref{eq:rxi-blxi} the {\sf Lazarsfeld extension class}.
\end{definition}

\begin{remark}
\label{rem:lec}
Since the sheaves~$\eta$ and~$\bR_C(\xi^{-1})$ are both simple, 
the extension class~$\epsilon_{\bL_S(\xi)}$ of~\eqref{eq:rxi-blxi} is well-defined up to rescaling.
\end{remark}

The main result of this subsection is the following.

\begin{proposition}
\label{prop:lec}
Let~$(S,H)$ be a smooth quasipolarized~$K3$ surface of genus~$g = r \cdot s \ge 6$ 
and let~$C \subset S$ be a BNP-general curve in~$|H|$.
Assume~$s \ge r \in \{2,3\}$.
If~$S$ does not have special Mukai classes of type~$(r,s)$ then for any Mukai pair~$(\xi,\eta)$ of type~$(r,s)$ on~$C$ 
the restriction~$j^*\bL_S(\xi)$ of the Lazarsfeld bundle to~$C$ is simple and
the Lazarsfeld extension class~$\epsilon_{\bL_S(\xi)}$ is nonzero, but the connecting morphism
\begin{equation*}
\epsilon_{\bL_S(\xi)} \colon \rH^0(C,\eta) \to \rH^1(C, \bR_C(\xi^{-1}))
\end{equation*}
of the exact sequence~\eqref{eq:rxi-blxi} vanishes.
\end{proposition}

\begin{proof}
Consider the exact sequence
\begin{equation*}
0 \to 
\bL_S(\xi)^\vee \otimes \bL_S(\xi) \otimes \cO_S(-H) \to 
\bL_S(\xi)^\vee \otimes \bL_S(\xi) \to 
j_*j^*(\bL_S(\xi)^\vee \otimes \bL_S(\xi)) \to 0.
\end{equation*}
We have~$h^0(\bL_S(\xi)^\vee \otimes \bL_S(\xi) \otimes \cO_S(-H)) = 0$ because $\bL_S(\xi)$ is spherical, and in particular simple, by Lemma~\ref{lem:lazarsfeld-bundle}, and $h^1(\bL_S(\xi)^\vee \otimes \bL_S(\xi) \otimes \cO_S(-H)) = 0$ by Proposition~\ref{prop:l-l-m-h}. 
Therefore 
\begin{equation*}
\End(j^*\bL_S(\xi)) =
\rH^0(C, j^*(\bL_S(\xi)^\vee \otimes \bL_S(\xi))) = 
\rH^0(S, \bL_S(\xi)^\vee \otimes \bL_S(\xi)) =
\End(\bL_S(\xi)).
\end{equation*}
Since $\bL_S(\xi)$ is simple, so is~$j^*\bL_S(\xi)$, and a fortiori $j^*\bL_S(\xi)$ is indecomposable.
Thus the exact sequence~\eqref{eq:rxi-blxi} does not split, hence~$\epsilon_{\bL_S(\xi)} \ne 0$.

Furthermore, consider the exact sequence
\begin{equation*}
0 \to \bL_S(\xi)^\vee \otimes \cO_S(-H) \to \bL_S(\xi)^\vee \to j_*j^*(\bL_S(\xi)^\vee) \to 0.
\end{equation*}
We have~$h^0(\bL_S(\xi)^\vee \otimes \cO_S(-H)) = h^1(\bL_S(\xi)^\vee \otimes \cO_S(-H)) = 0$ by Lemma~\ref{lem:l-m-h}, hence
\begin{equation*}
h^0(j^*\bL_S(\xi)^\vee) = h^0(\bL_S(\xi)^\vee) = r + s,
\end{equation*}
where the second equality uses~\eqref{eq:cus-h-dual}.
Finally, consider the cohomology exact sequence of~\eqref{eq:rxi-blxi}:
\begin{equation*}
0 \to \rH^0(C,\bR_C(\xi^{-1})) \to \rH^0(C, j^*\bL_S(\xi)^\vee) \to \rH^0(C, \eta) \to \rH^1(C,\bR_C(\xi^{-1})).
\end{equation*}
Its first three terms have dimension~$r$ (by Lemma~\ref{lem:rxi-h0}), 
$r+s$ (proved above), and~$s$, respectively, and therefore the last map vanishes.
\end{proof}

\subsection{Mukai extension class}

In this subsection we show that the properties of the Lazarsfeld extension class 
established in Proposition~\ref{prop:lec} characterize it uniquely if~$\Pic(S) = \ZZ \cdot H$.
We axiomatize these properties in the following

\begin{definition}
\label{def:mukai-extension}
Let~$(\xi,\eta)$ be a Mukai pair on~$C$.
{\sf A Mukai extension} is a non-split exact sequence
\begin{equation}
\label{eq:mec}
0 \to \bR_C(\xi^{-1}) \to \cG \to \eta \to 0
\end{equation}
such that the connecting morphism~$\rH^0(C, \eta) \to \rH^1(C, \bR_C(\xi^{-1}))$ vanishes.
The extension class
\begin{equation*}
\epsilon \in \Ext^1(\eta,\bR_C(\xi^{-1}))
\end{equation*}
of a Mukai extension is called {\sf a Mukai extension class}.
\end{definition}

To study Mukai extension classes we use the following generalization of the Lazarsfeld construction.
In~\S\ref{ss:lazarseld} we defined the Lazarsfeld bundle~$\bL_S(\xi)$ of a globally generated line bundle~$\xi$.
The same construction can be applied to any globally generated vector bundle~$\cG$ on~$C$; 
it defines a vector bundle~$\bL_S(\cG)$ that fits into an exact sequence
\begin{equation*}\label{eq:bl-cg}
0 \to \bL_S(\cG) \xrightarrow{\quad} \rH^0(C, \cG) \otimes \cO_S \xrightarrow{\ \ev\ } j_*\cG \to 0.
\end{equation*}
We will apply this construction to all bundles in~\eqref{eq:mec}.

\begin{lemma}\label{lem:blbrxi}
If~$(\xi,\eta)$ is a Mukai pair of type~$(r,s)$ with~$s \ge r \in \{2,3\}$
on a BNP-general curve~$C \subset S$ in a linear system~$|H|$ then
\begin{equation*}
\bL_S(\bR_C(\xi^{-1})) \cong \bL_S(\xi)^\vee \otimes \cO_S(-H).
\end{equation*}
\end{lemma}

\begin{proof}
Recall that~$\bR_C(\xi^{-1})$ is globally generated by Lemma~\ref{lem:rxivee-h0}, hence~$\bL_S(\bR_C(\xi^{-1}))$ is well-defined.
Moreover, $\rH^0(C,\bR_C(\xi^{-1})) \cong \rH^0(C,\xi)^\vee$ by Lemma~\ref{lem:rxi-h0}.
Therefore, twisting the defining exact sequence of~$\bL_S(\bR_C(\xi^{-1}))$ by~$\cO_S(H)$ we obtain
\begin{equation*}
0 \to \bL_S(\bR_C(\xi^{-1})) \otimes \cO_S(H) \to \rH^0(C,\xi)^\vee \otimes \cO_S(H) \to j_*(\bR_C(\xi^{-1}) \otimes \cO_C(K_C)) \to 0.
\end{equation*}
By~\eqref{eq:genus-s}, we have~$\rH^2(S, \bL_S(\bR_C(\xi^{-1})) \otimes \cO_S(H)) = \rH^1(C, \bR_C(\xi^{-1}) \otimes \cO_C(K_C))$, 
and by Lemma~\ref{lem:rxivee-h0} this is zero.
Applying Serre duality we deduce~$\rH^0(S, \bL_S(\bR_C(\xi^{-1}))^\vee \otimes \cO_S(-H)) = 0$.

Furthermore, consider the commutative diagram
\begin{equation*}
\xymatrix@R=3ex{
0 \ar[r] & 
0 \ar[r] \ar[d] & 
\rH^0(C,\xi)^\vee \otimes \cO_S \ar@{=}[r] \ar[d] & 
\rH^0(C,\bR_C(\xi^{-1})) \otimes \cO_S \ar[r] \ar[d]^{\ev} & 
0
\\
0 \ar[r] & 
j_*(\xi^{-1}) \ar[r] & 
\rH^0(C,\xi)^\vee \otimes j_*\cO_C \ar[r]^-{\ev} & 
j_*\bR_C(\xi^{-1}) \ar[r] & 
0,
}
\end{equation*}
where the bottom row is the pushforward of~\eqref{eq:def-rxi}.
The middle vertical arrow is induced by the natural morphism~$\cO_S \to j_*\cO_C$, 
hence it is surjective, therefore the exact sequence of kernels and cokernels takes the form 
\begin{equation*}
0 \to \rH^0(C,\xi)^\vee \otimes \cO_S(-H) \to \bL_S(\bR_C(\xi^{-1})) \to j_*(\xi^{-1}) \to 0,
\end{equation*}
and since~$\cRHom(j_*(\xi^{-1}), \cO_S) \cong j_*\xi \otimes \cO_S(H)[-1]$ by Grothendieck duality,
its dual sequence twisted by~$\cO_S(-H)$ takes the form
\begin{equation*}
0 \to \bL_S(\bR_C(\xi^{-1}))^\vee \otimes \cO_S(-H) \to \rH^0(C,\xi) \otimes \cO_S \to j_*\xi \to 0.
\end{equation*}
As we checked before, $\rH^0(S, \bL_S(\bR_C(\xi^{-1}))^\vee \otimes \cO_S(-H)) = 0$, hence the second arrow is the evaluation morphism, hence its kernel is isomorphic to the Lazarsfeld bundle~$\bL_S(\xi)$
and the required isomorphism~$\bL_S(\bR_C(\xi^{-1})) \cong \bL_S(\xi)^\vee \otimes \cO_S(-H)$ follows.
\end{proof}

\begin{proposition}\label{prop:lazarseld-mukai}
Let~$(S,H)$ be a quasipolarized~$K3$ surface of genus~$g = r \cdot s \ge 6$ 
and let~$C \subset S$ be a BNP-general curve in~$|H|$.
Assume~$s \ge r \in \{2,3\}$.
If~$S$ does not have special Mukai classes of type~$(r,s)$ 
then for any Mukai pair~$(\xi,\eta)$ of type~$(r,s)$ on~$C$
and any Mukai extension~\eqref{eq:mec} there is an exact sequence
\begin{equation}
\label{eq:blxi-bluc-bleta}
0 \to \bL_S(\xi)^\vee \otimes \cO_S(-H) \to \bL_S(\cG) \to \bL_S(\eta) \to 0.
\end{equation} 
of the corresponding Lazarsfeld bundles.
\end{proposition}

\begin{proof}
If~\eqref{eq:mec} is a Mukai extension, we have an exact sequence of global sections
\begin{equation*}
0 \to \rH^0(C, \xi)^\vee \to \rH^0(C, \cG) \to \rH^0(C, \eta) \to 0
\end{equation*}
(where we use Lemma~\ref{lem:rxi-h0} to identify the first term and Definition~\ref{def:mukai-extension} to prove exactness on the right).
Moreover, since~$\eta$ and~$\bR_C(\xi^{-1})$ are globally generated 
(by definition of a Mukai pair and Lemma~\ref{lem:rxivee-h0}, respectively), 
we conclude that~$\cG$ is also globally generated, hence we have a commutative diagram
\begin{equation}
\label{eq:big-diagram}
\vcenter{\xymatrix@R=3ex{
&
0 \ar[d] & 
0 \ar[d] & 
0 \ar[d] 
\\
0 \ar[r] & 
\bL_S(\xi)^\vee \otimes \cO_S(-H) \ar[r] \ar[d] & 
\bL_S(\cG) \ar[r] \ar[d] & 
\bL_S(\eta) \ar[r] \ar[d] & 
0 
\\
0 \ar[r] & 
\rH^0(C, \xi)^\vee \otimes \cO_S \ar[r] \ar[d] & 
\rH^0(C, \cG) \otimes \cO_S \ar[r] \ar[d] & 
\rH^0(C, \eta) \otimes \cO_S \ar[r] \ar[d] & 
0
\\
0 \ar[r] & 
j_*\bR_C(\xi^{-1}) \ar[r] \ar[d] & 
j_*\cG \ar[r] \ar[d] & 
j_*\eta \ar[r] \ar[d] & 
0
\\
& 0 & 0 & 0 
}}
\end{equation}
with exact rows and columns, where the first term of the left column is identified in Lemma~\ref{lem:blbrxi}.
Now the top row of the diagram gives~\eqref{eq:blxi-bluc-bleta}.
\end{proof}

From now on we work under the assumption that~$\Pic(S) = \ZZ \cdot H$.

\begin{lemma}\label{lem:unstable-extensions}
Assume~$\Pic(S) = \ZZ \cdot H$ and~$g = r \cdot s \ge 6$ with~$s \ge r \in \{2,3\}$.
Let~$(\xi,\eta)$ be a Mukai pair of type~$(r,s)$ on a curve~$C$ in~$|H|$ and let
\begin{equation}
\label{eq:cf-filtration}
0 \to \bL_S(\xi)^\vee \otimes \cO_S(-H) \to \cF \to \bL_S(\eta) \to 0
\end{equation} 
be an exact sequence.
If the sheaf~$\cF$ is not $H$-Gieseker stable then the sequence splits.
\end{lemma}

\begin{proof}
By Lemma~\ref{lem:lazarsfeld-bundle} and Proposition~\ref{prop:Picardrankonestable}, the subbundle~$\bL_S(\xi)^\vee \otimes \cO_S(-H)$ and quotient bundle~$\bL_S(\eta)$ of~$\cF$ are stable with Mukai vectors
\begin{align}
\label{eq:rv-subfactors}
&\rv(\bL_S(\xi)^\vee \otimes \cO_S(-H)) = (r,-(r - 1)H,s + (rs - 1)(r - 2)), 
&
&\rv(\bL_S(\eta)) = (s,-H,r)
\intertext{and slopes}
\notag
&\upmu_H(\bL_S(\xi)^\vee \otimes \cO_S(-H)) = \tfrac1r - 1,
&
&\upmu_H(\bL_S(\eta)) = -\tfrac1s.
\end{align}
Note that~$\tfrac1r - 1 \le -\tfrac12 < - \tfrac1s$ because~$r \ge 2$ and~$s \ge 3$.

Assume~$\cF$ is not $H$-Gieseker stable.
Consider the Harder--Narasimhan filtration of a sheaf~$\cF$
and refine it (via a Jordan--H\"older filtration of every semistable factor)
to a filtration with $H$-Gieseker stable factors~$\cF_1,\dots,\cF_m$.
We will show that~$m = 2$, $\cF_1 = \bL_S(\eta)$, and~$\cF_2 = \bL_S(\xi)^\vee \otimes \cO_S(-H)$, 
so that this filtration is opposite to filtration~\eqref{eq:cf-filtration}.

First, \eqref{eq:cf-filtration} implies that the slopes of~$\cF_i$ are bounded by the slopes of~$\bL_S(\xi)^\vee \otimes \cO_S(-H)$ and~$\bL_S(\eta)$:
\begin{equation}
\label{eq:slopes-cfi}
-\tfrac1s \ge \upmu_H(\cF_1) \ge \dots \ge \upmu_H(\cF_m) \ge \tfrac1r - 1.
\end{equation} 
Furthermore, since~$\Pic(S) = \ZZ \cdot H$, we can write~$\rv(\cF_i) = (x_i, - y_iH,  z_i)$. 
Then~$\upmu_H(\cF_i) \le -\tfrac1s < 0$ implies~$y_i > 0$ and~$\upmu_H(\cF_1) \ge \tfrac1r - 1 > -1$ implies~$x_i > y_i$.
Moreover, a combination of~\eqref{eq:rv-subfactors} and~\eqref{eq:slopes-cfi} gives the relations
\begin{equation*}
\sum x_i = r + s,
\qquad 
\sum y_i = r,
\qquad 
\sum z_i = (r-1)^2s + 2,
\qquad 
x_i > y_i{} \ge 1.
\end{equation*}
The inequalities~$x_i > y_i \ge 1$ with the equality~$\sum y_i = r$ imply, in particular, that
\begin{equation}
\label{eq:x-2}
m \le r
\qquad\text{and}\qquad
x_1, \dots, x_m \ge 2.
\end{equation} 
On the other hand, by Riemann--Roch we have~$\upchi(\cF_i,\cF_i) = 2x_iz_i - (2rs - 2)y_i^2$, and since the sheaves~$\cF_i$ are stable, they are simple, hence~$\upchi(\cF_i,\cF_i) \le 2$, hence
\begin{equation*}
\label{eq:xyz-9}
x_iz_i \le (rs - 1)y_i^2 + 1,
\qquad 
1 \le i \le m.
\end{equation*}
Dividing the $i$-th inequality by~$x_i$ and summing them up, we obtain
\begin{equation}
\label{eq:xyz-inequality}
(r-1)^2s + 2 = \sum z_i \le \sum \tfrac{(rs - 1)y_i^2 + 1}{x_i}. 
\end{equation}

From now on we assume~$r \in \{2,3\}$.
If~$m = 3$ then~\eqref{eq:x-2} implies~$r = 3$ and~$y_i = 1$ for all~$i$, 
hence the right-hand side of~\eqref{eq:xyz-inequality} equals~$3s \sum \tfrac1{x_i}$, and we obtain
\begin{equation}
\label{eq:sum-1x-inequality}
\sum \tfrac1{x_i} \ge 
\tfrac{4s + 2}{3s} =
\tfrac43 + \tfrac2{3s} > 
\tfrac43.
\end{equation}
Since~$x_i \ge 2$ by~\eqref{eq:x-2}, we must have~$x_1 = x_2 = x_3 = 2$, but then~$s = 3$, and~\eqref{eq:sum-1x-inequality} fails.
Therefore, we must have~$m = 2$; we assume this from now on.

As~$m = 2$, the inequality~\eqref{eq:xyz-inequality} can be rewritten as
\begin{equation*}
(rs - 1)(y_1^2x_2 + y_2^2x_1) + (x_1 + x_2) \ge ((r-1)^2s + 2)x_1x_2.
\end{equation*}
Since~$r \in \{2,3\}$, one of the~$y_i$ is~$1$ and the other is~$r - 1$.
If we denote the~$x_i$ that corresponds to~$y_i = 1$ by~$x$, so that the other is~$r + s - x$, 
then~$y_1^2x_2 + y_2^2x_1 = (r-1)^2 x + (r + s - x)$ and the above inequality takes the form
\begin{equation*}
f(x) \coloneqq (rs - 1)\Big((r-1)^2 x + (r + s - x)\Big) + (r + s) - \Big((r-1)^2s + 2\Big)x(r + s - x) \ge 0.
\end{equation*}
Note that
\begin{align*}
f(s) &= (rs - 1)((r - 1)^2 s + r) + (r + s) - ((r-1)^2s + 2)rs = 0,
\\
f(2) &= (rs - 1)(2(r-1)^2 + (r + s - 2)) + (r + s) - 2((r-1)^2s + 2)(r + s - 2)
\\
&= -s^2(r - 2)(2r - 1) - 5rs(r -2) - 2(r^2 - 4) \le 0.
\end{align*}
and the second inequality is strict unless~$r = 2$.
On the other hand, the function~$f$ is convex, hence the inequality~$f(x) \ge 0$ implies
\begin{equation*}
x \ge s
\qquad\text{or}\qquad
x \le 2
\end{equation*}
and the second inequality is strict unless~$r = 2$.
Now, if~$x \ge s$ the slope of~$\cF_i$ is greater or equal than~$-\tfrac1s$, if~$x \le 1$ 
then the slope of~$\cF_i$ is less than~$\tfrac1r - 1$, and if~$x = 2 = r$, it is equal to~$\tfrac1r - 1$.
In either case, \eqref{eq:slopes-cfi} implies that the slopes of the~$\cF_i$ are~$-\tfrac1s$ and~$\tfrac1r - 1$.
Since these coincide with the slopes of the stable bundles~$\bL_S(\eta)$ and~$\bL_S(\xi)^\vee \otimes \cO_S(-H)$, 
we conclude that 
\begin{equation*}
\cF_1 \cong \bL_S(\eta)
\qquad\text{and}\qquad 
\cF_2 \cong \bL_S(\xi)^\vee \otimes \cO_S(-H).
\end{equation*}
As we noticed above, $\upmu_H(\cF_1) = -\tfrac1s > \tfrac1r - 1 = \upmu_H(\cF_2)$,
and since both~$\cF_1$ and~$\cF_2$ are stable, we have~$\Hom(\cF_1,\cF_2) = 0$, hence
the exact sequence~\eqref{eq:cf-filtration} splits.
\end{proof}

Combining the above observations we obtain the main result of this section.

\begin{theorem}
\label{thm:mukai-extension}
Let~$(S,H)$ be a polarized $K3$ surface of genus~$g = r \cdot s \ge 6$ with~$\Pic(S) = \ZZ \cdot H$.
If~$C \in |H|$ is a BNP-general curve, $(\xi,\eta)$ is a Mukai pair of type~$(r,s)$ on~$C$ with~$s \ge r \in \{2,3\}$, and~$\epsilon \in \Ext^1(\eta, \bR_C(\xi^{-1}))$ is a Mukai extension class then~$\epsilon = \epsilon_{\bL_S(\xi)}$ is the Lazarsfeld class.
In particular, the Mukai extension class is unique.
\end{theorem}

\begin{proof}
Let~$\epsilon \ne 0$ be a Mukai extension class and let~\eqref{eq:mec} be the corresponding extension.
By Proposition~\ref{prop:lazarseld-mukai} we obtain the exact sequence~\eqref{eq:blxi-bluc-bleta}.
Now a simple Riemann--Roch computation shows that~$\upchi(\bL_S(\cG), \bL_S(\cG)) > 2$, hence the bundle~$\bL_S(\cG)$ is not stable.
Applying Lemma~\ref{lem:unstable-extensions} we conclude that~\eqref{eq:blxi-bluc-bleta} splits, hence
\begin{equation*}
\bL_S(\cG) \cong \big(\bL_S(\xi)^\vee \otimes \cO_S(-H)\big) \oplus \bL_S(\eta).
\end{equation*}
Let~$\psi \colon \bL_S(\eta) \to \bL_S(\cG)$ be a splitting of the projection from~\eqref{eq:blxi-bluc-bleta}.
Consider the upper right square 
\begin{equation*}
\xymatrix{
\bL_S(\cG) \ar[r] \ar[d] &
\bL_S(\eta) \ar@/^/@{-->}[l]^\psi \ar@{..>}[ld]^(.4){\tilde\psi} \ar[d]
\\
\rH^0(C, \cG) \otimes \cO_S \ar[r] &
\rH^0(C, \eta) \otimes \cO_S 
}
\end{equation*}
of the diagram~\eqref{eq:big-diagram}, where~$\tilde\psi$ is defined 
as the composition~$\bL_S(\eta) \xrightarrow{\ \psi\ } \bL_S(\cG) \hookrightarrow \rH^0(C, \cG) \otimes \cO_S$.
Let~$\cF \coloneqq \Coker(\tilde\psi)$.
It is easy to see that it fits into a commutative diagram with exact rows
\begin{equation}
\label{eq:small-diagram}
\vcenter{\xymatrix{
0 \ar[r] & 
\bL_S(\xi)^\vee \otimes \cO_S(-H) \ar[r] \ar[d] & 
\cF \ar[r] \ar@{=}[d] & 
j_*\cG \ar[r] \ar[d] & 
0
\\
0 \ar[r] & 
\rH^0(C, \xi)^\vee \otimes \cO_S \ar[r] & 
\cF \ar[r] & 
j_*\eta \ar[r] & 
0
}}
\end{equation}
where the left and right vertical arrows coincide with the arrows in~\eqref{eq:big-diagram}.
If the bottom row of~\eqref{eq:small-diagram} splits, then the composition~$j_*\eta \to \cF \to j_*\cG$ 
provides a splitting of the morphism~$j_*\cG \to j_*\eta$ from the bottom row of~\eqref{eq:big-diagram}, 
hence also a splitting of~\eqref{eq:mec}, which contradicts to the assumption~$\epsilon \ne 0$; 
therefore, the bottom row of~\eqref{eq:small-diagram} does not split.

Now, pulling back the bottom row of~\eqref{eq:small-diagram} to~$C$, we obtain an exact sequence
\begin{equation*}
0 \to \rL_1j^*\cF \to \eta(-K_C) \to \rH^0(C, \xi)^\vee \otimes \cO_C \to j^*\cF \to \eta \to 0.
\end{equation*}
The natural isomorphism~$\Ext^1(j_*\eta, \rH^0(C, \xi)^\vee \otimes \cO_S) \cong \Hom(\eta(-K_C), \rH^0(C, \xi)^\vee \otimes \cO_C)$ shows that the morphism~$\eta(-K_C) \to \rH^0(C, \xi)^\vee \otimes \cO_C$ in it is nontrivial, hence it is injective, hence~$\rL_1j^*\cF = 0$.
Therefore, using the bottom row of~\eqref{eq:small-diagram} to compute the invariants of~$j^*\cF$, we obtain
\begin{equation*}
\rank(j^*\cF) = \rank(\cF) = r = \rank(\cG)
\qquad\text{and}\qquad 
\rc_1(j^*\cF) = \rc_1(\cF)\vert_C = H\vert_C = K_C = \rc_1(\cG).
\end{equation*}
Now, restricting the top row of~\eqref{eq:small-diagram} to~$C$, we obtain
\begin{equation*}
0 \to \cG(-K_C) \to j^*\bL_S(\xi)^\vee \otimes \cO_C(-K_C) \to j^*\cF \to \cG \to 0.
\end{equation*}
The last arrow in it is an epimorphism of sheaves with the same rank and degree, hence it is an isomorphism.
Therefore, the first arrow is also an isomorphism, i.e., $j^*\bL_S(\xi)^\vee \cong \cG$.

Finally, since~$\Hom(\bR_C(\xi^{-1}), \eta) = 0$ by Lemma~\ref{lem:rxi-h0}, it follows from~\eqref{eq:mec} that the space
\begin{equation*}
\Hom(j^*\bL_S(\xi)^\vee, \eta) \cong \Hom(\cG, \eta)
\end{equation*}
is 1-dimensional, hence the composition~$j^*\bL_S(\xi)^\vee \xrightiso{} \cG \twoheadrightarrow \eta$ of the constructed isomorphism and the surjection of~\eqref{eq:mec}
coincides with the map~$j^*(\phi)$ in sequence~\eqref{eq:rxi-blxi}, hence~$\epsilon = \epsilon_{\bL_S(\xi)}$. 
\end{proof}

%%%%%%%%%%%%%%%%%%%%%%%%%

\section{Mukai bundles on Fano threefolds}
\label{sec:fano}

In this section we prove Theorem~\ref{thm:Main}. 
In Section~\ref{ss:mb-sing} we consider the more general situation 
of a prime Fano threefold with factorial terminal singularities,
and prove in this case Theorem~\ref{thm:mb-x}, a weaker version of Theorem~\ref{thm:Main},
and in Section~\ref{ss:mb-smooth} we consider the case where~$X$ is smooth.

Recall that a normal variety~$X$ is {\sf factorial} if any Weil divisor on~$X$ is Cartier;
in particular, the canonical divisor is Cartier. 
Note also that terminal singularities are Cohen--Macaulay, 
hence any variety with factorial terminal singularities is Gorenstein.
Finally, recall that a Fano threefold with factorial terminal singularities is {\sf prime} 
if its anticanonical class
\begin{equation*}
H \coloneqq -K_X,
\end{equation*}
is the ample generator of~$\Pic(X)$.
Recall that the genus~$g$ of~$X$ is defined by the equality
\begin{equation*}
H^3 = 2g - 2.
\end{equation*}
We assume in this section that~$g \ge 4$.

\begin{definition}\label{def:mb-x}
Let~$X$ be a prime Fano threefold~$X$ with factorial terminal singularities 
of genus~$g = r \cdot s$ with~$r,s \ge 2$.
A {\sf Mukai sheaf}~$\cU_X$ on~$X$ of type~$(r,s)$ is a maximal Cohen--Macaulay sheaf such that
\begin{aenumerate}
\item 
\label{it:v}
$\rank(\cU_X) = r$, $\rc_1(\cU_X) = -H$, 
\item 
\label{it:h}
$\rH^\bullet(X,\cU_X) = 0$, and
\item 
\label{it:dual}
the dual sheaf~$\cU_X^\vee$ is globally generated with~$\dim \rH^0(X, \cU_X^\vee)=r+s$ and~$\rH^{>0}(X, \cU_X^\vee) = 0$.
\end{aenumerate}
\end{definition}

Note that a Mukai sheaf satisfies all the properties of the bundle~$\cU_r$ from Theorem~\ref{thm:Main},
except for local freeness (which is replaced by the maximal Cohen--Macaulay property) 
and the equality~$\Ext^\bullet(\cU_X,\cU_X) = \kk$.

It is clear that Mukai sheaves with~$r = 1$ (hence~$s = g$) do not exist;
indeed, property~\ref{it:v} implies that~$\cU_X \cong \cO_X(-H) \cong \cO_X(K_X)$, and then property~\ref{it:h} fails.
For this reason we exclude the case~$r = 1$ and the symmetric case~$s = 1$ from the definition.

In contrast to Definition~\ref{def:mb-s} we do not include stability in the definition of a Mukai sheaf on~$X$, 
because it is automatic.

\begin{lemma}
\label{lem:mbx-mbs}
Let~$X$ be a prime Fano threefold~$X$ with factorial terminal singularities and~$g \ge 4$.
The anticanonical linear system~$|H|$ is very ample, $h^0(X, \cO_X(H)) = g + 2$, $h^{>0}(X, \cO_X(H)) = 0$,
and a very general anticanonical divisor in~$X$ 
is a smooth $K3$ surface~$S$ with~$\Pic(S) = \ZZ \cdot H_S$, where~$H_S \coloneqq H\vert_S$.

Moreover, if~$g = r \cdot s$ with~$r,s \ge 2$ and~$\cU_X$ is a Mukai sheaf of type~$(r,s)$ on~$X$ 
then~$\cU_X$ is $H$-Gieseker stable
and~$\cU_X\vert_S$ is a Mukai bundle on~$S$ for any smooth~$S \subset X$ with~$\Pic(S) = \ZZ \cdot H_S$.

Finally, if~$\cU_X$ is a Mukai sheaf of type~$(r,s)$ then
\begin{equation*}
\cU_X^\perp \coloneqq \Ker \Big(\rH^0(X, \cU_X^\vee) \otimes \cO_X \to \cU_X^\vee\Big)
\end{equation*}
is a Mukai sheaf of type~$(s,r)$ on~$X$.
\end{lemma}

\begin{proof}
The very ampleness of~$|H|$ is established in~\cite[Theorem~4.2]{Prokhorov:rationality} and~\cite[Theorem~4.4]{KP23},
the cohomology of~$\cO_X(H)$ can be computed by the Kawamata--Viehweg vanishing and Riemann--Roch
and the smoothness of a general anticanonical divisor is proved in~\cite[Theorem~2.9]{Mella}. 
The pushforward of~\mbox{$\cO_X(K_X) = \cO_X(-H)$} along the anticanonical embedding~$X \hookrightarrow \P^{g+1}$ 
is trivial after tensoring with~$\cO(1)$, 
and therefore the assumptions of~\cite[Theorem~1]{RS09} are satisfied; 
this implies that the very general anticanonical divisor~$S$ has Picard group generated by~$H_S$.

Now let~$i \colon S \hookrightarrow X$ be a very general anticanonical divisor and~$\cU_S \coloneqq \cU_X\vert_S$.
It is a maximal Cohen--Macaulay sheaf on~$S$, and since~$S$ is smooth, 
$\cU_S$ is locally free by the Auslander--Buchsbaum formula.
The restriction sequence
\begin{equation*}
0 \to \cU_X^\vee(-H) \to \cU_X^\vee \to i_*\cU_S^\vee \to 0
\end{equation*} 
combined with Definition~\ref{def:mb-x}\ref{it:h} and~\ref{it:dual} and Serre duality on~$X$ 
implies that~\mbox{$h^0(\cU_S^\vee) = r + s$}, \mbox{$\rH^{>0}(S, \cU_S^\vee) = 0$}, 
and~$\cU_S^\vee$ is globally generated.
Combining this with Definition~\ref{def:mb-x}\ref{it:v} and Riemann--Roch,
we conclude that~$\rv(\cU_S) = (r, -H_S, s)$.
To show that~$\cU_S$ is a Mukai bundle it remains to check its stability, which is standard.
Indeed, since~$\Pic(S) = \ZZ \cdot H_S$, the slope of~$\cU_S$ is the maximal negative slope for vector bundles of rank at most~$r$,
hence a saturated destabilizing subsheaf~$\cF \subset \cU_S$ must have nonnegative slope.
But global generation of~$\cU_S^\vee$ shows that~$\cU_S \subset \cO_S^{\oplus n}$, hence~$\cF \subset \cO_S^{\oplus n}$,
and Lemma~\ref{lem:chern0}\ref{it:mono} implies~$\cF \cong \cO_S^{\oplus m}$, hence~$H^0(S,\cU_S) \ne 0$,
which contradicts the vanishing of~$\rH^2(S, \cU_S^\vee)$ proved above.

Stability of~$\cU_X$ easily follows:
if~$\cF \subset \cU_X$ is a saturated destabilizing subsheaf 
then~$\cF\vert_S \subset \cU_S$ is a destabilizing subsheaf for~$\cU_S$, 
which is impossible because, as we observed above, $\cU_S$ is stable.

The last part is obvious from the definition
(for the maximal Cohen--Macaulay property of~$\cU_X^\perp$ see~\cite[Lemma~4.2.2]{Buch}).
\end{proof}

\subsection{Mukai sheaves on singular Fano threefolds}
\label{ss:mb-sing}

Throughout this subsection we assume that~$X$ is a prime Fano threefold 
with factorial terminal singularities of genus~$g \ge 4$.
The main result of this section is the following theorem (it will be used to deduce Theorem~\ref{thm:Main}).

\begin{theorem}\label{thm:mb-x}
If~$X$ is a prime Fano threefold 
with factorial terminal singularities of genus~$g = r \cdot s$ with~$r,s \ge 2$
then there exists a Mukai sheaf\/~$\cU_X$ of type~$(r,s)$ on~$X$,
which is unique if~$g \ge 6$.
\end{theorem}

To prove this, we will show that in any nice pencil of anticanonical divisors on~$X$ 
there is a du Val K3 surface~$S_0$ such that its minimal resolution~$\tS_0$ carries a special Mukai class (Definition~\ref{def:xi-big}).
In the next lemma we explain what do we mean by ``nice pencils''.

\begin{lemma}
\label{lem:pencil}
There is a pencil~$\cS \subset |H|$ of anticanonical divisors in~$X$ such that
\begin{aenumerate}
\item
\label{it:base-bnp}
the base locus~$C \coloneqq \Bs(\cS)$ is a smooth BNP-general curve;
\item
\label{it:pic1}
a very general member~$S_t$ of~$\cS$  has $\Pic(S_t) = \ZZ \cdot H_{S_t}$;
\item
\label{it:smooth}
a general member~$S_t$ of~$\cS$ is smooth;
\item 
\label{it:duval}
any singular member~$S_t$ of~$\cS$ has at worst du Val singularities.
\end{aenumerate}
\end{lemma}

\begin{proof}
By Lemma~\ref{lem:mbx-mbs} the anticanonical morphism of~$X$ is a closed embedding $X \subset \P^{g + 1}$.
Consider the dual projective space~$\check\P^{g+1}$,
the projectively dual variety~$X^\vee \subset \check\P^{g+1}$, 
and the hyperplanes~\mbox{$\cP_1, \dots, \cP_n \subset \check\P^{g+1}$} that correspond to the singular points~$p_1,\dots,p_n \in X$
(recall that terminal singularities of threefolds are isolated).
Since~$p_i \in X$ is a terminal Gorenstein singularity, 
it is a compound du Val singularity by~\cite[Main Theorem]{Reid},
and therefore a general hyperplane section of~$X$ containing~$p_i$ is du Val at~$p_i$ by~\cite[Corollary~2.10]{Reid:canonical}.
Let~$\cP^{\mathrm{ndV}}_i \subset \cP_i$ be the closure of the subset that corresponds to hyperplanes~$\P^g \subset \P^{g+1}$ 
that contain the point~$p_i$ and the singularity of~$X \cap \P^g$ at~$p_i$ is worse than du Val.
Then the subsets
\begin{equation}
\label{eq:bad-loci}
\Sing(X^\vee) \subset \check\P^{g+1}
\qquad\text{and}\qquad 
\cP^{\mathrm{ndV}}_i \subset \check\P^{g+1}
\end{equation}
all have codimension at least~$2$ in~$\check\P^{g+1}$.

On the other hand, let~$S \subset X$ be a very general anticanonical divisor.
Then~$S$ is a smooth K3 surface with~$\Pic(S) = \ZZ \cdot H_S$ by Lemma~\ref{lem:mbx-mbs}, 
in particular every curve in~$|H_S|$ is reduced and irreducible,
hence by Theorem~\ref{thm:bnpc}\ref{it:general-bnp} a general curve~$C \in |H_S|$ is BNP-general.
Now consider a general line in~$\check\P^{g+1}$ passing through the point that corresponds to~$S$ 
and avoiding all subsets in~\eqref{eq:bad-loci} and let~$\cS \subset |H|$ be the corresponding pencil.
Then it is clear that all the properties of the lemma are satisfied.
\end{proof}

Let~$\cS$ be a pencil of hyperplane sections of~$X$ as in Lemma~\ref{lem:pencil}.
Consider the blowup
\begin{equation*}
\pi \colon \tX \coloneqq \Bl_C(X) \to X
\end{equation*}
of the base locus~$C \subset X$ of~$\cS$ and denote by
\begin{equation*}
p \colon \tX \to \P^1
\end{equation*}
the induced morphism.
The exceptional divisor~$E$ of~$\pi$ has the form
\begin{equation*}
E \cong C \times \P^1,
\end{equation*}
the restriction of~$\pi$ to~$E$ is the projection to~$C$, the restriction of~$p$ to~$E$ is the projection to~$\P^1$, and the normal bundle of~$E$ is
\begin{equation}
\label{eq:ee}
\cO_E(E) \cong \cO_C(K_C) \boxtimes \cO_{\P^1}(-1).
\end{equation}
We denote by~$\iota \colon E \hookrightarrow \tX$ the embedding.

Since~$C$ is BNP-general (by Lemma~\ref{lem:pencil}\ref{it:base-bnp}), 
Lemma~\ref{lem:xi-eta} shows that~$C$ has a Mukai pair~$(\xi,\eta)$ of type~$(r,s)$.
Consider {\sf the relative Lazarsfeld bundle}~$\bL_{\tX/\P^1}(\xi)$ on~$\tX$ defined by the exact sequence
\begin{equation}
\label{eq:tce}
0 \to \bL_{\tX/\P^1}(\xi) \xrightarrow{\quad} \rH^0(C,\xi) \otimes \cO_\tX \xrightarrow{\ \ev\ } \iota_*(\xi \boxtimes \cO_{\P^1}) \to 0.
\end{equation} 
Below we investigate the pullbacks of~$\bL_{\tX/\P^1}(\xi)$ to the fibers of~$p$ and their minimal resolutions.

Note that the fiber~$S_t \coloneqq p^{-1}(t)$ of~$p$ over a point~$t \in \P^1$ 
is a polarized K3 surface with du Val singularities and its minimal resolution $\sigma_t \colon \tS_t \to S_t$
is a smooth quasipolarized K3 surface (where the quasipolarization~$H_{\tS_t}$ 
is the pullback of~$H_{S_t}$); sometimes we will think of~$\sigma_t$ as a map~$\tS_t \to \tX$.
Note also that the curve~$S_t \cap E = C$ is contained in the smooth locus of~$S_t$, 
hence the embedding~$C \hookrightarrow S_t$ lifts to an embedding~$j_t \colon C \hookrightarrow \tS_t$, 
so that~$\sigma_t(j_t(C)) = S_t \cap E \subset \tX$.

Pulling back the sequence~\eqref{eq:tce} to~$\tS_t$ we obtain an exact sequence
\begin{equation*}
	0 \to \sigma_t^*\bL_{\tX/\P^1}(\xi) \xrightarrow{\quad} \rH^0(C,\xi) \otimes \cO_{\tS_t} \xrightarrow{\ \ev\ } {j_t}_*\xi \to 0,
\end{equation*}
which coincides with the defining sequence of the Lazarsfeld bundle~$\bL_{\tS_t}(\xi)$ on the surface~$\tS_t$ 
associated with the curve~$j_t(C)$ and the Mukai pair~$(\xi,\eta)$ on it, i.e., we obtain an isomorphism
\begin{equation}
\label{eq:laz-laz}
\sigma_t^*\bL_{\tX/\P^1}(\xi) \cong \bL_{\tS_t}(\xi).
\end{equation}

Recall that special Mukai classes on quasipolarized K3 surfaces were introduced in Definition~\ref{def:xi-big}. We will use the relative Lazarsfeld bundle to prove the following crucial observation.

\begin{proposition}
\label{prop:special-s}
Assume~$g = r \cdot s$ with~$s \ge r \in \{2,3\}$.
Let~$\cS \in |H|$ be a pencil as in Lemma~\textup{\ref{lem:pencil}}.
There is a point~\mbox{$t_0 \in \P^1$} 
such that the corresponding surface~$\tS_{t_0}$ is BN-general 
and has a special Mukai class~\mbox{$\Xi \in \Pic(\tS_{t_0})$} of type~$(r,s)$.
\end{proposition}

\begin{proof}
We restrict the sequence~\eqref{eq:tce} to~$E \cong C \times \P^1$.
Using~\eqref{eq:ee} we obtain an exact sequence
\begin{equation*}
0 \to 
\xi(-K_C) \boxtimes \cO_{\P^1}(1) \xrightarrow{\quad} 
\iota^*\bL_{\tX/\P^1}(\xi) \xrightarrow{\quad} 
\rH^0(C,\xi) \otimes (\cO_C \boxtimes \cO_{\P^1}) \xrightarrow{\ \ev\ } 
\xi \boxtimes \cO_{\P^1} \to 
0.
\end{equation*}
Its last morphism is the pullback of~$\ev \colon \rH^0(C,\xi) \otimes \cO_C \to \xi$ along~$p$, 
therefore its kernel is isomorphic to~$\bR_C(\xi^{-1})^\vee \boxtimes \cO_{\P^1}$, 
where~$\bR_C(\xi^{-1})$ is defined in~\eqref{eq:def-rxi}.
Thus, we obtain an exact sequence
\begin{equation*}
0 \to 
\xi(-K_C) \boxtimes \cO_{\P^1}(1) \xrightarrow{\quad} 
\iota^*\bL_{\tX/\P^1}(\xi) \xrightarrow{\quad} 
\bR_C(\xi^{-1})^\vee \boxtimes \cO_{\P^1} \to 
0
\end{equation*}
and dualizing it, we obtain an exact sequence
\begin{equation*}
0 \to 
\bR_C(\xi^{-1}) \boxtimes \cO_{\P^1} \to
\iota^*\bL_{\tX/\P^1}(\xi)^\vee \to
\eta \boxtimes \cO_{\P^1}(-1) \to
0.
\end{equation*}
Consider its extension class 
\begin{equation*}
\epsilon \in 
\Ext^1(\eta \boxtimes \cO_{\P^1}(-1), \bR_C(\xi^{-1}) \boxtimes \cO_{\P^1}) \cong 
\Ext^1(\eta, \bR_C(\xi^{-1})) \otimes \rH^0(\P^1,\cO_{\P^1}(1)).
\end{equation*}
For each~$t \in \P^1$ we can evaluate~$\epsilon$ at~$t$ and obtain an extension class~$\epsilon(t) \in \Ext^1(\eta, \bR_C(\xi^{-1}))$.
Using~\eqref{eq:laz-laz} we obtain an isomorphism
\begin{equation*}
(\iota^*\bL_{\tX/\P^1}(\xi))\vert_{C \times \{t\}} \cong
j_{t}^*\sigma_{t}^*\bL_{\tX/\P^1}(\xi) \cong
j_{t}^*\bL_{\tS_{t}}(\xi).
\end{equation*}
It follows that the extension of~$\eta$ by~$\bR_C(\xi^{-1})$ corresponding to the class~$\epsilon(t)$ 
is isomorphic to~$j_{t}^*\bL_{\tS_{t}}(\xi)$,
which means that~$\epsilon(t)$ is the Lazarsfeld extension class for all~$t$, see Definition~\ref{def:lec}.

We claim that there is a unique point~$t_0 \in \P^1$ such that~$\epsilon(t_0) = 0$.
Indeed, for a very general~$t$ we have~$\Pic(\tS_t) = \Pic(S_t) = \ZZ \cdot H_{\tS_t}$ 
(by Lemma~\ref{lem:pencil}\ref{it:pic1}, \ref{it:smooth}), 
in particular~$\tS_t$ does not have special Mukai classes.
Therefore, for such~$t$ Proposition~\ref{prop:lec} shows 
that~$\epsilon(t)$ is a Mukai extension class as in Definition~\ref{def:mukai-extension},
and then Theorem~\ref{thm:mukai-extension} proves that the classes~$\epsilon(t)$ for different very general~$t$ are all proportional.
Thus, we can write
\begin{equation*}
\epsilon = \epsilon_0 \otimes f_0,
\end{equation*}
where~$\epsilon_0 \in \Ext^1(\eta, \bR_C(\xi^{-1}))$ and~$f_0 \in \rH^0(\P^1,\cO_{\P^1}(1))$ is a linear function on~$\P^1$.
Hence there is a unique point~$t_0 \in \P^1$ such that~$f_0(t_0) = 0$, i.e., the Lazarsfeld extension class~$\epsilon(t_0)$ vanishes.
Applying Proposition~\ref{prop:lec} we conclude that~$\tS_{t_0}$ must have a special Mukai class~$\Xi$ of type~$(r,s)$.
Finally, $\tS_{t_0}$ is BN-general by Theorem~\ref{thm:bnpc}\ref{it:bnp-bn} 
because it contains the BN-general curve~$C$.
\end{proof}

\begin{remark*} 
The existence of the special member~$S_{t_0}$ in the pencil may look mysterious, 
but its existence has a simple a posteriori explanation 
in terms of the morphism~$X \to \Gr(r,r+s)$ from Theorem~\ref{thm:descriptions}. 
The restriction of this morphism to~$C$ is induced 
by the globally generated vector bundle~$\cU_X^\vee\vert_C = \bL_{S_t}(\xi)\vert_C^\vee$.
The exact sequence~\eqref{eq:rxi-blxi} implies that the composition
\begin{equation*}
\rH^0(C, \bR_C(\xi^{-1})) \otimes \cO_C \to 
\rH^0(C, \bL_{S_t}(\xi)\vert_C^\vee) \otimes \cO_C \to 
\bL_{S_t}(\xi)\vert_C^\vee \to 
\eta
\end{equation*}
vanishes, hence the image of the curve~$C$ in~$\Gr(r, r+s)$ 
is contained in the Schubert divisor of~$\Gr(r, r+s)$ 
associated with the subspace~$\rH^0(C, \bR_C(\xi^{-1})) \subset \rH^0(C, \bL_{S_t}(\xi)\vert_C^\vee)$ of dimension~$r$.
Then~$S_{t_0}$ is the intersection of~$X$ with the pullback of this Schubert divisor.

Along these lines one can show that there are bijections between the following three sets 
of cardinality~$\rN(r,s)$ (see Remark~\ref{rem:number-mp}): 
the set of Mukai pairs on~$C$, the set of Schubert divisors in~$\Gr(r,r+s)$ containing~$C$, 
and the set of members of the pencil~$\cS$ carrying a special Mukai class.
\end{remark*}

In Proposition~\ref{prop:special-s} we constructed an anticanonical divisor~$S_0 \subset X$
such that its minimal resolution~$\sigma \colon \tS_0 \to S_0$ is BN-general and has a special Mukai class~$\Xi$.
Recall from Corollary~\ref{cor:bs-h-xi} that~$\sigma_*\cO_{\tS_0}(\Xi)$ 
is a maximal Cohen--Macaulay globally generated sheaf on~$S_0$.

Now we define a coherent sheaf~$\bL_X(\Xi)$ on~$X$ from the exact sequence
\begin{equation}
\label{eq:laz-x}
0 \to \bL_X(\Xi) \xrightarrow{\quad} \rH^0(\tS_0, \cO_{\tS_0}(\Xi)) \otimes \cO_{X} \xrightarrow{\ \ev\ } \sigma_*\cO_{\tS_0}(\Xi) \to 0.
\end{equation} 
This is an incarnation of the Lazarsfeld bundle construction in dimension three rather than two.

\begin{corollary}
\label{cor:mb-x-existence}
If~$S_0 \subset X$ is an anticanonical divisor with du Val singularities 
such that its minimal resolution~$\sigma \colon \tS_0 \to S_0$ 
is BN-general and has a special Mukai class~$\Xi$ of type~$(r,s)$, 
then the sheaf~$\bL_X(\Xi)$ defined in~\eqref{eq:laz-x} is a Mukai sheaf on $X$.
\end{corollary}

\begin{proof}
Since the last term of~\eqref{eq:laz-x} is a maximal Cohen--Macaulay sheaf on a Cartier divisor in~$X$,
it follows that~$\bL_X(\Xi)$ is maximal Cohen--Macaulay on~$X$.
Further, properties~\ref{it:v} and~\ref{it:h} of Definition~\ref{def:mb-x}
follow immediately from the defining exact sequence~\eqref{eq:laz-x} and Corollary~\ref{cor:bs-h-xi}.

To prove property~\ref{it:dual} of Definition~\ref{def:mb-x}, we dualize~\eqref{eq:laz-x} to obtain an exact sequence
\begin{equation*}\label{eq:laz-x-dual}
0 \to 
\rH^0(\tS_0, \cO_{\tS_0}(\Xi))^\vee \otimes \cO_{X} \to 
\bL_X(\Xi)^\vee \to 
\sigma_*\cO_{\tS_0}(H - \Xi) \to 0,
\end{equation*}
where the last term is identified in Corollary~\ref{cor:bs-h-xi}. 
In the same corollary~$\sigma_*\cO_{\tS_0}(H - \Xi)$ is proved to be globally generated 
with~$h^0(\sigma_*\cO_{\tS_0}(H - \Xi)) = s$ and~$h^{>0}(\sigma_*\cO_{\tS_0}(H - \Xi)) = 0$.
Since~$h^{>0}(\cO_{X}) = 0$, it follows that~$\bL_X(\Xi)^\vee$ is globally generated, 
$h^0(\bL_X(\Xi)^\vee) = r + s$ and~$h^{>0}(\bL_X(\Xi)^\vee) = 0$.
\end{proof}

\begin{proof}[Proof of Theorem~\textup{\ref{thm:mb-x}}]
First, we prove the existence.

Assume~$s \ge r$.
Since Gorenstein Fano threefolds with terminal singularities are smoothable, 
their genus has the same bound~$g \le 12$ as for smooth prime Fano threefolds,
hence~$r \in \{2,3\}$.
Therefore, Corollary~\ref{cor:mb-x-existence} (whose hypotheses are satisfied by Proposition~\ref{prop:special-s}) applies, 
proving that~$\cU_X \coloneqq \bL_X(\Xi)$ is a Mukai sheaf on~$X$.

Now assume~$r > s$.
The above argument proves the existence of a Mukai sheaf of type~$(s,r)$,
and then Lemma~\ref{lem:mbx-mbs} implies the existence of a Mukai sheaf of type~$(r,s)$.

Now we prove the uniqueness.
So, assume~$g = r \cdot s \ge 6$ and let~$\cU_1$, $\cU_2$ be two Mukai sheaves of type~$(r,s)$ on~$X$.
Consider the pencil~$\cS = \{S_t\}_{t \in \P^1}$ constructed in Lemma~\ref{lem:pencil} with base curve~$C$.
By Lemma~\ref{lem:mbx-mbs} if~$t \in \P^1$ is very general, $\cU_i\vert_{S_t}$ are Mukai bundles, 
hence by Lemma~\ref{lem:mbs-unique} there is an isomorphism~$\phi_t \colon \cU_1\vert_{S_t} \xrightiso{} \cU_2\vert_{S_t}$, 
unique up to rescaling.
Restricting any of the isomorphisms~$\phi_t$ to~$C$, 
we obtain an isomorphism~$\phi_C \colon \cU_1\vert_C \xrightiso{}  \cU_2\vert_C$,
and since the bundles~$\cU_i\vert_C \cong (\cU_i\vert_{S_t})\vert_C$ are simple 
by Proposition~\ref{prop:lec} (it applies because~$g \ge 6$), 
it is also unique up to rescaling.
Therefore, we can fix an isomorphism~$\phi_C$ and rescale~$\phi_t$ in such a way 
that~$\phi_t\vert_C = \phi_C$ for all very general~$t$.

On the other hand, Lemmas~\ref{lem:mbx-mbs} and~\ref{lem:l-m-h} imply that the restriction maps
\begin{equation*}
\rH^0(X, \cU_i^\vee) \to \rH^0(S_t, \cU_i^\vee\vert_{S_t}) \to \rH^0(C, \cU_i^\vee\vert_C)
\end{equation*}
are isomorphism, hence we have a commutative diagram
\begin{equation*}
\xymatrix{
\rH^0(X, \cU_2^\vee) \ar[r]^-\sim & 
\rH^0(S_t, \cU_2^\vee\vert_{S_t}) \ar[d]^{\phi_t^\vee}_\cong \ar[r]^-\sim & 
\rH^0(C, \cU_2^\vee\vert_C) \ar[d]^{\phi_C^\vee}_\cong
\\
\rH^0(X, \cU_1^\vee) \ar[r]^-\sim & 
\rH^0(S_t, \cU_1^\vee\vert_{S_t}) \ar[r]^-\sim & 
\rH^0(C, \cU_1^\vee\vert_C).
}
\end{equation*}
It follows that there is 
a unique isomorphism~$\rH^0(X, \cU_2^\vee) \xrightiso{} \rH^0(X, \cU_1^\vee)$ compatible with~$\phi_C^\vee$,
and, moreover, the diagram shows that it is compatible with~$\phi_t^\vee$ for all very general~$t$.

Now consider the composition of morphisms
\begin{equation*}
\psi \colon
\cU_2^\perp \hookrightarrow 
\rH^0(X, \cU_2^\vee) \otimes \cO_X \xrightiso{}
\rH^0(X, \cU_1^\vee) \otimes \cO_X \twoheadrightarrow 
\cU_1^\vee,
\end{equation*}
where~$\cU_2^\perp$ is the Mukai sheaf of type~$(s,r)$ associated with~$\cU_2$ in Lemma~\ref{lem:mbx-mbs}.
For each very general~$t \in \P^1$, restricting~$\psi$ to~$S_t$, we obtain a chain of maps
\begin{equation*}
\cU_2^\perp\vert_{S_t} \hookrightarrow 
\rH^0(S_t, \cU_2^\vee) \otimes \cO_{S_t} \xrightarrow{\ \phi_t^\vee\ }
\rH^0(S_t, \cU_1^\vee) \otimes \cO_{S_t} \twoheadrightarrow 
\cU_1^\vee\vert_{S_t},
\end{equation*}
and by functoriality of the evaluation morphism, its composition is zero.
Thus, the morphism~$\psi$ vanishes after restriction to~$S_t$ for all very general~$t$, hence it vanishes everywhere.
It follows that the composition~$\rH^0(X, \cU_2^\vee) \otimes \cO_X \xrightiso{} \rH^0(X, \cU_1^\vee) \otimes \cO_X \twoheadrightarrow \cU_1^\vee$
factors through an epimorphism~$\cU_1^\vee \twoheadrightarrow \cU_2^\vee$,
and since the sheaves~$\cU_1^\vee$ and~$\cU_2^\vee$ are torsion free of the same rank, it is an isomorphism.
\end{proof}

\subsection{The Mukai bundle}
\label{ss:mb-smooth}

In this section we show that if~$\g(X) \ge 6$ and the Mukai sheaf on~$X$ is locally free, 
it is exceptional (see Remark~\ref{rem:g4} for a discussion of the case~$g = 4$).
In particular, we prove Theorem~\ref{thm:Main}.

We will need the following simple observation.

\begin{lemma}
\label{lem:hilb-mbx}
If~$\cF$ is a vector bundle on~$X$ then~$\upchi(\cF, \cF) = \tfrac12\upchi(\cF\vert_S, \cF\vert_S)$.
In particular, if~$\cF\vert_S$ is numerically spherical, i.e., $\upchi(\cF\vert_S, \cF\vert_S) = 2$,
then~$\cF$ is numerically exceptional, i.e., $\upchi(\cF, \cF) = 1$.
\end{lemma}

\begin{proof}
Let~$i \colon S \hookrightarrow X$ be the embedding.
The short exact sequence~$0 \to \cF(-H) \to \cF \to i_* i^* \cF \to 0$
combined with Serre duality and adjunction gives
\begin{equation*}
\upchi(\cF, \cF) = \upchi(\cF, \cF(-H)) + \upchi(\cF, i_* i^* \cF) = 
	- \upchi(\cF, \cF) + \upchi(\cF\vert_S, \cF\vert_S),
\end{equation*}
which shows our claim.
\end{proof}

The following lemma reduces verification of some $\Ext$-vanishing on~$X$ to $\Ext$-vanishing on~$S$.

\begin{lemma}\label{lem:ext1-f1-f2}
Let~$\cF_1,\cF_2$ be vector bundles on~$X$ and let~$S \subset X$ be a divisor in~$|H|$ such that
\begin{equation*}
\Hom(\cF_1\vert_S, \cF_2\vert_S) = \kk
\qquad\text{and}\qquad 
\Ext^1(\cF_1\vert_S, \cF_2(-H)\vert_S) = 0.
\end{equation*}
Then~$\Hom(\cF_1, \cF_2) = \kk$ and~$\Ext^1(\cF_1, \cF_2(-kH)) = 0$ for all~$k \ge 1$.
\end{lemma}

\begin{proof}
Consider the Koszul complex for three general global sections of the sheaf~$\cO_S(H_S)$ 
(which is globally generated by Lemma~\ref{lem:mbx-mbs}):
\begin{equation*}
0 \to \cO_S(-3H_S) \to \cO_S(-2H_S)^{\oplus 3} \to \cO_S(-H_S)^{\oplus 3} \to \cO_S \to 0.
\end{equation*}
Tensoring it by~$\cF_2(H)\vert_S$ we obtain an exact sequence
\begin{equation*}
0 \to 
\cF_2(-2H)\vert_S \to
\cF_2(-H)\vert_S^{\oplus 3} \to
\cF_2\vert_S^{\oplus 3} \to
\cF_2(H)\vert_S \to
0
\end{equation*}
Since~$\Ext^1(\cF_1\vert_S, \cF_2(-H)\vert_S) = 0$, 
a simple spectral sequence implies that~$\Ext^1(\cF_1\vert_S, \cF_2(-2H)\vert_S)$ 
is a quotient of the space
\begin{equation*}
\Ker\left(\Hom(\cF_1\vert_S, \cF_2\vert_S)^{\oplus 3} \to \Hom(\cF_1\vert_S, \cF_2(H)\vert_S)\right).
\end{equation*}
Since~$\Hom(\cF_1\vert_S, \cF_2\vert_S) = \kk$, the kernel is nonzero 
only if~$\Hom(\cF_1\vert_S, \cF_2\vert_S)$ is annihilated by a section of~$\cO_S(H_S)$, 
which is impossible because the~$\cF_i$ are locally free.
Thus, $\Ext^1(\cF_1\vert_S, \cF_2(-2H)\vert_S) = 0$.
Twisting the same Koszul complex by~$-H$, $-2H$, and so on, and repeating the same argument, 
we see that~$\Ext^1(\cF_1\vert_S, \cF_2(-kH)\vert_S) = 0$ for all~$k \ge 1$.

Now consider the restriction exact sequence
\begin{equation*}
\Ext^1(\cF_1, \cF_2(-(k + 1)H)) \to \Ext^1(\cF_1, \cF_2(-kH)) \to \Ext^1(\cF_1\vert_S, \cF_2(-kH)\vert_S).
\end{equation*}
Its right term is zero for~$k \ge 1$ as we just showed, hence its first arrow must be surjective.
Thus, if~$\Ext^1(\cF_1, \cF_2(-kH)) \ne 0$ for some~$k \ge 1$, the same nonvanishing holds for all sufficiently large~$k$, 
which contradicts Serre vanishing because the~$\cF_i$ are locally free.
This proves the vanishing of~$\Ext^1$.

Similarly, the restriction exact sequence for~$k = 0$ gives
\begin{equation*}
0 \to \Hom(\cF_1, \cF_2(-H)) \to \Hom(\cF_1, \cF_2) \to \Hom(\cF_1\vert_S, \cF_2\vert_S) \to \Ext^1(\cF_1, \cF_2(-H)),
\end{equation*}
and since the right term vanishes and the next to it term is~$\kk$, we have~$\Hom(\cF_1, \cF_2) \ne 0$.
Finally, if the dimension of~$\Hom(\cF_1, \cF_2)$ is greater than~$1$, then~$\Hom(\cF_1, \cF_2(-H)) \ne 0$, and then by induction~$\Hom(\cF_1, \cF_2(-kH)) \ne 0$ for all~$k \ge 0$, again contradicting Serre vanishing.
\end{proof}

\begin{corollary}
\label{cor:eu}
Let~$X$ be a prime Fano threefold with factorial terminal singularities of genus~$g$
and assume that~\mbox{$g = r \cdot s \ge 6$} with~$r,s \ge 2$.
If the Mukai sheaf\/~$\cU_X$ of type~$(r,s)$ on~$X$ is locally free, it is exceptional.
\end{corollary}

\begin{proof}
First, assume~$s \ge r$.
By Lemma~\ref{lem:mbx-mbs}, if~$S \subset X$ is a very general anticanonical divisor, 
then~$\Pic(S) = \ZZ \cdot H_S$ and~$\cU_X\vert_S$ is a Mukai bundle on~$S$.
Moreover, $\cU_X\vert_S \cong \bL_S(\xi)$ is the Lazarsfeld bundle 
associated with a general curve~$C$ on~$S$ and a Mukai pair~$(\xi,\eta)$ of type~$(r,s)$ on~$C$ by Theorem~\ref{thm:mb-s}.
In particular, by Lemma~\ref{lem:lazarsfeld-bundle} this bundle is spherical and Proposition~\ref{prop:l-l-m-h} shows that
\begin{equation*}
\Hom(\bL_S(\xi), \bL_S(\xi)) = \kk
\qquad\text{and}\qquad 
\Ext^1(\bL_S(\xi), \bL_S(\xi) \otimes \cO_S(-H)) = 0.
\end{equation*}
As~$\cU_X$ is locally free by assumption, applying Lemma~\ref{lem:ext1-f1-f2} to~$\cF_1 = \cF_2 = \cU_X$, we obtain
\begin{equation*}
\Hom(\cU_X, \cU_X) = \kk
\qquad\text{and}\qquad 
\Ext^1(\cU_X, \cU_X \otimes \cO_X(-H)) = 0.
\end{equation*}
By Serre duality, the second equality implies that $\Ext^2(\cU_X, \cU_X) = 0$. 
On the other hand, we have~$\upchi(\cU_X, \cU_X) = \tfrac12\,\upchi(\bL_S(\xi),\bL_S(\xi)) = 1$ by Lemma~\ref{lem:hilb-mbx}.
Therefore,
\begin{equation*}
\Ext^1(\cU_X, \cU_X) = \Ext^3(\cU_X, \cU_X) = 0,
\end{equation*}
hence~$\cU_X$ is exceptional.

Now assume~$r > s$.
Then the Mukai sheaf~$\cU_X^\perp$ of type~$(s,r)$ defined in Lemma~\ref{lem:mbx-mbs} is also locally free,
hence it is exceptional by the above argument.
Furthermore, the defining sequence
\begin{equation*}
0 \to \cU_X^\perp \to \rH^0(X, \cU_X^\vee) \otimes \cO_X \to \cU_X^\vee \to 0
\end{equation*}
shows that~$\cU_X^\vee$ is the right mutation of~$\cU_X^\perp$ through~$\cO_X$,
hence~$\cU_X$ is also exceptional.
\end{proof}

\begin{proof}[Proof of Theorem~\textup{\ref{thm:Main}}]
The existence and uniqueness of a Mukai sheaf~$\cU_X$ of type~$(r,s)$ follows from Theorem~\ref{thm:mb-x}.
Since~$\cU_X$ is maximal Cohen--Macaulay and~$X$ is smooth, 
$\cU_X$ is locally free by the Auslander--Buchsbaum formula.
It remains to note that~$\cU_X$ is exceptional by Corollary~\ref{cor:eu}.
\end{proof}

\begin{remark}
\label{rem:g4}
If~$X$ is a smooth prime Fano threefold of genus~$g = 4$ 
then~$X = Q \cap R$ is a complete intersection of a quadric and a cubic hypersurface in~$\P^5$;
moreover, $Q$ is smooth or is a cone over a smooth 3-dimensional quadric~$\bar{Q}$,
and if~$Q$ is a cone then~$R$ does not contain its vertex.

If~$Q$ is smooth, \cite[Proposition~4.2]{KS23} shows that
the restriction to~$X$ of either of the two spinor bundles on~$Q$ is an exceptional Mukai bundle on $X$, 
but they are not isomorphic and thus a Mukai bundle on~$X$ exists, but it is not unique. 

If~$Q$ is singular, \cite[Proposition~4.2]{KS23} shows that
the pullback of the unique spinor bundle from~$\bar{Q}$ is the unique Mukai sheaf on $X$;
however, the sheaf is not exceptional. 

In both cases our arguments proving the uniqueness or exceptionality do not apply 
because Proposition~\ref{prop:l-l-m-h} fails for~$g = 4$.

Furthermore, consider a complete intersection $X = Q \cap R$ 
of a cone~$Q$ over a smooth quadric threefold~$\bar Q$ 
and a cubic hypersurface $R$ containing the vertex, where $R$ is very general among such cubics; 
such~$X$ is a factorial 1-nodal Fano threefold of genus~$g = 4$, 
see~\cite[Theorem~1.6(ii) and the proof of Theorem~1.8, Subcase~(4f-b)]{KP23}. 
Let $\cU_X$ be the reflexive extension of the pullback of the spinor sheaf on~$\bar Q$ to the complement of the node in $X$. 
Then $\cU_X$ is the unique Mukai sheaf on $X$, but it is neither locally free, nor exceptional.
\end{remark}

In~\cite{BKM:models} we prove that the Mukai sheaf 
on any prime Fano threefold~$X$ with factorial terminal singularites of genus~$g \ge 6$ is locally free,
hence it is exceptional.

%%%%%%%%%%%%%%%%%%%%%%%%%

\end{document}